\documentclass{acta_2022}
\usepackage{amssymb}
\usepackage{amsmath}
\usepackage{physics,bm,bbm}

\usepackage[longnamesfirst]{natbib} \usepackage{har2nat}
\setcitestyle{comma,aysep={}}
\shortcites{samplepreprint} %\shortcites{} %for 5+ authors
\usepackage{amssymb,amsmath}
\usepackage{newtxtext}
\usepackage{newtxmath}
\DeclareSymbolFont{CMlargesymbols}{OMX}{cmex}{m}{n}
\DeclareMathDelimiter{(}{\mathopen} {operators}{"28}{CMlargesymbols}{"00}
\DeclareMathDelimiter{)}{\mathclose}{operators}{"29}{CMlargesymbols}{"01}
\DeclareMathAlphabet\mathcal{OMS}{cmsy}{m}{n}
\SetMathAlphabet\mathcal{bold}{OMS}{cmsy}{b}{n}
\pagerange{\pageref{firstpage}--\pageref{lastpage}}
\allowdisplaybreaks
\doi{XXXXXXXX}  %GES% 8 figures, starting with 1
\usepackage[UKenglish]{babel} %GS%
\newcommand{\ignore}[1]{}
\usepackage[twoside,
	paperheight=247mm,
	paperwidth=174mm,
	marginratio={3:5,2:3},
	left=18mm,
	top=20mm]{geometry}
\usepackage[hang]{subfigure}
\numberwithin{figure}{section}
\numberwithin{table}{section}
\usepackage{subfigure}
\usepackage{graphics,graphicx}
\usepackage[cmyk]{xcolor}
\usepackage{enumitem} %or \usepackage{enumerate}

\newcommand{\bigO}{\mathcal{O}}

\newcommand{\poly}{\mathrm{poly}}

\newcommand{\width}{\mathrm{width}}
\newcommand{\depth}{\mathrm{depth}}

\renewcommand{\hat}{\widehat}

\newcommand{\Id}{\mathrm{Id}}

\DeclareMathOperator*{\argmin}{argmin}

\renewcommand{\bar}{\overline}

\newcommand{\dom}{\Omega}
\newcommand{\hu}{\hat{u}}
\newcommand{\ltwo}[1]{\norm{#1}_{L^2(\Omega)}}
\newcommand{\T}{\mathbb{T}}
\renewcommand{\R}{\mathbb{R}}

\newcommand{\N}{\mathbb{N}}

\renewcommand{\S}{\mathscr{S}}
\renewcommand{\L}{\mathcal{L}}

\renewcommand{\div}{{\mathrm{div}}}
\newcommand{\cA}{\mathcal{A}}
\newcommand{\cB}{\mathcal{B}}
\newcommand{\cC}{\mathcal{C}}
\newcommand{\cD}{\mathcal{D}}
\newcommand{\cE}{\mathcal{E}}
\newcommand{\cF}{\mathcal{F}}
\newcommand{\cG}{\mathcal{G}}
\newcommand{\cH}{\mathcal{H}}

\newcommand{\cJ}{\mathcal{J}}
\newcommand{\cK}{\mathcal{K}}
\newcommand{\cL}{\mathcal{L}}

\newcommand{\cN}{\mathcal{N}}

\newcommand{\cP}{\mathcal{P}}
\newcommand{\cQ}{\mathcal{Q}}
\newcommand{\cR}{\mathcal{R}}
\newcommand{\cS}{\mathcal{S}}

\newcommand{\cU}{\mathcal{U}}

\newcommand{\cX}{\mathcal{X}}

\newcommand{\bmalpha}{{\bm{\alpha}}}

\newcommand{\inspace}{\mathcal{X}}
\newcommand{\outspace}{\mathcal{Y}}
\newcommand{\uhat}{{\widehat{u}}}

\renewcommand{\hat}{\widehat}

\renewcommand{\tr}{{{\tau}}}
\newcommand{\trunk}{{\bm{\tau}}}

\newcommand{\branch}{{\bm{\beta}}}

\newcommand{\map}{\Psi}

\renewcommand{\S}{\mathcal{S}} %training set
\newcommand{\Et}{\mathcal{E}_T} %training error
\newcommand{\Eg}{\mathcal{E}} %generalization error
\newcommand{\E}[1]{{\mathbb{E}\left[ #1 \right]}} %expectation
\newcommand{\df}{\mathcal{L}}
\newcommand{\bu}{u}
\newcommand{\f}{f}

\newcommand{\bv}{v}

\newcommand{\rpde}{\mathcal{R}_{\mathrm{PDE}}}
\newcommand{\rdiv}{\mathcal{R}_{\mathrm{div}}}
\newcommand{\rs}{\mathcal{R}_{s}}
\newcommand{\rsu}{\mathcal{R}_{s,u}}
\newcommand{\rsp}{\mathcal{R}_{s,p}}
\newcommand{\rsgu}{\mathcal{R}_{s,\nabla u}}
\newcommand{\rt}{\mathcal{R}_{t}}

\newcommand{\hn}{\hat{n}}
\newcommand{\qu}[1]{\mathcal{Q}_{#1}}

\def\sA{{\mathbb{A}}}

\def\sC{{\mathbb{C}}}
\def\sD{{\mathbb{D}}}
\def\sP{{\mathbb{P}}}

\def\sR{{\mathbb{R}}}

\newcommand{\1}{\mathbbm{1}}
\newcommand{\sgn}[1]{\mathrm{sgn}\left( #1 \right)}

\newtheorem{theorem}{Theorem}[section]
\newtheorem{lemma}[theorem]{Lemma}
\newtheorem{proposition}[theorem]{Proposition}
\newtheorem{corollary}[theorem]{Corollary}
\newtheorem{example}[theorem]{Example}
\newtheorem{remark}[theorem]{Remark}
\newtheorem{assumption}[theorem]{Assumption}
\newtheorem{definition}[theorem]{Definition}

\newtheorem{question}{Question}

\title[Numerical analysis of PINNs]{Numerical analysis of physics-informed neural networks and related models in physics-informed machine learning}

\author[T. De Ryck and S. Mishra]{Tim De Ryck\\
Seminar for Applied Mathematics, ETH Zürich\\ Rämistrasse 101, 8092 Zürich, Switzerland\\
E-mail: {tim.deryck@math.ethz.ch}\\
\and
Siddhartha Mishra\\
Seminar for Applied Mathematics \& ETH AI Center, ETH Zürich\\ Rämistrasse 101, 8092 Zürich, Switzerland\\
E-mail: {siddhartha.mishra@math.ethz.ch}}

\AtBeginDocument{\mathcode`v=\varv} %redefining $v$ because it looks too much like $\nu$

\begin{document}

\label{firstpage}
\maketitle

\vspace{2.5cm}

\begin{abstract}
Physics-informed neural networks (PINNs) and their variants have been very popular in recent years as algorithms for the numerical simulation of both forward and inverse problems for partial differential equations. This article aims to provide a comprehensive review of currently available results on the numerical analysis of PINNs and related models that constitute the backbone of physics-informed machine learning. We provide a unified framework in which analysis of the various components of the error incurred by PINNs in approximating PDEs can be effectively carried out. A detailed review of available results on approximation, generalization and training errors and their behavior with respect to the type of the PDE and the dimension of the underlying domain is presented. In particular, the role of the regularity of the solutions and their stability to perturbations in the error analysis is elucidated. Numerical results are also presented to illustrate the theory. We identify training errors as a key bottleneck which can adversely affect the overall performance of various models in physics-informed machine learning.  
\end{abstract}

\newpage

\tableofcontents 

\section{Introduction}
Machine learning, particularly deep learning \cite{DLbook}, has permeated into every corner of modern science, technology and society, whether it is computer vision, natural language understanding, robotics and automation, image and text generation or protein folding and prediction. 

The application of deep learning to computational science and engineering, in particular, to the numerical solution of (partial) differential equations has gained enormous momentum in recent years. One notable avenue highlighting this application falls within the framework of \emph{supervised learning}, i.e., approximating the solutions of PDEs by deep learning models from (large amounts of) data about the underlying PDE solutions. Examples include high-dimensional parabolic PDEs \cite{E1} and references therein, parametric elliptic \cite{SchwabZech2019,kuty} or hyperbolic PDEs \cite{de2021analysis,lye2020deep} and learning the solution operator directly from data \cite{chen1995universal,lu2021machine,FNO,cno}
and references therein. However, generating and accessing large amounts of data on solutions of PDEs requires either numerical methods or experimental data, both of which can be (prohibitively) expensive. Consequently, there is a need for so-called \emph{unsupervised} or \emph{semi-supervised} learning methods which require only small amounts of data. 

Given this context, it would be useful to solve PDEs using machine learning models and methods, directly from the underlying physics (governing equations) and without having to access data about the underlying solutions. Such methods can be loosely termed as constituting \emph{physics-informed machine learning}. 

The most prominent contemporary models in physics-informed machine learning are \emph{physics-informed neural networks} or PINNs. The idea behind PINNs is very simple: as neural networks are universal approximators of large variety of function classes (even measurable functions), one considers the strong form of the PDE residual within the \emph{ansatz space} of neural networks and minimizes this residual via (stochastic) gradient descent to obtain a neural network that approximates the solution of the underlying PDE. This framework was already considered in the 1990s in \citet{DPT,Lag1,Lag2} and references therein and these papers should be considered as the progenitors of PINNs.  However, the modern version of PINNs was introduced and they were named as such more recently in \citet{KAR1,KAR2}. Since then, there have been an explosive growth in the literature on PINNs and they have been applied in numerous settings, both in the solution of the forward as well as inverse problems for PDEs and related equations (SDEs, SPDEs), see \citet{KarRev,cuomo} for extensive reviews of the available literature on PINNs.

However, PINNs are not the only models within the framework of physics-informed machine learning. One can consider other forms of the PDE residual such as the variational form, resulting in VPINNs \cite{kharazmi2019variational}, the weak form resulting in wPINNs \cite{deryck2022wpinns} or minimizing the underlying energy resulting in DeepRitz method \cite{DR1}. Similarly, one can use Gaussian processes \cite{GP}, DeepONets \cite{lu2021machine} or neural operators \cite{NO} as ansatz spaces to realize alternative models for phyiscs-informed machine learning. 

Given the exponentially growing literature on PINNs and related models in physics-informed machine learning, it is essential to ask whether these methods possess rigorous mathematical guarantees on their performance and whether one can analyze these methods in a manner that is analogous to the huge literature on the analysis of traditional numerical methods such as finite differences, finite elements, finite volumes and spectral methods. Although the overwhelming focus of research has been on the wide-spread applications of PINNs and its variants in different domains in science and engineering, a significant amount of papers rigorously analyzing PINNs has emerged in recent years. The aim of this article is to review the available literature on the \emph{numerical analysis of PINNs and related models that constitute phyiscs-informed machine learning.} Our goal is to critically analyze PINNs and its variants with a view to ascertain when they can be applied and what are the limits to their applicability. To this end, we start by presenting the formulation for physics-informed machine-learning in terms of the underlying PDEs, their different forms of residuals and the approximating ansatz spaces in Section \ref{sec:2}. In Section \ref{sec:analysis}, we outline the main components of the underlying errors with physics-informed machine learning and the resulting approximation, stability, generalization and training errors are analyzed in Sections \ref{sec:approximation-error}, \ref{sec:stability}, \ref{sec:generalization}  and \ref{sec:training}, respectively.

\section{Physics-informed machine learning}
\label{sec:2}

In this section, we set the stage for the rest of the article. In Section \ref{sec:setting}, we introduce an abstract PDE setting and consider a number of important examples of PDEs that will be used in the numerical analysis later on. Next, different model classes are introduced (Section \ref{sec:models}). Variations of physics-informed learning are then introduced based on different formulations of PDEs in Section \ref{sec:residuals}, of which the resulting residuals will constitute the physics-informed loss function by being discretized (Section \ref{sec:quad}) and optimized (Section \ref{sec:intro-optimization}). Finally, a summary of the whole method is given in Section \ref{sec:summary}. 

\subsection{PDE setting}\label{sec:setting}

Throughout this article, we consider the following setting. 
Let $X,Y,W$ be separable Banach spaces with norms $\| \cdot \|_{X}$, $\|\cdot\|_{Y}$ resp. $\|\cdot\|_{W}$, and analogously let $X^{\ast} \subset X$, $Y^{\ast} \subset Y$ and $W^*\subset W$ be closed subspaces with norms $\|\cdot \|_{X^{\ast}}$, $\|\cdot\|_{Y^{\ast}}$ and $\|\cdot\|_{W^{\ast}}$, respectively. Typical examples of such spaces include $L^p$ and Sobolev spaces. We consider PDEs of the following abstract form,
\begin{equation}
    \label{eq:pde}
    \cL[u] = f, \qquad \cB[u] = g,
\end{equation}
where $\cL: X^{\ast} \to Y^{\ast}$ is a \emph{differential operator} and $f \in Y^{\ast}$ is an \emph{input} or \emph{source} function. The boundary conditions (including the initial condition for time-dependent PDEs) are prescribed by the \emph{boundary operator} $\cB: X^{\ast} \to W^{\ast}$ and the boundary data $g\in W^*$. 
We assume that for all $f \in Y^{\ast}$ there exists a unique $u \in X^{\ast}$ such that \eqref{eq:pde} holds. Finally, we assume that for all $u\in X^*$ and $f\in Y^*$ it holds that
\begin{equation}
\label{eq:assm1}
\|\cL[u]\|_{Y^{\ast}} < +\infty, \qquad \|f\|_{Y^{\ast}} < +\infty, \qquad \|\cB[u]\|_{W^{\ast}} < +\infty
\end{equation}
We will also denote the domain of $u$ as $\Omega$, where either $\Omega = D \subset \R^d$ for time-independent PDEs or $\Omega = D \times [0,T] \subset \R^{d+1}$ for time-dependent PDEs. 

%In what follows, we present the framework of physics-informed neural networks, which allows to obtain a surrogate model for the solution of \eqref{eq:pde} with little or even no training data. 

\subsubsection{Forward problems}\label{sec:forward}

The \emph{forward problem} for the abstract PDE \eqref{eq:pde} amounts to the following: given a source $f$ and boundary conditions $g$ (and potentially another parameter function $a$ such that $\cL := \cL_a$), find the solution $u$ of the PDE. This can be summarized as finding an \emph{operator} that maps an input function $v\subset \{f, g, a\} \subset \cX$ to the corresponding solution $u\in X^*$ of the PDE, 
\begin{eqnarray}
    \cG : \cX \to X^*: v \mapsto \cG[v] = u. 
\end{eqnarray}
Note that it can also be interesting to determine the whole operator $\cG$ rather than just the function $u$. Finding a good approximation of $\cG$ is referred to as \emph{operator learning}. A first concrete example of an operator would be the mapping between the input or source function $f$ and the corresponding solution $u$, and hence $\cX = Y^*$ and $v=f$, 
\begin{eqnarray}
    \cG: Y^* \to X^*: f \mapsto \cG[f] = u. 
\end{eqnarray}
Another common example for a time-dependent PDE, where $\cB[u] = g$ entails the prescription of the initial condition $u(\cdot, 0) = u_0$, is the \emph{solution operator} that maps the initial condition $u_0$ to the solution of the PDE, 
\begin{eqnarray}
    \cG : W^* \to X^* : u_0 \mapsto \cG[u_0] = u, 
\end{eqnarray}
where $\cX = W^*$ and $v=u_0$. Alternatively one can also consider the mapping $u_0 \mapsto u(\cdot, T)$. 

A final example of operator learning can be found in \emph{parametric PDEs}, where the differential operator, source function, initial condition or boundary condition depends on scalar parameters or a parameter function. Although solving parametric PDEs is usually not called operator learning, it is evident that it can be identified as such, with the only difference being that the input space is not necessarily infinite-dimensional. 

In what follows we give examples of PDEs \eqref{eq:pde} for which we will consider the forward problem later onwards. 

\begin{example}[Semilinear heat equation]\label{def:semilinear-heat}
We first consider the set-up of the semilinear heat equation. Let $D \subset \R^d$ be an open connected bounded set with a continuously differentiable boundary $\partial D$. The semi-linear parabolic equation is then given by,
\begin{equation}
    \label{eq:heat}
    \begin{cases}
     u_t = \Delta u + f(u),  &\forall x\in D,~ t \in (0,T), \\
    u(x,0) = u_0(x),  &\forall x \in D, \\
    u(x,t)= 0,  &\forall x\in \partial D, ~ t \in (0,T). 
    \end{cases}
\end{equation}
Here, $u_0 \in C^k(D)$, $k\geq 2$, is the initial data, $u \in C^k([0,T] \times D)$ is the classical solution and $f:\R \times \R$ is the non-linear source (reaction) term. We assume that the non-linearity is globally Lipschitz i.e, there exists a constant $C_f>0$ such that 
\begin{equation}
    \label{eq:assf}
    |f(v) - f(w)| \leq C_f|v-w|, \quad \forall v,w \in \R.
\end{equation}
In particular, the homogeneous linear heat equation with $f(u) \equiv 0$ and the linear source term $f(u) = \alpha u$ are examples of \eqref{eq:heat}. Semilinear heat equations with globally Lipschitz nonlinearities arise in several models in biology and finance \cite{Jent1}. 
The existence, uniqueness and regularity of the semi-linear parabolic equations with Lipschitz non-linearities such as \eqref{eq:heat} can be found in classical textbooks such as \cite{Frdbook}. 
\end{example}

\begin{example}[Navier-Stokes equations]\label{def:NS-equation}
We consider the well-known incompressible Navier-Stokes equations \citep{temam2001navier} and references therein,
\begin{equation}\label{eq:navier-stokes}
   \begin{cases} u_t + u\cdot \nabla u + \nabla p = \nu
  \Delta u  &\text{in } D \times [0,T] ,\\
  \div(u) = 0 &\text{in } D \times [0,T],\\
  u(t=0) = u_0 &\text{in }  D.
  \end{cases}
\end{equation}
Here, $u:D\times [0,T] \to \mathbb{R}^d$ is the fluid velocity, $p:D\times [0,T]\to\mathbb{R}$ is the pressure and $u_0:D \to \mathbb{R}^d$ is the initial fluid velocity. The viscosity is denoted by $\nu\geq 0$. 
For the rest of the paper, we consider the Navier-Stokes equations \eqref{eq:navier-stokes} on the $d$-dimensional torus $D = \mathbb{T}^d = [0,1)^d$ with periodic boundary conditions. 

\end{example}

\begin{example}[Viscous and inviscid scalar conservation laws]\label{def:scl}
We consider the following one-dimensional version of \emph{viscous scalar conservation laws} as a model problem for quasilinear, convection-dominated diffusion equations,
\begin{equation}
\label{eq:vscl}
\begin{aligned}
u_t + f(u)_x &= \nu u_{xx}, \quad \forall x\in (0,1),~t \in [0,T], \\
u(x,0) &= u_0(x), \quad \forall x \in (0,1). \\
u(0,t) &= u(1,t) \equiv 0, \quad \forall t \in [0,T].
\end{aligned}
\end{equation}
Here, $u_0 \in C^k([0,1])$, for some $k \geq 1$, is the initial data and we consider zero Dirichlet boundary conditions. Note that $0 < \nu \ll 1$ is the viscosity coefficient. The flux function is denoted by $f\in C^k(\R;\R)$. One can follow standard textbooks such as \cite{GRbook} to conclude that as long as $\nu > 0$, there exists a classical solution $u \in C^k([0,T)\times [0,1])$ of the viscous scalar conservation law \eqref{eq:vscl}. 

It is often interesting to set $\nu=0$ and consider the \emph{inviscid scalar conservation law},
\begin{eqnarray}\label{eq:scl}
    u_t + f(u)_x = 0. 
\end{eqnarray}
In this case, the solution might not be continuous anymore as it can develop shocks, even for smooth initial data \cite{holden2015front}. 
\end{example}

\begin{example}[Poisson's equation]\label{def:poisson}
Finally, we consider a prototypical elliptic PDE. Let $\Omega\subset \R^d$ be a bounded open set that satisfies the exterior sphere condition, let $f\in C^1(\Omega)\cap L^\infty(\Omega)$ and let $g\in C(\partial\Omega)$. Then there exists $u\in C^2(\Omega)\cap C(\overline{\Omega})$ that satisifies Poisson's equation,
\begin{equation}\label{eq:poisson}
        \begin{cases}
            -\Delta u=f &\text{in }\Omega, \\
            u=g &\text{on }\partial\Omega.
        \end{cases}
\end{equation}
Moreover, if $f$ is smooth then $u$ is smooth in $\Omega$. 
\end{example}

\subsubsection{Inverse problems}\label{sec:inverse}

Secondly, we consider settings where one does not have complete information on the inputs to the PDE \eqref{eq:pde} (e.g. initial data, boundary conditions and parameter coefficients). As a result, the forward problem can not be solved uniquely. However, when one has access to (noisy) data for (observables of) the underlying solution, one can attempt to determine the unknown inputs of the PDEs and consequently its solution. Problems with the aim to recover the PDE input based on data are referred to as \emph{inverse problems}. 

We will mainly consider one kind of inverse problem, namely the \emph{unique continuation problem} (for time-dependent problems also known as \emph{data assimilation problem}). Recall the abstract PDE $\cL[u] = \f$ in $\dom$ \eqref{eq:pde}. Moreover, we adopt the assumptions \eqref{eq:assm1} of the introduction of this chapter. The difference with the forward problem is that we now have incomplete information on the boundary conditions. As a result, a unique solution to the PDE does not exist. Instead, we assume that we have access to (possibly noisy) measurements of a certain observable
\begin{equation}\label{eq:dt}
    \map(u)= g \qquad \text{in } \dom',
\end{equation}
where $\dom'\subset \dom$ is the observation domain, $g$ is the measured data and $\map:X^*\to Z^*$, where $Z^*$ is some Banach space. We make the assumption that 
\begin{equation}
\label{eq:assm2-inv}
\begin{aligned}
&(H3): \quad \|\map(\bu)\|_{Z^{\ast}} < +\infty, \quad \forall~ \bu \in X^{\ast}, ~{\rm with}~\|\bu\|_{X^{\ast}} < +\infty. \\
&(H4):\quad \|g\|_{Z^{\ast}} < +\infty. 
\end{aligned}
\end{equation}

If $\dom'$ includes the boundary of $\dom$, meaning that one has measurements on the boundary, then it is again feasible to just use the standard framework for forward problems. However, this is often not possible. In heat transfer problems, for example, the maximum temperature will be reached at the boundary. As a result, it might be too hot near the boundary to place a sensor there and one will only have measurements that are far enough away from the boundary. Fortunately, one can adapt the physics-informed framework to retrieve an approximation to $u$ based on $f$ and $g$. 

In analogy with operator learning for forward problems, one can also attempt to entirely learn the \emph{inverse operator}, 
\begin{equation}
    \cG^{-1} : Z^* \to X^* : \Psi[u] = g \mapsto \cG^{-1}[g] = v, 
\end{equation}
where $v$ could for instance be the boundary conditions, the parameter function $a$ or the input/source function $f$. 

\begin{example}[Heat equation]\label{ex:heat-inv}
We revisit the heat equation $u_t - \Delta u = f$ with zero Dirichlet boundary conditions and $f\in L^2(\Omega)$, where $\Omega = D\times (0,T)$. We will assume that $u \in H^1(((0,T);H^{-1}(D)) \cap L^2((0,T);H^1(D))$ will solve the heat equation in a weak sense. The heat equation would have been well-posed if the initial conditions i.e $u_0 = u(x,0)$ had been specified. Recall that the aim of the {data assimilation} problem is to infer the initial conditions and the entire solution field at later times from some measurements of the solution in time. To model this, we consider the following \emph{observables}:
\begin{equation}
    \label{eq:dtht}
    \Psi[u] := u = g, \quad \forall (x,t) \in D^{\prime} \times (0,T) =: \Omega',
\end{equation}
for some open, simply connected observation domain $D^{\prime} \subset D$ and for $g \in L^2(\Omega')$.
Thus solving the data assimilation problem for the heat equation amounts to finding the solution $u$ of the heat equation in the whole space-time domain $\Omega = D\times (0,T)$, given data on the observation sub-domain $\Omega' = D'\times (0,T)$. 
\end{example}

\begin{example}[Poisson equation]\label{ex:poisson-inv}
We revisit the Poisson equation (Example \ref{def:poisson}), but now with data on a subdomain of $\dom$, 
\begin{eqnarray}\label{eq:ps}
    -\Delta u = f, \quad \text{in } \dom \subset \mathbb{R}^d, \qquad u\vert_{\dom'} = g\quad \text{in } \dom' \subset \dom. 
\end{eqnarray}
In this case the observable $\Psi$ introduced in \eqref{eq:dt} is the identity, $\Psi[u] = u$. 
\end{example}

\begin{example}[Stokes equation]\label{ex:stokes-inv}
Finally, we consider a simplified version of the Navier-Stokes equations, namely the stationary Stokes equation. 
Let $D \subset \R^d$ be an open, bounded, simply connected set with smooth boundary. We consider the Stokes' equations as a model of stationary, highly viscous fluid:
\begin{equation}
    \label{eq:st}
    \begin{aligned}
    \Delta \bu + \nabla p &=\f, \quad \forall x \in D, \\
    \div(\bu) &= f_d, \quad \forall x \in D.
    \end{aligned}
\end{equation}
Here, $\bu: D \to \R^d$ is the velocity field, $p: D \to \R$ is the pressure and $\f: D\to \R^d$, $f_d: D \to \R$ are source terms. 

Note that the Stokes equation \eqref{eq:st} is not well-posed as we are not providing any boundary conditions. In the corresponding data assimilation problem \cite{BH1} and references therein, one provides the following \emph{data},
\begin{equation}
    \label{eq:dtst}
   \Psi[u] := \bu = g, \quad \forall x \in D^{\prime},
\end{equation}
for some open, simply connected set $D^{\prime} \subset D$.
Thus, the data assimilation inverse problem for the Stokes equation amounts to inferring the velocity field $\bu$ (and the pressure $p$), given $\f,f_d$ and $g$. 
\end{example}

\subsection{Ansatz spaces}\label{sec:models}

We give an overview of different classes of parametric functions $\{u_\theta\}_{\theta\in\Theta}$ that are often used to approximate the true solution $u$ of the PDE \eqref{eq:pde}. We let $n$ be the number of (real) parameters, such that the parameter space $\Theta$ is a subset of $\R^n$. 

\subsubsection{Linear models}\label{sec:linear-models}

We first consider models that linearly depend on the parameter $\theta$, i.e. models that are a linear combination of a fixed set of functions $\{\phi_i: \Omega \to \R\}_{1\leq i\leq n}$. For any $\theta\in\R^n$ the parameterized linear model $u_\theta$ is then defined as,
\begin{equation}
    u_\theta(x) = \sum_{i=1}^n \theta_i \phi_i(x).
\end{equation}
This general model class constitutes the basis of many existing numerical methods for approximation solutions to PDEs, of which we give an overview below. 

First, \emph{spectral methods} use smooth functions that are globally defined i.e., supported on almost all of $\Omega$, and often form an orthogonal set, see e.g. \citet{hesthaven2007spectral}. The particular choice of functions generally depends on the geometry and boundary conditions of the considered problem: whereas trigonometric polynomials (Fourier basis) are a default choice on periodic domains, Chebyshev and Legendre polynomials might be more suitable choices on non-periodic domains. The optimal parameter vector $\theta^*$ is then determined by solving a linear system of equations. Assuming that $\cL$ is a linear operator, this system is given by
\begin{eqnarray}\label{eq:spectral-linear-system}
    \sum_{i=1}^n \theta_i \int_\Omega \psi_k \cL(\phi_i) = \int_\Omega \psi_k f \qquad \text{for } 1\leq k \leq K, 
\end{eqnarray}
where either $\psi_k = \phi_k$, resulting in the \emph{Galerkin method}, or $\psi_k = \delta_{x_k}$ with $x_k\in\Omega$, resulting in the \emph{collocation method}, or the more general case where the $\psi_k$ and $\phi_k$ are distinct (and also not Dirac deltas), which is referred to as the \emph{Petrov-Galerkin method}. In the collocation method, the choice of the so-called collocation points $\{x_k\}_k$ depends on the functions $\phi_i$. Boundary conditions are implemented through the choice of the functions $\phi_i$ or by adding more equations to the above linear system. 

In large contrast to the global basis functions of the spectral method stands the \emph{finite element method (FEM)}, where one considers locally defined polynomials. More precisely, one first divides $\Omega$ into (generally overlapping) subdomains $\Omega_i$ (subintervals in 1D) of diameter at most $h$ and then considers on each $\Omega_i$ a piecewise-defined polynomial $\phi_i$ of degree at most $p$ that is zero outside out $\Omega_i$. Much like with the spectral Galerkin method one can then define a linear system of equations as \eqref{eq:spectral-linear-system} that needs to be solved to find the optimal parameter vector $\theta^*$. As $\phi_i$ is now non-smooth on $\Omega$ one needs to rewrite $\int_\Omega \phi_k \cL(\phi_i)$ using integration by parts. Another key difference with the spectral method is that, due to the local support of the functions $\phi_i$, the linear system one has to solve will be sparse and have structural properties related to the choice of grid and polynomials. The accuracy of the approximation can be improved by reducing $h$ ($h$-FEM), increasing $p$ ($p$-FEM), or both ($hp$-FEM). Alternatively, the finite element method can also be used in combination with a squared loss, resulting in the \emph{Least-Squares Finite Element Method} \cite{bochev2009least}. The finite element method is ubiquitous in scientific computing and many variants can be found. 

Another choice for the functions $\phi_i$ are \emph{radial basis functions (RBFs)}, see e.g. \citet{buhmann2000radial}. The key ingredient of RBFs is a radial function $\varphi:[0,\infty)\to \R$, for which it must hold  for any choice of nodes $\{x_i\}_{1\leq i\leq n}\subset \Omega$ that both (1) the functions $x\mapsto \phi_i(x) := \varphi(\norm{x-x_i})$ are linearly independent and (2) the matrix $A$ with entries $A_{ij} = \varphi(\norm{x_i-x_j})$ is non-singular. Radial functions can be both globally supported (e.g. Gaussians, polyharmonic splines) as well as compactly supported (e.g. bump functions). The optimal parameter $\theta^*$ can then be determined in a Galerkin, collocation or least-squares fashion. RBFs are particularly advantageous as they are mesh-free (as opposed to FEM), which is useful for complex geometries, and straightforward to use in high-dimensional settings, as one can use randomly generated nodes, rather than nodes on a grid. 

Finally, RBFs with randomly generated nodes can be seen as an example of a more general \emph{random feature model (RFM)}, where the functions $\phi_i$ are randomly drawn from a chosen distribution over a function space, e.g. \citet{rahimi2008uniform}. In practice, one often chooses the function space to be parametric itself, such that generating $\phi_i$ reduces to randomly drawing finite vectors $\alpha_i$ and setting $\phi_i(x) := \phi(x; \alpha_i)$. Random feature models where $\phi_i$ is a neural network with randomly initialized weights and biases are closely connected to neural networks of which only the outer layer is trained. When the network is large then the latter model is often denoted as an \emph{extreme learning machine (ELM)} \citet{huang2006extreme}.

% spectral methods: https://www.uni-muenster.de/imperia/md/content/physik_tp/lectures/ss2016/num_methods_ii/intro.pdf

\subsubsection{Nonlinear models}\label{sec:nonlinear-models}

\paragraph{Neural networks.} A first, but very central, example of parametric models that depend nonlinearly on their parameters are \emph{neural networks}. The simplest form of a neural network is a \emph{multilayer perceptron (MLP)}. MLPs get their name from their definition as a concatenation of affine maps and activation functions. As a consequence of this structure, the computational graph of an MLP is organized in layers. At the $k$-th layer, the activation function is component-wise applied to obtain the value
\begin{equation}\label{eq:mlp-zk}
    z^k = \sigma(C_{k-1}(z^{k-1})) = \sigma(W_{k-1}z^{k-1}+b_{k-1}), 
\end{equation}
where $z^0$ is the input of the network, $W_k\in \mathbb{R}^{d_{k+1}\times d_{k}}$, $b_k\in\mathbb{R}^{d_{k+1}}$ and $C_k:\mathbb{R}^{d_{k}}\to\mathbb{R}^{d_{k+1}}:x\mapsto W_k x+b_k$ for $d_k\in\mathbb{N}$ is an affine map. 
A {multilayer perceptron (MLP)} with $L$ layers, activation function $\sigma$, output function $\sigma_o$, input dimension $d_{in}$ and output dimension $d_{out}$ then is a function $f_\theta:\mathbb{R}^{d_{in}}\to\mathbb{R}^{d_{out}}$ of the form
\begin{equation}
    f_\theta(x) = (\sigma_o \circ C_L \circ \sigma \circ \cdots \circ \sigma\circ  C_0)(x) \quad \text{for all } x\in\mathbb{R}^{d_{in}}, 
\end{equation}
where $\sigma$ is applied element-wise and $\theta = ((W_0,b_0), \ldots, (W_L,b_L))$.
The dimension $d_k$ of each layer is also called the \textit{width} of that layer. The width of the MLP is defined as $\max_k d_k$. Furthermore, the matrices $W_k$ are called the \textit{weights} of the MLP and the vectors $b_k$ are called the \textit{biases} of the MLP. Together, they constitute the (trainable) parameters of the MLP, denoted by $\theta = ((W_0,b_0), \ldots, (W_L,b_L))$. Finally, an output function $\sigma_o$ can be employed to increase the interpretability of the output, but is often equal to the identity in the setting of function approximation. 

\begin{remark}
In some texts, the number of layers $L$ corresponds to the number of affine maps in the definition of $f_\theta$. Of these $L$ layers, the first $L-1$ layers are called hidden layers. An $L$-layer network in this alternative definition hence corresponds to a $(L-1)$-layer network in our definition. 
\end{remark}

A lot of the properties of a neural network depend on the choice of activation function. We discuss the most common activation functions below.

\begin{itemize}
    \item The \emph{Heaviside} function is the activation that was used in the McCullocks-Pitts neuron and is defined as
    \begin{equation}
        H(x) = \begin{cases}0, &x<0 \\ 1 &x\geq 0, \end{cases}
    \end{equation}
    Because the gradient is zero wherever it exists, it is not possible to train the network using a gradient-based approach. For this reason, the Heaviside function is not used anymore. 
    \item The \emph{logistic} function is defined as 
    \begin{equation}
        \rho(x) = \frac{1}{1+e^{-x}} 
    \end{equation}
    and can be thought of as a smooth approximation of the Heaviside function as $\rho(\lambda x)\to H(x)$ for $\lambda\to\infty$ for almost every $x$. It is a so-called sigmoidal function, meaning that the function is smooth, monotonic and has horizontal asymptotes for $x\to\pm\infty$. Many early-days approximation-theoretical results were stated for neural networks with a sigmoidal activation function. %We will see in Chapter \ref{ch:nnt} that for neural networks with a sigmoidal activation function many theoretical results are available. 
    \item The hyperbolic tangent or \emph{tanh} function, defined by
    \begin{equation}
        \sigma(x) = \frac{e^x-e^{-x}}{e^x+e^{-x}},
    \end{equation}
    is another smooth, sigmoidal activation function that is symmetric, unlike the logistic function. One problem with the tanh (and all other sigmoidal functions) is that the gradient vanishes away from zero, which might be problematic when using a gradient-based optimization approach. 
    \item The Rectified Linear Unit, rectifier function or \emph{ReLU} function, defined as
    \begin{equation}
        \text{ReLU}(x) = \max\{x,0\},
    \end{equation}
    is a very popular activation function. It is easy to compute, scale-invariant and reduces the vanishing gradient problem that sigmoidal functions have. One common issue with ReLU networks is that of `dying neurons', caused by the fact that the gradient of ReLU is zero for $x<0$. This issue can however be overcome by using the so-called \emph{leaky ReLU} function, 
    \begin{equation}
        \text{ReLU}_\nu(x) = \max\{x,-\nu x\},
    \end{equation}
    where $\nu>0$ is small. Other (smooth) adaptations are the Gaussian Error Linear Unit, the Sigmoid Linear Unit, Exponential Linear Unit and softplus function. 
\end{itemize}

\paragraph{Neural operators.} In order to approximate operators, one needs to allow the input and output of the learning architecture to be infinite-dimensional. A possible approach is to use \emph{deep operator networks} (DeepONets), as proposed in \cite{ChenChen1995, deeponets}. 
Given $m$ fixed sensor locations $\{x_j\}_{j=1}^m \subset \Omega$ and the corresponding \emph{sensor values} $\{v(x_j)\}_{j=1}^m$ as input,
a DeepONet can be formulated in terms of two (deep) neural networks: a \emph{branch net} $\branch: \R^m \to \R^{p}$ and a \emph{trunk net} $\trunk :\Omega\to\R^{p+1}$.
The branch and trunk nets are then combined to approximate the underlying nonlinear operator as the following \emph{DeepONet}, 
\begin{eqnarray}
    \cG_\theta:\cX \to L^2(\Omega)\quad \text{with}\quad \cG_\theta[v](y) = \tr_0(y)+\sum_{k=1}^p \beta_k(v) \tr_k(y). 
\end{eqnarray}
% A second approach is that of \emph{neural operators}, which generalize hidden layers by including a non-local integral operator \cite{GKO}, of which particularly \emph{Fourier neural operators} (FNOs) \cite{FNO} are already well-established. The practical implementation (i.e. discretization) of an FNO maps from and to the space of trigonometric polynomials of degree at most $N\in\N$, denoted by $L^2_N$, and can be identified with a finite-dimensional mapping that is a composition of affine maps and nonlinear layers of the form $\mathfrak{L}_l(z)_j = \sigma(W_l v_j + b_{l,j}
% \cF^{-1}_N(P_l(k)\cdot \cF_N(z)(k)_j))$, where the $P_l(k)$ are coefficients that define a non-local convolution operator via the discrete Fourier transform $\cF_N$, see \cite{kovachki2021universal}.

One of the alternatives to DeepONets that also allow to learn maps between infinite-dimensional spaces are \emph{neural operators} \cite{li2020neural}. Neural operators are inspired by the fact that under some general conditions the solution $u$ to the PDE \eqref{eq:pde} can be represented in the form
\begin{eqnarray}
    u(x) = \int_D G(x,y)f(y)dy, 
\end{eqnarray}
where $G:\Omega\times \Omega \to\R$ is the Green's function. This kernel representation led the authors of \cite{li2020neural} to propose to replace the formula for a hidden layer of a standard neural network by the following operator, 
\begin{eqnarray}\label{eq:no-integral}
    \cL_\ell(v)(x) = \sigma\left(W_\ell v(x) + \int_D \kappa_\phi(x,y,a(x),a(y))v(y) d\mu_x(dy)\right), 
\end{eqnarray}
where $W_\ell$ and $\phi$ are to be learned from the data, $\kappa_\phi$ is a kernel and $\mu_x$ is a Borel measure for every $x\in D$. In \cite{li2020neural}, the kernel $\kappa_\phi$ is modeled as a neural network. 
More general, neural operators are mappings of the form \cite{li2020neural, kovachki2021universal},
\begin{equation}\label{eq:no}
    \cN: \inspace\to\outspace: u\mapsto \cQ \circ \cL_L \circ \cL_{L-1} \circ \dots \circ \cL_{1} \circ \cR(u),
\end{equation}
for a given depth $L\in \N$, where $\cR: \inspace(D;\R^{d_\inspace}) \to \cU(D;\R^{d_v})$, $d_v \ge d_u$, is a \emph{lifting} operator (acting locally), of the form
\begin{align} \label{eq:no-r}
\cR(a)(x) = R a(x), 
\quad 
R \in \R^{d_v\times d_\inspace},
\end{align}
and $\cQ: \cU(D;\R^{d_v}) \to \outspace(D;\R^{d_\outspace})$ is a local \emph{projection} operator, of the form 
\begin{align} \label{eq:no-q}
\cQ(v)(x) = Qv(x), 
\quad
Q \in \R^{d_\outspace \times d_v}.
\end{align}
In analogy with canonical finite-dimensional neural networks, the layers $\cL_1, \dots, \cL_L$ are non-linear operator layers, $\L_\ell: \cU(D;\R^{d_v}) \to \cU(D;\R^{d_v})$, $v\mapsto \L_\ell(v)$, which we assume to be of the form 
\begin{equation}\label{eq:no-layer}
    \cL_\ell(v)(x)
=
\sigma\bigg(
W_\ell v(x)
+
b_\ell(x)
 + \big(\cK(a;\theta_\ell) v\big)(x)
\bigg),
\quad
\forall \, x\in D.
\end{equation}
Here, the weight matrix $W_\ell \in \R^{d_v\times d_v}$ and bias $b_\ell(x)\in \cU(D;\R^{d_v})$ define an affine pointwise mapping $W_\ell v(x) + b_\ell(x)$. We conclude the definition of a neural operator by stating that $\sigma: \R \to \R$ is a non-linear activation function that is applied component-wise and $\cK$ is a \emph{non-local} linear operator,
\begin{equation}
    \cK: \inspace\times \Theta \to L\left(\cU(D;\R^{d_v}), \cU(D;\R^{d_v})\right),
\end{equation}
that maps the input field $a$ and a parameter $\theta \in \Theta$ in the parameter set $\Theta$ to a bounded linear operator $\cK(a,\theta): \cU(D;\R^{d_v}) \to \cU(D;\R^{d_v})$. When one defines the linear operators $\cK(a,\theta)$ through an integral kernel, then \eqref{eq:no-layer} reduces again to \eqref{eq:no-integral}. \emph{Fourier Neural Operators (FNOs)} \cite{FNO} are special cases of such general neural operators in which this integral kernel corresponds to a convolution kernel, which on the periodic domain $\T^d$ leads to non-linear layers $\cL_\ell$ of the form ,
\begin{align} \label{eq:fno-layer}
\cL_\ell(v)(x)
=
\sigma
\bigg(
W_\ell v(x) + b_\ell(x) + \cF^{-1} \Big(P_{\ell}(k) \cdot \cF(v)(k)\Big)(x)
\bigg).
\end{align}
In practice, however, this definition can not be used as evaluating the Fourier transform requires the exact computation of an integral. The practical implementation (i.e. discretization) of an FNO maps from and to the space of trigonometric polynomials of degree at most $N\in\N$, denoted by $L^2_N$, and can be identified with a finite-dimensional mapping that is a composition of affine maps and nonlinear layers of the form, 
\begin{eqnarray}
    \mathfrak{L}_\ell(z)_j = \sigma(W_\ell v_j + b_{\ell,j} \cF^{-1}_N(P_\ell(k)\cdot \cF_N(z)(k)_j))
\end{eqnarray}
where the $P_\ell(k)$ are coefficients that define a non-local convolution operator via the discrete Fourier transform $\cF_N$, see \cite{kovachki2021universal}.

It has been observed that the proposed operator learning models above may not behave as operators when implemented on a computer, questioning the very essence of what operator learning should be. \citet{bartolucci2023neural} contend that in addition to defining the operator at the continuous level, some form of continuous-discrete equivalence is necessary for an architecture to genuinely learn the underlying operator, rather than just discretizations of it. To this end, they introduce the unifying mathematical framework of \emph{Representation equivalent Neural Operators (ReNOs)} to ensure operations at the continuous and discrete level are equivalent. Lack of this equivalence is quantified in terms of aliasing errors. 

\emph{Gaussian processes.} Deep neural networks and their operator versions are only one possible nonlinear surrogate model for the true solution $u$. Another popular class of nonlinear models is \emph{Gaussian process regression} \cite{GP}, which belongs to a larger class of so-called Bayesian models. Gaussian process regression (GPR) relies on the assumption that $u$ is drawn from a Gaussian measure on a suitable function space, parameterized by,
\begin{equation}
\label{eq:gpr1}
u_\theta({x}) \sim \mathrm{GP}(m_\theta({x}), k_\theta({x}, {x}')),
\end{equation}
where $m_\theta({{x}})=\E{u_\theta({x})}$ is the mean and the underlying covariance function is given by $k_\theta({{x}}, {x}')=\E{(u_\theta({x}) - m_\theta({x})) (u_\theta({x}') - m_\theta({x}'))}$. 

Popular choices for the covariance function in \eqref{eq:gpr1} are the squared exponential kernel (which is a radial function as for RBFs) and Matérn covariance functions, 
\begin{align}
\label{eq:SE_cov}
	k_{\mathrm{SE}}({x}, {x}') &= \exp{\bigg(-\frac{|| {x} - {x}'||^2}{2\ell^2}\bigg)}, \\
	 k_{\mathrm{Matern}}(x, x')&={\frac {2^{1-\nu }}{\Gamma (\nu )}}{\Bigg (}{\sqrt {2\nu }}{\frac {||x -x'||}{\ell }}{\Bigg )}^{\nu }K_{\nu }{\Bigg (}{\sqrt {2\nu }}{\frac {||x -x'||}{\ell }}{\Bigg )}.
\end{align}
Here $||\cdot||$ denotes the standard euclidean norm, $K_\nu$ is the modified Bessel function of the second kind and $\ell$ the \textit{characteristic length}, describing the length scale of the correlations between the points ${x}$ and ${x}'$. The parameter vector then consists of the hyperparameters of $m_\theta$ and $k_\theta$, such as $\nu$ and $\ell$. 

\subsection{Physics-informed residuals and loss functions}\label{sec:residuals}

With a chosen model class in place, thee key of physics-informed learning is to find a \emph{physics-informed loss functional} $\cE_{\mathrm{PIL}}[\cdot]$, independent of $u$, such that when it is minimized over the class of models $\{u_\theta : \theta\in\Theta\}$ the minimizer is a good approximation of the PDE solution $u$. In other words, we must find a functional $\cE_{\mathrm{PIL}}[\cdot]$ such that for $\theta^* = \argmin_{\theta} \cE_{\mathrm{PIL}} [u_\theta]$ it holds that the minimizer $u_{\theta^*}$ is a good approximation of $u$. We discuss three alternative formulations for a PDE, i.e. definitions of what it means to solve a PDE, all leading to different functionals and their corresponding \emph{loss functions}:
\begin{itemize}
    \item The classical formulation of the PDE, as in \eqref{eq:pde}, in which a point-wise condition on a function and its derivatives is stated. 
    \item The weak formulation of the PDE, which is a more general (weaker) condition than the classical formulation, often obtained by multiplying the classical formulation with a smooth test function and performing integration by parts. 
    \item The variational formulation of the PDE, where the PDE solution can be defined as the minimizer of a functional. 
\end{itemize}
Note that not all formulations are available for all PDEs. 

\subsubsection{Classical formulation}\label{sec:classical-residual}

Physics-informed learning is based on the elementary observation that for the true solution of the PDE $u$ it holds that $\cL(u)=f$ and $\cB(u)=g$, and the subsequent hope that if a model $u_\theta$ satisfies that $\cL(u_\theta)-f\approx 0$ and $\cB(u_\theta)-g\approx 0$, that then also $u\approx u_\theta$. To this end, a number of \emph{residuals} are introduced. First and foremost, the \emph{PDE residual} (or \emph{interior residual}) measures how well the model $u_\theta$ satisfies the PDE, 
\begin{eqnarray}\label{eq:PDE-residual}
    \rpde[u_\theta](z) = \cL[u_\theta](z) -f(z) \qquad \forall z\in \Omega,  
\end{eqnarray}
where we recall that $\Omega = D$ and $z=x$ for a time-independent PDE and $\Omega = D\times [0,T]$ and $z = (x,t)$ for a time-dependent PDE. In the first case, the boundary operator $\cB$ in \eqref{eq:pde} uniquely refers to the spatial boundary conditions, such that $\cB(u)=g$ corresponds to $D^k_x u = g$. The corresponding \emph{spatial boundary residual} is then given by, 
\begin{eqnarray}
    \rs[u_\theta](x) = D^k_x u_\theta(x) - g(x) \qquad \forall x\in \partial D. 
\end{eqnarray}
For time-dependent PDEs the boundary operator $\cB$ generally also prescribes the initial condition, i.e. that $u(x,0) = u_0(x)$ for all $x\in D$ for a given function $u_0$. This condition gives rise to the \emph{temporal boundary residual}, 
\begin{eqnarray}
    \rt[u_\theta](x) =  u_\theta(x,0) - u_0(x) \qquad \forall x\in  D. 
\end{eqnarray}
The physics-informed learning framework hence dictates to simultaneously minimize the residuals $ \rpde[u_\theta]$ and $\rs[u_\theta]$, and additionally $\rt[u_\theta]$ for a time-dependent PDE. This is usually translated as the minimization problem
\begin{eqnarray}\label{eq:pil-min}
     \min_{\theta\in\Theta} \cE_{\mathrm{PIL}}[u_\theta]^2, 
\end{eqnarray}
where for a time-independent PDE the \emph{physics-informed loss} $\cE_{\mathrm{PIL}}$ is given by, 
\begin{eqnarray}\label{eq:epil-ti}
    \cE_{\mathrm{PIL}}[u_\theta]^2 = \int_{D\times[0,T]}  \rpde[u_\theta]^2 + \lambda_s \int_{\partial D} \rs[u_\theta]^2,
\end{eqnarray}
where $\lambda_s>0$ is a weighting hyperparameter, and for a time-dependent PDE the quantity $\cE_{\mathrm{PIL}}$ is given by, 
\begin{eqnarray}\label{eq:epil-td}
    \cE_{\mathrm{PIL}}[u_\theta]^2 = \int_{D\times[0,T]}  \rpde[u_\theta]^2 + \lambda_s \int_{\partial D\times [0,T]} \rs[u_\theta]^2+ \lambda_t \int_D \rt[u_\theta]^2,
\end{eqnarray}
where $\lambda_s, \lambda_t>0$ are weighting hyperparameters. Since the boundary conditions (and initial condition) are enforced through an additional term in the loss function, we allow the possbility that the final model will only approximately satisfy the boundary conditions. This is generally referred to as having \emph{soft boundary conditions}. On the other hand, if one chooses the model class in such a way that all $u_\theta$ exactly satisfy the boundary conditions ($\cB[u_\theta]=g$), one has implemented \emph{hard boundary conditions}.

\begin{remark}
The residuals and loss functions defined above can readily be extended to systems of PDEs and more complicated boundary conditions. If the system of PDEs is given by $\cL^{(i)}[u] = f^{(i)}$ then one can readily define multiple PDE residuals $\rpde^{(i)}$ and consider the sum of integrals of the square PDE residuals  $\sum_i \int (\rpde^{(i)})^2$ in the physics-informed loss. The generalization to more complicated boundary conditions, including piecewise defined boundary conditions, proceeds in a completely analogous manner. 
\end{remark}

The minimization problem \eqref{eq:pil-min}, solely based on the physics-informed loss functions defined in \eqref{eq:epil-ti} and \eqref{eq:epil-td}, is the most basic version of physics-informed learning and can be extended by adding additional terms to the objective in \eqref{eq:pil-min}. For example, in some settings one might have access to measurements $\map(u) = g$ on a subdomain $\dom'$, where $\map$ is a so-called \emph{observable}. This case is particularly relevant for inverse problems (Section \ref{sec:inverse}). In this case it makes sense to add a \emph{data term} to the optimization objective, defined in terms of a \emph{data residual} $\cR_d$, 
\begin{eqnarray}\label{eq:data-residual}
    \cR_d[u_\theta](x) = \map[u](x) - g(x) \qquad \forall x\in\dom'. 
\end{eqnarray}
Note that the boundary residual can be seen as a special case of the data residual with $\Psi = \cB$ and $\dom'=\partial \Omega$. Additionally, it might happen in practice that one adds a \emph{parameter regularization term}, often the $L^q$-norm of the parameters $\theta$, to the loss function to either aid the minimization process or improve the (generalization) properties of the final model. A generalization of the physics-informed minimization problem \eqref{eq:pil-min} could therefore be,
\begin{eqnarray}
    \min_{\theta\in\Theta} \left( \cE_{\mathrm{PIL}}[u_\theta]^2 + \lambda_{d} \int_{\dom'} \cR_d[u_\theta]^2 + \lambda_r \norm{\theta}^q_q\right). 
\end{eqnarray}

\subsubsection{Weak formulation}\label{sec:weak-residuals}

In the previous section, we have considered PDE residuals and physics-informed loss functions that are based on the strong formulation of the PDE. However, for various classes of PDEs there might not be a smooth classical solution $u$ for which it holds in every $x\in\Omega$ that $\cL[u](x)=f(x)$. Fortunately, one can sometimes generalize the definition of `a solution to a PDE' in a precise manner such that there are functions that satisfy this generalized definition after all. Such functions are called \emph{weak solutions}. They are said to satisy the \emph{weak formulation} of the PDE, which can be obtained by multiplying the strong formulation of the PDE \eqref{eq:pde} with a smooth and compactly supported test function $\psi$ and then performing integration by parts, 
\begin{eqnarray}\label{eq:weak-formulation}
   \int_\Omega u \cL^*[\psi] = \int_\Omega f\psi, 
\end{eqnarray}
where $\cL^*$ is the (formal) adjoint of $\cL$ (if it exists). %and $\abs{\cL}$ is the number of times one needs to perform integration by parts to obtain the left-hand side of \eqref{eq:weak-formulation} from $\int_\Omega \cL[u]\psi$. 
Consequently, a function $u$ is said to be a weak solution if it satisfies \eqref{eq:weak-formulation} for all (compactly supported) test functions $\psi$ in some predefined set of functions. Note that now $u$ does not need to be differentiable anymore, but merely integrable. 

Weak formulations such as \eqref{eq:weak-formulation} form the basis of many classical numerical methods, in particular the finite element method (FEM; Section \ref{sec:linear-models}). Often, there is however no need to perform integral by parts until $u$ is completely free of derivatives. In FEM, the weak formulation for second-order elliptic equations is generally obtained by performing integration by parts only one time. For the Laplacian $\cL = \Delta$ one then has that $\int_\Omega \Delta u \psi = - \int_\Omega \nabla u \cdot \nabla \psi$ (given suitable boundary conditions), which is a formulation with many useful properties in practice. 

Other than using spectral methods (as in Section \ref{sec:linear-models}), there are now different avenues one can follow to obtain a \emph{weak physics-informed loss function} based on \eqref{eq:weak-formulation}. The first method draws inspiration from the Petrov-Galerkin method (Section \ref{sec:linear-models}) and considers a finite set of $K$ test functions $\psi_k$ for which \eqref{eq:weak-formulation} must hold . The corresponding loss function is then defined as, 
\begin{eqnarray}
    \cE_{\mathrm{VPIL}}[u_\theta]^2 = \sum_{k=1}^K \left(  \int_\Omega u_\theta \cL^*[\psi_k] - \int_\Omega f\psi_k \right)^2, 
\end{eqnarray}
where additional terms can be added to take care of any boundary conditions. This method was first introduced for neural networks under the name \emph{variational PINNs (VPINNs)} \cite{kharazmi2019variational} and used sine and cosine functions or polynomials as test functions. In analogy with $hp$-FEM, the method was later adapted to use localized test functions, leading to $hp$-VPINNs \cite{kharazmi2021hp}. In this case the model $u_\theta$ is of limited regularity and might not be sufficiently many times differentiable for $\cL[u_\theta]$ to be well-defined. 

A second variant of a weak physics-informed loss function can be obtained by replacing $\psi$ by a parametric model $\psi_\vartheta$. In particular, one chooses the $\psi_\vartheta$ to be the function that maximizes the squared weak residual, leading to the loss function, 
\begin{eqnarray}
    \cE_{\mathrm{wPIL}}[u_\theta]^2 = \max_\vartheta \left( \int_\Omega u_\theta \cL^*[\psi_\vartheta] - \int_\Omega f\psi_\vartheta \right)^2,
\end{eqnarray}
where again additional terms can be added to take care of any boundary conditions. Using $\cE_{\mathrm{wPIL}}$ has the advantage over $\cE_{\mathrm{VPIL}}$ that one does not need a basis of test functions $\{\psi_k\}_k$, which might be inconvenient in high-dimensional settings and settings with complex geometries. On the other hand, minimizing $\cE_{\mathrm{wPIL}}$ corresponds to solving a min-max optimization problem where maximum as well as the minimum are taken with respect to neural network training parameters, which is considerably more challenging than a minimization problem. 

The weak physics-informed loss $\cE_{\mathrm{wPIL}}$ was originally proposed to approximate scalar conservation laws with neural networks under the name \emph{weak PINNs (wPINNs)} \cite{deryck2022wpinns}. As there might be infinitely many weak solutions for a given scalar conservation law, one needs to impose further requirements other than \eqref{eq:weak-formulation} to guarantee that a scalar conservation law has a unique solution. This additional challenge is discussed in the below example. 

\begin{example}[wPINNs for scalar conservation laws]\label{ex:scl-weak-residual}
We revisit Example \ref{def:scl} and recall that weak solutions of scalar conservation laws need not be unique \cite{holden2015front}. To recover uniqueness, one needs to impose additional admissibility criteria in the form of \emph{entropy conditions}. To this end, we consider the so-called \emph{Kruzkhov entropy functions}, given by $|u-c|$ for any $c\in\R$, and the resulting entropy flux functions,
\begin{equation}\label{eq:def-q}
    Q:\R^2\to \R: (u,c)\mapsto \sgn{u-c}(f(u)-f(c)).
\end{equation}
We then say that a function $u\in L^\infty(\R\times \R_+)$ is an \emph{entropy solution} of \eqref{eq:scl} with initial data $u_0\in L^\infty(\R)$ if $u$ is a weak solution of \eqref{eq:scl} and if $u$ satisfies that ,
\begin{eqnarray}
    \partial_t \abs{u-c}+\partial_x Q[u;c]\leq 0,
\end{eqnarray}
in a distributional sense, or more precisely that,
\begin{eqnarray}\label{eq:kruzkov}
    &&\int_{0}^T\int_\R \left(\abs{u-c}\varphi_t + Q[u;c]\varphi_x\right) dxdt \\&&\quad- \int_\R\left(\abs{u(x,T)-c}\phi(x,T)- \abs{u_0(x)-c}\phi(x,0)\right)dx \geq 0
\end{eqnarray}
for all $\phi\in C^1_c(\R\times \R_+)$, $c\in \R$ and $T>0$. It holds that entropy solutions are unique and continuous in time. For more details on classical, weak and entropy solutions for hyperbolic PDEs we refer the reader to \cite{leveque2002finite, holden2015front}. 

This definition inspired the authors of \cite{deryck2022wpinns} to define the following \emph{Kruzkhov entropy residual}, 
\begin{equation}\label{eq:R-def}
    \cR(v,\phi, c) := - \int_D\int_{[0,T]} \left(\abs{v(x,t)-c}\partial_t\phi(x,t) + Q[v(x,t); c]\partial_x\phi(x,t) \right)dx dt, 
\end{equation}
based on which then the following weak physics-informed loss is defined,
\begin{eqnarray}
    \cE_{\mathrm{wPIL}}[u_\theta] = \max_{\vartheta,c} \cR(u_\theta,\phi_\vartheta, c),
\end{eqnarray}
where additional terms need to be added to take care of initial and boundary conditions. \cite[Section 2.4]{deryck2022wpinns}. 
\end{example}
\subsubsection{Variational formulation}\label{sec:variational-residual}

Finally, we discuss a third physics-informed loss function that can be formulated for \emph{variational problems}. Variational problems are PDEs \eqref{eq:pde} where the differential operator can (at least formally) be written as the derivative of an \emph{energy functional} $I$, in the sense that,
\begin{eqnarray}\label{eq:varitional-principle}
    \cL[u] -f = I'[u] = 0. 
\end{eqnarray}
The \emph{calculus of variations} thus allows to replace the problem of solving the PDE \eqref{eq:pde} to finding the minimum of the functional $I$, which might be much easier \cite{evans2022partial}. A first instance of this observation is known as \emph{Dirichlet's principle}. 

\begin{example}[Dirichlet's principle]\label{ex:dirichlet-principle}
    Let $\Omega\subset \R^d$ be open and bounded and define $\cA = \{w\in C^2(\Omega): w=g \text{ on } \partial \Omega\}$. If one then considers the functional $I(w) = \tfrac{1}{2}\int_\Omega \abs{\nabla w}^2$ then it holds for function $u\in\cA$ satisfies, 
    \begin{eqnarray}
        I[u] = \min_{w\in\cA} I[w] \quad \iff \quad \begin{cases}
            \Delta u = 0 & x\in\Omega, \\ u=g & x\in\partial \Omega. 
        \end{cases}
    \end{eqnarray}
    In other words, $u$ minimizes the functional $I$ over $\cA$ if and only if $u$ is a solution to Laplace's equation on $\Omega$ with Dirichlet boundary conditions $g$. 
\end{example}

Dirichlet's principle can be generalized to functionals $I$ of the form,
\begin{eqnarray}\label{eq:energy-functional-L}
    I[w] = \int_\Omega L(\nabla w(x), w(x), x)dx,
\end{eqnarray}
that are the integral of a so-called \emph{Lagrangian} $L$, 
\begin{eqnarray}
    \quad L: \R^d\times \R \times \Omega \to \R: (p,z,x)\mapsto L(p,z,x), 
\end{eqnarray}
where $w:\overline{\Omega}\to\R$ is a smooth function that satisfies some prescribed boundary conditions. One can then verify that any smooth minimizer $u$ of $I[\cdot]$ must satisfy the following nonlinear PDE, 
\begin{eqnarray}
    - \sum_{i=1}^d \partial_{x_i}(\partial_{p_i} L(\nabla u, u, x)) + \partial_{z} L(\nabla u, u, x) = 0, \qquad x\in\Omega. 
\end{eqnarray}
This equation is the \emph{Euler-Lagrange equation} associated with the energy functional $I$ from \eqref{eq:energy-functional-L}. Revisiting Example \ref{ex:dirichlet-principle}, one can verify that Laplace's equation is the Euler-Langrage equation associated with the energy functional based on the Lagrangian $L(p,z,x) = \tfrac{1}{2}\abs{p}^2$. Dirichlet's principle can also be further generalized. 

\begin{example}[Generalized Dirichlet's principle]\label{ex:gen-dirichlet-principle}
Let $a^{ij}:\Omega\to\R$ be functions with $a^{ij}=a^{ji}$ and consider the elliptic PDE, 
\begin{eqnarray}
    -\sum_{i,j=1}^d \partial_{x_i}(a^{ij}\partial_{x_j} u) = f, \qquad x\in\Omega.
\end{eqnarray}
The above PDE is the Euler-Lagrange equation associated with the functional of the form \eqref{eq:energy-functional-L} and the Lagrangian, 
\begin{eqnarray}
      L(p,z,x) = \frac{1}{2} \sum_{i,j=1}^d a^{ij}(x) p_i p_j -zf(x).
\end{eqnarray}
\end{example}
More examples and conditions under which $I[\cdot]$ has a (unique) minimizer can be found in e.g. \cite{evans2022partial}. 

Solving a PDE within the class of variational problems by minimizing its associated energy functional is known in the literature as the \emph{Ritz(-Galerkin) method}, or also \emph{Rayleigh-Ritz method} for eigenvalue problems. More recently, \cite{yu2018deep} have proposed the \emph{Deep Ritz method}, which aims to minimize the energy functional over the set of deep neural networks, 
\begin{eqnarray}
    \min_\theta \cE_{\mathrm{DRM}}[u_\theta]^2, \qquad \cE_{\mathrm{DRM}}[u_\theta]^2:= I[u_\theta] + \lambda \int_{\partial \Omega} (\cB[u_\theta]-g)^2,
\end{eqnarray}
where the second term was added as the used neural networks do not automatically satisfy the boundary conditions of the PDE \eqref{eq:pde}. 
% see https://web.stanford.edu/class/math220b/handouts/calcvar.pdf

\subsection{Quadrature rules}\label{sec:quad}

All loss functions discussed in the previous sections are formulated as integrals in terms of the model $u_\theta$. Very often it will not be feasible or computationally desirable to evaluate these integrals in an exact way. Instead, one can discretize the integral and approximate it using a sum. We recall some basic results on numerical quadrature rules for integral approximation. 

For a mapping $g: \Omega \to \R^m$, $\Omega\subset \R^{\overline{d}}$, such that $g \in L^1(\Omega)$, we are interested in approximating the integral,
$$
\overline{g}:= \int\limits_{\dom} g(y) dy,
$$
with $dy$ denoting the Lebesgue measure on $\Omega$. In order to approximate the above integral by a quadrature rule, we need the quadrature points $y_{i} \in \dom$ for $1 \leq i \leq N$, for some $N \in \N$ as well as weights $w_i$, with $w_i \in \R_+$. Then a quadrature is defined by,
\begin{equation}
    \label{eq:quad}
    \cQ^\Omega_N[g] := \sum\limits_{i=1}^N w_i g(y_i),
\end{equation}
for weights $w_i$ and quadrature points $y_i$. Depending on the choice of weights and quadrature points, as well as the regularity of $g$, the quadrature error is bounded,
\begin{equation}
    \label{eq:assm3}
    \left|\overline{g} - \cQ^\Omega_N[g]\right| \leq C_{\text{quad}} N^{-\alpha},
\end{equation}
for some $\alpha > 0$ and where $C_{\text{quad}}$ depends on $g$ and its properties. 

As an elementary example, we mention the \emph{midpoint rule}. For $M\in\mathbb{N}$, we partition $\dom$ into $N \sim M^d$ cubes of edge length $1/M$ and we denote by $\{y_i\}_{i=1}^N$ the midpoints of these cubes. The formula and accuracy of the midpoint rule $\qu{N}^\dom$ are then given by, 
\begin{equation}\label{eq:quad-error}
    \qu{N}^\dom[g] := \frac{1}{N}\sum_{i=1}^Ng(y_i), \qquad \abs{\overline{g} - \qu{M}^\dom[g]} \leq C_g N^{-2/d},
\end{equation}
where $C_g \lesssim \norm{g}_{C^2}$. 

As long as the domain $\dom$ is in reasonably low dimension, i.e. $\bar{d} \leq 4$, we can use the midpoint rule and more general standard (composite) Gauss quadrature rules on an underlying grid. In this case, the quadrature points and weights depend on the order of the quadrature rule \cite{SBbook} and the rate $\alpha$ depends on the regularity of the underlying integrand. 
On the other hand, these grid based quadrature rules are not suitable for domains in high dimensions. For moderately high dimensions, i.e. $4 \leq \bar{d} \leq 20$, we can use low-discrepancy sequences, such as the Sobol and Halton sequences, as quadrature points. As long as the integrand $g$ is of bounded Hardy-Krause variation, the error in \eqref{eq:assm3} converges at a rate $(\log(N))^{\bar{d}}N^{-1}$ \cite{longo2020higher}. For problems in very high dimensions, $d \gg 20$, Monte-Carlo quadrature is the numerical integration method of choice. In this case, the quadrature points are randomly chosen, independent and identically distributed (with respect to a scaled Lebesgue measure). 

As an example, we discuss how the physics-informed loss based on the classical formulation for a time-dependent PDE \eqref{eq:epil-td} can be discretized. As the loss function \eqref{eq:epil-td} contains integrals over three different domains ($D\times[0,T]$, $\partial D \times[0,T]$ and $D$), we must consider three different quadratures. We consider quadratures \eqref{eq:quad} for which the weights are the inverse of the number of grid points, such as for the midpoint or Monte-Carlo quadrature, and we denote the quadrature points by $\cS_i = \{(z_i)\}_{1\leq i \leq N_i}\subset D\times [0,T]$, $\cS_s = \{y_i\}_{1\leq i \leq N_s}\subset\partial D\times[0,T]$ and $\cS_t = \{x_i\}_{1\leq i \leq N_t}\subset D$. Following machine learning terminology, we call the set of all quadrature points $\cS = (\cS_i , \cS_s, \cS_t)$ as the \emph{training set}. The resulting loss function $\cJ$ now does not only depend on $u_\theta$ but also on $\S$ and is given by, 
\begin{eqnarray}
    \cJ(\theta, \cS) &=& \qu{N_i}^{D\times[0,T]}[\rpde^2] + \lambda_s \qu{N_s}^{\partial D\times[0,T]} [\rs^2] + \lambda_t \qu{N_t}^{D} [\rt^2] \\
    &=& \frac{1}{N_i}\sum_{i=1}^{N_i} (\cL[u_\theta](z_i) - f(z_i))^2 + \frac{\lambda_s}{N_s}\sum_{i=1}^{N_s} (u_\theta(y_i)-u(y_i))^2 \\&&+ \frac{\lambda_t}{N_t}\sum_{i=1}^{N_t} (u_\theta(x_i,0)-u(x_i,0))^2. 
\end{eqnarray}
For any of the other loss functions defined in Section \ref{sec:residuals} a similar discretization can be defined. In practice, physics-informed learning thus comes down to solving the minimization problem given by,
\begin{eqnarray}\label{eq:minimization-problem}
    \theta^*(\cS) := \argmin_{\theta\in\Theta} \cJ(\theta,\cS).
\end{eqnarray}
The final optimized or \emph{trained} model is then given by $u^*:=u_{\theta^*(\cS)}$. 

\subsection{Optimization}\label{sec:intro-optimization}

Since $\Theta$ is high-dimensional and the map $\theta\mapsto \mathcal{J}(\theta,\S)$ is generally non-convex, solving the minimization problem \eqref{eq:minimization-problem} can be very challenging. Fortunately, the loss function $\mathcal{J}$ is almost everywhere differentiable, so that gradient-based iterative optimization methods can be used. 

The simplest example of such an algorithm is \textbf{gradient descent}. Starting from an initial guess $\theta_0$, the idea is to take a small step in the parameter space $\Theta$ in the direction of the steepest descent of the loss function to obtain a new guess $\theta_1$. Note that this comes down to taking a small step in the opposite direction of the gradient evaluated in $\theta_0$. Repeating this procedure yields the following iterative formula,
\begin{equation}
    \theta_{\ell+1} = \theta_\ell - \eta_\ell \nabla_\theta \mathcal{J}(\theta_\ell,\S), \qquad \forall \ell\in\mathbb{N},
\end{equation}
where the learning rate $\eta_\ell$ controls the size of the step and is generally quite small. The gradient descent formula yields a sequence of parameters $\{\theta_\ell\}_{\ell\in\mathbb{N}}$ that converges to a local minimum of the loss function under very general conditions. Because of the non-convexity of the loss function, convergence to a global minimum can not be ensured. Another issue lies in the computation of the gradient $\nabla_\theta \mathcal{J}(\theta_\ell,\S)$. A simple rewriting of the definition of this gradient (up to regularization term),
\begin{equation}
    \nabla_\theta \mathcal{J}(\theta_\ell,\S) = \frac{1}{N}\sum_{i=1}^N \nabla_\theta \mathcal{J}_i(\theta_\ell) \quad \text{where }\mathcal{J}_i(\theta_\ell)=\mathcal{J}(\theta_\ell,\{(x_i,f(x_i))\}),
\end{equation}
reveals that one actually needs to evaluate $N=\abs{\S}$ gradients. As a consequence, the computation of the full gradient will be very slow and memory intensive for large training data sets. 

\textbf{Stochastic gradient descent} (SGD) overcomes this problem by only calculating the gradient in one data point instead of the full training data set. More precise, for every $\ell$ a random index $i_\ell$, with $1\leq i_\ell \leq N$ is chosen, resulting in the following update formula, 
\begin{eqnarray}
\theta_{\ell+1} = \theta_\ell - \eta_\ell \nabla_\theta \mathcal{J}_{i_\ell}(\theta_\ell), \qquad \forall \ell\in\mathbb{N}.
\end{eqnarray}
Similarly to gradient descent, stochastic gradient descent provably converges to a local minimum of the loss function under rather general conditions, although the convergence will be slower. In particular, the convergence will be noisier and there is no guarantee that $\mathcal{J}(\theta_{\ell+1})\leq \mathcal{J}(\theta_{\ell})$ for every $\ell$. Because the cost of each iteration is much lower than that of gradient descent, it might still makes sense to use stochastic gradient descent. 

\textbf{Mini-batch gradient descent} tries to combine the advantages of gradient descent and SGD by evaluating the gradient of the loss function in a subset of $\S$ in each iteration (as opposed to the full training set for gradient descent and one training sample for SGD). More precise, for every $\ell\in\mathbb{N}$ a subset $\S_\ell\subset \S$ of size $m$ is chosen. These subsets are referred to as \textit{mini-batches} and their size $m$ as the mini-batch size. In contrast to SGD, these subsets are not entirely randomly selected in practice. Instead, one randomly partitions the training set into $\lceil N/m \rceil$ mini-batches, which are then all used as a mini-batch $\S_\ell$ for consecutive $\ell$. Such a cycle of $\lceil N/m \rceil$ iterations is called an \textit{epoch}. After every epoch, the mini-batches are reshuffled, meaning that a new partition of the training set is drawn. In this way, every training sample is used one time during every epoch. The corresponding update formula reads as,
\begin{equation}\label{eq:MBGD}
    \theta_{\ell+1} = \theta_\ell - \eta_\ell \nabla_\theta \mathcal{J}(\theta_\ell,\S_{\ell}), \qquad \forall \ell\in\mathbb{N}. 
\end{equation}
By setting $m=N$, resp. $m=1$, one retrieves gradient descent, resp. SGD. For this reason, gradient descent is also often referred to as \textit{full batch} gradient descent. 

The performance of mini-batch gradient descent can be improved in many ways. One particularly popular example of such an improvement is the \textbf{Adam} optimizer \cite{adam}. Adam, of which the name is derived from \textit{adaptive moment estimation}, calculates at each $\ell$ (exponential) moving averages of the first and second moment of the mini-batch gradient $g_\ell$, 
\begin{equation}
   m_\ell = \beta_1 m_{\ell-1}+(1-\beta_1)g_\ell, \quad v_\ell = \beta_2 v_{\ell-1}+(1-\beta_2)g_\ell^2, \quad     g_\ell = \nabla_\theta \mathcal{J}(\theta_\ell,\S_{\ell}), 
\end{equation}
where $\beta_1$ and $\beta_2$ are close to 1. However, one can calculate that these moving averages are biased estimators. This bias can be removed using the following correction,
\begin{eqnarray}
\hat{m}_\ell = \frac{m_\ell}{1-\beta_1^\ell}, \quad \hat{v}_\ell = \frac{v_\ell}{1-\beta_2^\ell}, \qquad \forall \ell\in\mathbb{N}, 
\end{eqnarray}
where $\hat{m}_\ell, \hat{v}_\ell$ are unbiased estimators. 
One can obtain Adam from the basic mini-batch gradient descent update formula \eqref{eq:MBGD} by replacing $\nabla_\theta \mathcal{J}(\theta_\ell,\S_{\ell})$ by the moving average $m_\ell$ and by setting $\eta_\ell = \frac{\alpha}{\sqrt{\hat{v}_\ell+\epsilon}}$, where $\alpha>0$ is small and $\epsilon>0$ is close to machine precision. This leads to the update formula
\begin{eqnarray}
\theta_{\ell+1} = \theta_\ell - \alpha\frac{\hat{m}_\ell}{\sqrt{\hat{v}_\ell+\epsilon}}, \qquad \forall \ell\in\mathbb{N}. 
\end{eqnarray}
Compared to mini-batch gradient descent, the convergence of Adam will be less noisy for two reasons. First, a fraction of the noise will be removed since a moving average of the gradient is used. Second, the step size of Adam decreases if the gradient $g_\ell$ varies a lot, i.e. when $\hat{v}_\ell$ is large. These design choices make that Adam performs well on a large number of tasks, making it one of the most popular optimizers in deep learning. 

Finally, the interest in \textbf{second-order} optimization algorithms has been steadily growing over the past years. Most second-order methods used in practice are adaptations of the well-known Newton method, 
\begin{equation}
    \theta_{\ell+1} = \theta_\ell - \eta_\ell (\nabla^2_\theta \mathcal{J}(\theta, \S))^{-1} \nabla_\theta \mathcal{J}(\theta, \S), \qquad \forall \ell\in\mathbb{N}, 
\end{equation}
where $\nabla^2_\theta \mathcal{J}(\theta, \S)$ denotes the Hessian of the loss function. The generalization to a mini-batch version is immediate. As the computation of the Hessian and its inverse is quite costly, one often uses a \textit{quasi}-Newton method that computes an approximate inverse by solving the linear system $\nabla^2_\theta \mathcal{J}(\theta, \S)h_\ell = \nabla_\theta \mathcal{J}(\theta, \S)$. An example of a popular quasi-Newton method in deep learning for scientific computing is \textbf{L-BFGS} \cite{liu1989limited}. In many other application areas, however, first-order optimizers remain the dominant choice for a myriad of reasons. 
\subsubsection{Parameter initialization}\label{sec:initialization}
All iterative optimization algorithms above require an initial guess $\theta_0$. Making this initial guess is referred to as the parameter initialization, or weights initialization (for neural networks). Although it is often overlooked, a bad initialization can cause a lot of problems, as will be discussed in Section \ref{sec:training}. 

Especially for neural networks many initialization schemes have already been considered. When the initial weights are chosen too large, the corresponding gradients calculated by the optimization algorithm might be very large, leading to the \textit{exploding gradient problem}. Similarly, when the initial weights are too small, the gradients might also be very close to zero, leading to the \textit{vanishing gradient problem}. Fortunately, it is possible to choose the variance of the initial weights in such a way that the output of the network has the same order of magnitude as the input of the network. One such possible choice is \textbf{Xavier initialization} \cite{glorot2010understanding}, meaning that the weights are initializated as
\begin{eqnarray}
   (W_k)_{ij} \sim N\left(0,\frac{2g^2}{d_{k-1}+d_k}\right)\: \text{or} \: (W_k)_{ij} &\sim U\left(-g\sqrt{\frac{6}{d_{k-1}+d_k}},g\sqrt{\frac{6}{d_{k-1}+d_k}}\right), 
\end{eqnarray}
where $N$ denotes the normal distribution, $U$ the uniform distribution, $d_{k}$ is the output dimension of the $k$-th layer and the value $g$ depends on the activation function. For the ReLU activation one sets $g=\sqrt{2}$ and for the tanh activation function $g=1$ is a valid choice. Note that this initialization has primarily been proposed with a supervised learning setting in mind, e.g. when using a square loss such as $\int_\Omega(u_\theta-u)$. For the different loss functions of Section \ref{sec:residuals} different initialization schemes might be better suited. This will be discussed in Section \ref{sec:training}. 

Finally, it is customary to retrain neural network with different starting values $\theta_0$ (drawn from the same distribution) as the gradient-based optimizer might converge to a different local minimum for each $\theta_0$. One can then choose the `best' neural network based on some suitable criterion or combine the different neural networks if the setting allows this. Also note that this task can be performed in parallel. 

\subsection{Summary of algorithm}\label{sec:summary}

We summarize the physics-informed learning algorithm as follows, 
\begin{enumerate}
\item Choose a model class $\{u_\theta: \theta\in\Theta\}$ (Section \ref{sec:models}). 
    \item Choose a loss function based on the classical, weak or variational formulation of the PDE (Section \ref{sec:residuals}). 
    \item Generate a training set $\cS$ based on a suitable quadrature rule to obtain a feasibly computable loss function $\cJ(\theta,\cS)$ (Section \ref{sec:quad}). 
    \item Initialize the model parameters and run an optimization algorithm from Section \ref{sec:intro-optimization} until an approximate local minimum $\theta^*(\cS)$ is reached. The resulting function $u^*:=u_{\theta^*(\cS)}$ is the desired model for approximating the solution $u$ of the PDE \eqref{eq:pde}. 
\end{enumerate}

\section{Analysis}\label{sec:analysis}

As intuitive as the physics-informed learning framework, as summarized in Section \ref{sec:summary}, might seem, there are a priori little theoretical guarantees that it will actually work well. The aim of the next sections is therefore to theoretically analyze physics-informed machine learning and to give an overview of the available mathematical guarantees. 

One central element will be the development of error estimates for physics-informed machine learning. The relevant concept of error is the error emanating from approximating the solution $\bu$ of \eqref{eq:pde} by the model $\bu^{\ast} = u_{\theta^*(\S)}$,
    \begin{equation}
    \label{eq:pinn-error}
    \Eg(\theta):= \|\bu-\bu_\theta\|_{X}, \quad \Eg^* :=  \Eg(\theta^*(\S)).
\end{equation}
In general, we will choose $\norm{\cdot}_X = \norm{\cdot}_{L^2(\Omega)}$. 
%Clearly, the error depends on the chosen training set $\S$ and the optimized model with parameters $\theta^{\ast}(\S)$. However, we will suppress this dependence due to notational convenience. 
Note that it is not possible to compute $\Eg$ during the training process. On the other hand, we monitor the so-called \emph{training error} given by,
\begin{equation}
    \label{eq:pinn-train}
    \Eg_T(\theta,\S)^2:= \cJ(\theta,\cS) , \quad \Eg_T^* :=  \Eg_T(\theta^*(\S),\S),
\end{equation}
with $\cJ$ being the loss function defined in Section \ref{sec:quad} as the discretized version of the (integral-form) loss functions in Section \ref{sec:residuals}. 
Hence, the training error $\Eg_T^*$ can be readily computed, after training has been completed, from the loss function. We will also refer to the integral form of the physics-informed loss function (cf. Section \ref{sec:residuals}) as the \emph{generalization error} $\Eg_G$, 
\begin{eqnarray}\label{eq:pinn-gen}
     \Eg_G(\theta):= \cE_{\mathrm{PIL}}[u_\theta] , \quad \Eg_G^* :=  \Eg_G(\theta^*(\S),\S), 
\end{eqnarray}
as it measures how well the performance of a model transfers or \emph{generalizes} from the training set $\cS$ to the full domain $\Omega$. 

Given these definitions, some fundamental theoretical questions arise immediately: 

\begin{enumerate}
    \item[Q1.] \textbf{Existence}: \emph{Does there exist a model $\hat{u} $ in the hypothesis class for which the generalization error $\Eg_G(\hat{\theta})$ is small?}  More precisely, given a chosen error tolerance $\epsilon>0$, does there exist a model in the hypothesis class $\hat{u} = u_{\hat{\theta}}$ for some $\hat{\theta}\in\Theta$ such that for the corresponding generalization error $\Eg_G(\hat{\theta})$ it holds that $\Eg_G(\hat{\theta})<\epsilon$? As minimizing the generalization error (i.e. the physics-informed loss) is the key pillar of physics-informed learning, obtaining a positive answer to this question is of the utmost importance. Moreover, it would be desirable to relate the size of the parameter space $\Theta$ (eq. hypothesis class) to the accuracy $\epsilon$. For example, for linear models (Section \ref{sec:linear-models}) one would want to know how many functions $\phi_i$ are needed to ensure the existence of $\hat{\theta}\in\Theta$ such that $\Eg_G(\hat{\theta})<\epsilon$. Similarly, for neural networks, one desires to find estimates of how large a neural network (Section \ref{sec:nonlinear-models}) should be (in terms of depth, width and modulus of the weights) in order to ensure this. We will answer this question by considering the \emph{approximation error} of a model class in Section \ref{sec:approximation-error}. 
    \item[Q2.] \textbf{Stability}: \emph{Given that a model $u_\theta$ has a small generalization error $\cE_G({\theta})$, will the corresponding total error $\Eg({\theta})$ be small as well?} In other words, is there a function $\delta: \R_+\to\R_+$ with $\lim_{\epsilon\rightarrow 0} \delta(\epsilon)=0$ such that $\Eg_G({\theta})<\epsilon$ implies that $\Eg({\theta})<\delta(\epsilon)$ for any $\epsilon>0$? In practice, we will be even a bit more ambitious as we hope to find constants $C, \alpha>0$ (independent of $\theta$) such that 
    \begin{eqnarray}\label{eq:stability-q2}
        \cE(\theta) \leq C \cE_G(\theta)^\alpha.
    \end{eqnarray}
    This is again an essential ingredient, as obtaining a small generalization error $\cE_G^*$ is a priori meaningless, given that one only cares about the total error $\cE^*$. The answer to this question first and foremost depends on the properties of the underlying PDE, and only to a lesser extent on the model class. We will discuss this question in Section \ref{sec:stability}. 
    \item[Q3.] \textbf{Generalization}: \emph{Given a small training error $\Eg_T^*$ and a sufficiently large training set $\S$, will the corresponding generalization error $\cE_G^*$ also be small?} This question is not uniquely tied to physics-informed machine learning and is closely tied to the accuracy of the chosen quadrature rule (Section \ref{sec:quad}). We will use standard arguments to answer this question in Section \ref{sec:generalization}. 
    \item[Q4.] \textbf{Optimization}: \emph{Can the training error $\Eg_T^*$ be made sufficiently close to $\min_\theta \cJ(\theta, \cS)$?} This final question acknowledges the fact that it might be very hard to solve the minimization problem \eqref{eq:minimization-problem}, even approximately. Especially for physics-informed neural networks there is a growing literature on the settings in which the optimization problem is hard to solve, or even infeasible, and how to potentially overcoming these training issues. We theoretically analyze these phenomena in Section \ref{sec:training}. 
\end{enumerate}

The above questions are of fundamental importance as affirmative answers to them certify that the model that minimizes the optimization problem \eqref{eq:minimization-problem}, denoted by $u^*$, is a good approximation of $u$ in the sense that the error $\cE^*$ is small. We show this by first proving an error decomposition for the generalization error $\cE_G^*$, for which we need the following auxiliary result. 

\begin{lemma}\label{lem:bound-Eg}
For any $\theta_1, \theta_2 \in \Theta$ and training set $\S$ it holds that
\begin{equation}
    \Eg_G(\theta_1) \leq \Eg_G(\theta_2) + 2 \sup_{\theta\in\Theta}\abs{\Et(\theta,\S)-\Eg_G(\theta)} + \Et(\theta_1,\S)-\Et(\theta_2,\S). 
\end{equation}
\end{lemma}
\begin{proof}
Fix $\theta_1,\theta_2\in\Theta$ and generate a random training set $\S$. The proof consists of the repeated adding, subtracting and removing from terms. It holds that
\begin{eqnarray}
        \Eg_G(\theta_1) &=& \Eg_G(\theta_2) + \Eg_G(\theta_1) - \Eg_G(\theta_2) \\
        &=& \Eg_G(\theta_2) - (\Et(\theta_1,\S)-\Eg_G(\theta_1)) + (\Et(\theta_2,\S)-\Eg_G(\theta_2)) \\&&+\Et(\theta_1,\S) -\Et(\theta_2,\S) \\
        &\leq& \Eg_G(\theta_2) + 2\max_{\theta\in\{\theta_2,\theta_1\}}\abs{\Et(\theta,\S)-\Eg_G(\theta)}\\&&+\Et(\theta_1,\S) -\Et(\theta_2,\S). %\\
        %&\leq& \Eg(\theta_2) + 2 \sup_{\theta\in\Theta}\abs{\Et(\theta,\S)-\Eg(\theta)} + \Et(\theta_1,\S),
\end{eqnarray}
%where we used that $\Et(\theta_2,\S)\geq 0$. 
\end{proof}
Assuming that we can answer Q1 in an affirmative manner, we can now take $\theta_1 = \theta^*(\S)$ and $\theta_2 = \hat{\theta}$ (as in Q1). Moreover, if we also assume that we can positively answer question Q2, or more strongly we can find constants $C,\alpha> 0$ such that \eqref{eq:stability-q2} holds, then we find the following error decomposition,
\begin{equation}\label{eq:error-decomposition}
    \cE^* \leq C \big(\underbrace{\Eg_G(\hat{\theta})}_{\text{approximation error}} + 2 \underbrace{\sup_{\theta\in\Theta}\abs{\Et(\theta,\S)-\Eg_G(\theta)}}_{\text{generalization gap}} + \underbrace{\Et^*-\Et(\hat{\theta})}_{\text{optimization error}}\big)^\alpha. 
\end{equation}
The total error $\cE^*$ is now decomposed into three error sources. The \textit{approximation error} $\Eg(\hat{\theta})$ should be provably small, by answering Q1. The second term of the RHS, termed \textit{generalization gap}, quantifies how well the training error approximates the generalization error, conform question Q3. Finally, an \textit{optimization error} is incurred due to the inability of the optimization procedure to find $\hat{\theta}$ based on a finite training data set. Indeed one can see that if $\theta^*(\S)=\hat{\theta}$, the optimization error vanishes. This source of error is addressed by question Q4. Hence, an affirmative answer to questions Q1-Q4 leads to a small error $\cE^*$. 

% The error decomposition presented in \eqref{eq:error-decomposition} is suitable for use as an \emph{a priori error estimate}, in the sense that it if one can prove in advance that the three sources of error are small, then one has a guarantee that $\cE^*$ will be small even before optimizing the model $u_\theta$. 

% In what follows, we will also use the following error estimate, which holds if one can answer Q1 in an affirmative manner, 
% \begin{eqnarray}\label{eq:error-decomp-2}
%     \cE^* \leq C\left( \cE_G^*  \right)^\alpha. 
% \end{eqnarray}

We will present general results to answer questions Q1-Q4 and then apply them to the following prototypical examples to further highlight interesting phenomena: 
\begin{itemize}
    \item The Navier-Stokes equation (Example \ref{def:NS-equation}) as an example of a challenging, but low-dimensional PDE. 
    \item The heat equation (Example \ref{def:semilinear-heat}) as a prototypical example of a potentially very high-dimensional PDE. We will investigate whether one can overcome the \emph{curse of dimensionality} through physics-informed learning for this equation and related PDEs. 
    \item Invisic scalar conservation laws (Example \ref{def:scl} with $\nu=0$) will serve as an example of a PDE where the solutions might not be smooth and even discontinuous. As a result, the weak residuals from Section \ref{sec:weak-residuals} will be needed. 
    \item Poisson's equation (Example \ref{def:poisson}) as an example for which the variational formulation of Section \ref{sec:variational-residual} can be used. 
\end{itemize}

\section{Approximation error}\label{sec:approximation-error}

In this section we answer the first question Q1. 

\begin{question}[Q1]
    \emph{Does there exist a model $\hat{u} = u_{\hat{\theta}}$ in the hypothesis class for which the generalization error $\Eg_G(\hat{\theta})$ is small?}
\end{question}

The difficulty in answering this question lies in the fact that the generalization error $\Eg_G({\theta})$ in physics-informed learning is given by the physics-informed loss of the model $\cE_{\mathrm{PIL}}(u_\theta)$ rather than just being $\norm{u-u_\theta}_{L^2}^2$ as it would be the case for regression. For neural networks, for example, the universal approximation theorems for neural networks (e.g. \citet{Cybenko1989}) guarantees that any measurable function can be approximated by a neural network in supremum-norm. They however do not imply an affirmative answer to question Q1: a neural network that approximates a function well in supremum-norm might be highly oscillatory, such that the derivatives of the network and that of the function are very different, giving rise to a large PDE residual. Hence, we require results on the existence of models that approximate functions in a \emph{stronger norm} than the supremum-norm. In particular, the norm should quantify how well the derivatives of the model approximate those of the function. 

For solutions of PDEs, a very natural norm that satisfies this criterion is the \emph{Sobolev norm}. For $s\in\N$ and $q\in[1,\infty]$, the Sobolev space $W^{s,q}(\dom;\R^m)$ is the space of all functions $f:\dom\to \R^m$ for which $f$, as well as the (weak) derivatives of $f$ up to order $s$ belong to $L^q(\dom;\R^m)$. When more smoothness is available, one could replace the space $W^{s,\infty}(\dom;\R^m)$ with the space of $s$ times continuously differentiable functions $C^s(\dom;\R^m)$ or can find more information in Appendix \ref{sec:sobolev}. 

Given an approximation result in some Sobolev (semi-)norm, we can prove that the physics-informed loss of the model $\cE_{\mathrm{PIL}}(u_\theta)$ if the following assumption holds true. 

When the physics-informed loss is based on the strong (classical) formulation of the PDE (cf. Section \ref{sec:classical-residual}), we assume that it can be bounded in terms of the errors related to all relevant partial derivatives, denoted by $D^{(k,\bmalpha)}:=D^k_tD^\bmalpha_x := \partial_t^k \partial^{\alpha_1}_{x_1}\ldots \partial^{\alpha_d}_{x_d}$, for $(k,\bmalpha)\in\N_0^{d+1}$ for time-dependent PDEs. This assumption is valid for many classical solutions of PDEs. 
\begin{assumption}\label{ass:diff-operator}
Let $k,\ell\in\N$,  $q\in[1,+\infty]$, $C>0$ be independent from $d$ and let $\cX \subset \{u_\theta: \theta\in\Theta\}$ and $\Omega = [0,T]\times D$. It holds for every $u_\theta\in\cX$ that 
\begin{equation}
    \norm{\cL[u_\theta]}_{L^q(\Omega)} \leq C \cdot \poly(d)\cdot 
    %\sum_{(k^*,\bmalpha)\in\N^{d+1}_0\: :\: k'\leq k, \norm{\bmalpha}_1\leq \ell } \norm{D^{(k',\bmalpha)}(\cG-\cG_\theta)}_{L^2}. 
    \sum_{\substack{(k',\bmalpha)\in\N^{d+1}_0\\ k'\leq k, \norm{\bmalpha}_1\leq \ell }} \norm{D^{(k',\bmalpha)}(u-u_\theta)}_{L^q(\Omega)}. 
\end{equation}
\end{assumption}
Note that the corresponding assumption for operator learning can be obtained by replacing $u$ by the operator $\cG$ and $u_\theta$ by $\cG_\theta$. We give a brief example to illustrate the assumption. 

\begin{example}[Heat equation]
    For the heat equation it holds that
    \begin{eqnarray}
        \norm{\cL[u_\theta]}_{L^q(\Omega)} &=& \norm{\cL[u_\theta] - \cL[u]}_{L^q(\Omega)} \\&\leq& \norm{\Delta(u_\theta-u)}_{L^q(\Omega)} + \norm{\partial_t(u_\theta-u)}_{L^q(\Omega)}.
    \end{eqnarray}
\end{example}

\subsection{First general result}\label{sec:approx-general}

Note that in particular Assumption \ref{ass:diff-operator} is satisfied when 
\begin{eqnarray}
    \norm{\cL[u_\theta]}_{L^q(\Omega)} \leq C \cdot \poly(d) \cdot \norm{u-u_\theta}_{W^{k, q}(\Omega)}. 
\end{eqnarray}
Hence, it suffices to prove that there exists $\hat{\theta}\in\Theta$ for which $\norm{u-u_{\hat{\theta}}}_{W^{k, q}(\Omega)}$ can be made small. 

For many linear models, such approximation results are widely available. For example, when one chooses the (default enumeration of the) $d$-dimensional Fourier basis as functions $\phi_i$ (cf. Section \ref{sec:linear-models}) then one can use the following result \cite{Specbook}. 

\begin{theorem}\label{lem:acc-trig-pol}
Let $s, k\in\N_0$ with $s>d/2$ and $s\geq k$, 
and $u\in C^s(\T^d)$ it holds that there exists $\hat{\theta}\in \R^{N}$ such that
\begin{equation}%\label{eq:acc-trig-pol}
    \textstyle \norm{u-\sum_{i=1}^{N^d} \hat{\theta}_i \phi_i}_{H^k(\T^d)}\leq C(s,d)N^{-(s-k)/d}\norm{u}_{H^s(\T^d)}, 
\end{equation}
for a constant $C(s,d)>0$ that only depends on $s$ and $d$. 
\end{theorem}

Analogous results are also available for neural networks. We state a result that is tailored to tanh neural networks \cite[Theorem 5.1]{deryck2021approximation}, but more general results are also available \cite{guhring2021approximation}. 

\begin{theorem}\label{thm:approx-sobolev}
Let $\Omega = [0,1]^d$. For every $N\in \mathbb{N}$ and every $f\in W^{s,\infty}(\Omega)$, there exists a tanh neural network $\widehat{f}$ with 2 hidden layers of width $N$ such that for every $0\leq k<s$ it holds that,
\begin{equation}
        \textstyle\norm{f-\widehat{f}}_{W^{k,\infty}(\Omega)} \leq C (\ln(cN))^k N^{-(s-k)/d},
\end{equation}
where $c,C>0$ are independent of $N$ and explicitly known. 
\end{theorem}
\begin{proof}
The main ingredients are a piecewise polynomial approximation, the existence of which is guaranteed by the Bramble-Hilbert lemma \cite{verfurth1999note}, and the ability of tanh neural networks to efficiently approximate polynomials, the multiplication operator and an approximate partition of unity. The full proof can be found in \cite[Theorem 5.1]{deryck2021approximation}. 
\end{proof}

\begin{remark}
In this section, we have mainly focused on bounding the (strong) PDE residual as it generally is the most difficult residual to bound, as it contains the highest derivatives. Indeed, when using the classical formulation the bounds on the spatial (and temporal) boundary residuals generally follow from a (Sobolev) trace inequality. Similarly, loss functions resulting from the weak or variational formulation contain less high derivatives of the neural network and are therefore easier to bound.
\end{remark}

As an example, we apply the above result to the Navier-Stokes equations. 

\begin{example}[Navier-Stokes]
One can use Theorem \ref{thm:approx-sobolev} to prove the existence of neural networks with a small generalization error for the Navier-Stokes equations (Example \ref{def:NS-equation}). 
The existence and regularity of the solution to  \eqref{eq:navier-stokes} depends on the regularity of $u_0$, as is stated by the following well-known theorem \cite[Theorem 3.4]{majda2002vorticity}. Other regularity results with different boundary conditions can be found in e.g. \citep{temam2001navier}.

\begin{theorem}\label{thm:NS-existence}
If $u_0\in H^r(\mathbb{T}^d)$ with $r>\frac{d}{2}+2$ and $\div{u_0}=0$, then there exist $T>0$ and a classical solution $u$ to the Navier-Stokes equation such that $u(t=0)=u_0$ and $u\in C([0,T];H^r(\mathbb{T}^d))\cap C^1([0,T];H^{r-2}(\mathbb{T}^d))$.
\end{theorem}
Based on this result, one can prove that $u$ is Sobolev regular i.e., that $u\in H^k(D\times [0,T])$ for some $k\in\mathbb{N}$, provided that $r$ is large enough \cite{deryck2021navierstokes}. 

\begin{corollary}\label{cor:NS-Hk}
If $k\in\mathbb{N}$ and $u_0\in H^r(\mathbb{T}^d)$ with $r>\frac{d}{2}+2k$ and $\div{u_0}=0$, then there exist $T>0$ and a classical solution $u$ to the Navier-Stokes equation such that $u\in H^{k}(\mathbb{T}^d\times [0,T])$, $\nabla p\in H^{k-1}(\mathbb{T}^d\times [0,T])$ and $u(t=0)=u_0$. 
\end{corollary}
Combining this regularity result with Theorem \ref{thm:approx-sobolev} then gives rise to the following approximation result for the Navier-Stokes equations \cite[Theorem 3.1]{deryck2021navierstokes}. 

\begin{theorem}\label{thm:pinn-approx-ns}
Let $n\geq 2$, $d,r,k\in\mathbb{N}$, with $k\geq 3$,
%, let $D\subset \mathbb{R}^d$  be a Lipschitz domain
and let $u_0\in H^r(\mathbb{T}^d)$ with $r>\frac{d}{2}+2k$ and $\div{u_0}=0$. Let $T>0$ be the time from Corollary \ref{cor:NS-Hk} such that the classical solution of $u$ exists on $\T^d\times [0,T]$. 
Then for every $N>5$, there exist tanh neural networks $\hu_j$, $1\leq j\leq d$, and $\widehat{p}$, each with two hidden layers, of widths $3\left\lceil\frac{k+n-2}{2}\right\rceil\binom{d+k-1}{d}+\lceil TN\rceil +dN$ and $3\left\lceil\frac{d+n}{2}\right\rceil \binom{2d+1}{d}\lceil TN \rceil N^{d}$, such that, 
    \begin{align}
    \begin{split}
        \ltwo{(\hu_j)_t + \hu\cdot \nabla \hu_j + (\nabla\widehat{p})_j - \nu \Delta \hu_j} &\leq C_1\ln^2(\beta N) N^{-k+2},
    \end{split}\label{eq:bound-pinn-res-1}\\
    \ltwo{\div{\hu}} &\leq C_2 \ln(\beta N) N^{-k+1},\label{eq:bound-pinn-res-2}\\
     \norm{(u_0)_j-\hu_j(t=0)}_{L^2(\mathbb{T}^d)} &\leq C_3 \ln(\beta N)N^{-k+1}, \label{eq:bound-pinn-res-3}
    \end{align}
for every $1\leq j\leq d$ and where the constants $\beta,C_1,C_2,C_3$ are explicitly defined in the proof and can depend on $k$, $d$, $T$, $u$ and $p$ but not on $N$. {\color{black}The weights of the networks can be bounded by $\bigO(N^\gamma \ln(N))$ where $\gamma = \max\{1,d(2+k^2+d)/n\}$.}
\end{theorem}
\end{example}

\subsection{Second general result}

Combining results as Theorem \ref{lem:acc-trig-pol} or Theorem \ref{thm:approx-sobolev} with Assumption \ref{ass:diff-operator} thus allows to prove that the generalization error $\cE_G$ can be made arbitrarily small, thereby answering question Q1. There is however still room for improvement as,
\begin{itemize}
    \item For many model classes there are primarily results available that only state that the error $\norm{u-u_\theta}_{L^q}$ can be made small. Is there a way to infer that if $\norm{u-u_\theta}_{L^q}$ is small then also $\norm{u-u_\theta}_{W^{k,q}}$ is small (under some assumptions)? 
    \item Both Theorem \ref{lem:acc-trig-pol} or Theorem \ref{thm:approx-sobolev} suffer from the \emph{curse of dimensionality (CoD)}, meaning that the number of parameters needed to reach an accuracy of $\epsilon$ scales exponentially with the input dimension $d$, namely as $\epsilon^{-(s-k)/d}$. Unless $s\approx k+d $ this will mean that an infeasibly large number of parameters is needed to guarantee a sufficiently small generalization error. Experimentally, however, one has observed for many PDEs that the CoD can be overcome. This raises the following question. Can we improve upon the approximation results stated in this section? 
    \item Both Theorem \ref{lem:acc-trig-pol} or Theorem \ref{thm:approx-sobolev} require that the true solution $u$ of the PDE is sufficiently (Sobolev) regular. Is there a way to still prove that the approximation error is small if $u$ is less regular, e.g. as for inviscid scalar conservation laws? 
\end{itemize}
As it turns out, the first two of the questions above can be answered in the same manner, which is the topic of the below. 

We first show that, under some assumptions, one can indeed answer question Q1 even if one only has access to an approximation result for $\norm{u-u_\theta}_{L^q}$ and not for $\norm{u-u_\theta}_{W^{k,q}}$. We do this by using finite difference (FD) approximations. 
Depending on whether forward, backward or central differences are used, a FD operator might not be defined on the whole domain $D$, e.g. for $f\in C([0,1])$ the (forward) operator $\Delta^+_h[f] := f(x+h)-f(x)$ is not well-defined for $x\in (1-h,1]$. This can be solved by resorting to piecewise-defined FD operators, e.g. a forward operator on $[0,0.5]$ and a backward operator on $(0.5,1]$. In a general domain $\Omega$ one can find a well-defined piecewise FD operator if $\Omega$ satisfies the following assumption, which is satisfied by many domains (e.g. rectangular or smooth domains). 
\begin{assumption}\label{ass:partition}
There exists a finite partition $\cP$ of $\Omega$ such that for all $P\in\cP$ there exists $\epsilon_P>0$ and $v_P\in B^1_\infty = \{x\in\R^{\dim(\Omega)}: \norm{x}_\infty\leq 1\}$ such that for all $x\in P$ it holds that $x+\epsilon_P(v_P+B^1_\infty) \subset \Omega$. 
\end{assumption}
Under this assumption we can prove an upper bound on the (strong) PDE residual of the model $u_\theta$. 

\begin{theorem}\label{thm:approx-f-to-df}
Let $q\in [1,\infty]$, $r, \ell\in\N$ with $\ell\leq r$ and $u, u_\theta \in C^{r}(\Omega)$. If Assumption \ref{ass:diff-operator} and Assumption \ref{ass:partition} hold then there exists a constant $C(r)>0$ such that for any $\bmalpha\in\N^d_0$ with $\ell := \norm{\bmalpha}_1$ it holds for all $h>0$ that
\begin{equation}
    \norm{\cL(u_\theta)}_{L^q} \leq C\left(\norm{u-u_\theta}_{L^q}h^{-\ell} + (|u| _{C^{r}}+|u_\theta|_{C^{r}})h^{r-\ell}\right). 
\end{equation}
\end{theorem}
\begin{proof}
    The proof follows from approximating any $D^\bmalpha (u_\theta-u)$ using a $r$-th order accurate finite-difference formula (which is possible thanks to Assumption \ref{ass:partition}), and combining this with Assumption \ref{ass:diff-operator}. See also \cite[Lemma B.1]{deryck2022generic}. 
\end{proof}

Theorem \ref{thm:approx-f-to-df} can show that $\norm{\cL(u_\theta)}_{L^q}$ is small if $\norm{u-u_\theta}_{L^q}$ is small and $|u_\theta|_{C^{r}}$ does not increase too much in terms of the model size. It must be stated that these assumptions are not necessarily trivially satisfied. Indeed, assume that $\Theta\subset \R^n$, that $\norm{u-u_\theta}_{L^q} \sim n^{-\alpha}$ and  $|u_\theta|_{C^{r}}\sim n^\beta$, and finally that $n \sim h^{-\gamma}$ for $\alpha,\beta,\gamma \geq 0$. In this case, the upper bound of Theorem \ref{thm:approx-f-to-df} is equivalent to
\begin{eqnarray}
    \norm{\cL(u_\theta)}_{L^q} \lesssim h^{\alpha\gamma} h^{-\ell} + h^{-\beta\gamma} h^{r-\ell}. 
\end{eqnarray}
To make sure that the right hand side is small for $h\to 0$ the inequality $\beta \ell < \alpha (r-\ell)$ should hold. Hence, either $\norm{u-u_\theta}_{L^q}$ should converge fast (large $\alpha$), $|u_\theta|_{C^{r}}$ should diverge slowly (small $\beta$) or not at all ($\beta=0$) or the true solution $u$ should be very smooth (large $r$). By way of example, we investigate what kind of bound is produced by Theorem \ref{thm:approx-f-to-df} if the Fourier basis is used. 

\begin{example}[Fourier basis]
    Theorem \ref{lem:acc-trig-pol} tells us that $\alpha = r/d$ and that $\beta=0$. Hence, the condition $\beta \ell < \alpha (r-\ell)$ is already satisfied. We now choose $\gamma =d $ such that the optimal rate is obtained, for which $h^{\alpha\gamma} h^{-\ell} = h^{r-\ell}$. The final convergence rate is hence $ N^{-(r-\ell)/d}$, in agreement with that of Theorem \ref{lem:acc-trig-pol}. 
\end{example}

\begin{remark}
Unlike in the previous example, Theorem \ref{thm:approx-f-to-df} is not expected to produce optimal convergence rates, particularly when loose upper bounds for $|u_\theta|_{C^{r}}$ are used. However, this is not a problem when trying to prove that a model class can overcome the curse of dimensionality in the approximation error, as is the topic of the next section. 
\end{remark}

Finally, as noted in the beginning of this section, all results in the section so far all assume that the solution $u$ of the PDE is sufficiently (Sobolev) regular. In the next example, we show that one can also prove approximation results, thereby answering Question 1, when $u$ has discontinuities. 

\begin{example}[Scalar conservation laws]
In Example \ref{ex:scl-weak-residual} a physics-informed loss based on the (weak) Kruzhkov entropy residual was introduced for scalar conservation laws (Example \ref{def:scl}). It consists of the term $\max_{\vartheta,c} \cR(u_\theta,\phi_\vartheta, c)$ together with the integral of the squared boundary residual $\rs[u_\theta]$ and the integral of the squared initial condition residual $\rt[u_\theta]$. Moreover, as solutions to scalar conservation laws might not be Sobolev regular, an approximation result cannot be directly proven based on Theorem \ref{thm:approx-sobolev}. Instead, in \cite{deryck2022wpinns} it was proven that the relevant counterpart to Assumption \ref{ass:diff-operator} is the following lemma. 
\begin{lemma}\label{thm:approx-general}
Let $p,q>1$ be such that $\frac{1}{p}+\frac{1}{q}=1$ or let $p=\infty$ and $q=1$. Let $u$ be the entropy solution of \eqref{eq:scl} and let $v\in L^q(D\times [0,T])$. Assume that $f$ has Lipschitz constant $L_f$. Then it holds for any $\varphi\in W_0^{1,\infty}(D\times [0,T])$ that
\begin{equation}
    \cR(v,\varphi, c) \leq (1+3L_f)\abs{\varphi}_{W^{1,p}(D\times [0,T])} \norm{u-v}_{L^q(D\times[0,T])}. 
\end{equation}
\end{lemma}
Hence, we need an approximation result for $\norm{u-u_\theta}_{L^q(D\times[0,T])}$ now rather than for a higher-order Sobolev norm of $u-u_\theta$. The following holds \cite[Lemma 3.3]{deryck2022wpinns}. 
\begin{lemma}\label{lem:bv-nn}
For every $u\in BV([0,1]\times [0,T])$ and $\epsilon>0$ there is a tanh neural network $\hu$ with two hidden layers and at most $O(\epsilon^{-2})$ neurons such that
\begin{equation}
    \norm{u-\hu}_{L^1([0,1]\times [0,T])}\leq \epsilon.
\end{equation}
\end{lemma}
\begin{proof}
Write $\Omega = [0,1]\times [0,T]$. For every $u\in BV(\Omega)$ and $\epsilon>0$ there exists a function $\Tilde{u}\in C^\infty(\Omega)\cap BV(\Omega)$ such that $\norm{u-\Tilde{u}}_{L^1(\Omega)}\lesssim \epsilon$ and $\norm{\nabla \Tilde{u}}_{L^1(\Omega)} \lesssim \norm{u}_{BV(\Omega)}+\epsilon$ \cite{bartels2012total}. 
Then use the approximation techniques of \cite{deryck2021approximation, deryck2021navierstokes} and the fact that $\norm{\Tilde{u}}_{W^{1,1}(\Omega)}$ can be uniformly bounded in $\epsilon$ to find the existence of a tanh neural network $\uhat$ with two hidden layers and at most $O(\epsilon^{-2})$ neurons that satisfies the mentioned error estimate. 
\end{proof}
If one additionally knows that $u$ is piecewise smooth, for instance as in the solutions of convex scalar conservation laws \cite{holden2015front}, then one can use the results of \cite{petersen2018optimal} to obtain the following result \cite[Lemma 3.4]{deryck2022wpinns}. 
\begin{lemma}\label{lem:pw-nn}
Let $m,n\in\N$, $1\leq q<\infty$ and let $u:[0,1]\times [0,T]\to \R$ be a function that is piecewise $C^m$ smooth and with a $C^n$ smooth discontinuity surface. Then there is a tanh neural network $\hu$ with at most three hidden layers and $\bigO(\epsilon^{-2/m}+\epsilon^{-q/n})$ neurons such that
\begin{equation}
    \norm{u-\hu}_{L^q([0,1]\times [0,T])}\leq \epsilon.
\end{equation}
\end{lemma}

Finally, \citet{deryck2022wpinns} found that one has to consider test functions $\varphi$ that might grow as $\abs{\varphi}_{W^{1,p}}\sim\epsilon^{-3(1+2(p-1)/p)}$. Consequently, we will need to use Lemma \ref{lem:bv-nn} with $\epsilon \leftarrow \epsilon^{4+6(p-1)/p}$, leading to the following corollary.
\begin{corollary}
\label{cor:er}
Assume the setting of Lemma \ref{lem:pw-nn}, assume that $m q > 2n$ and let $p\in[1,\infty]$ be such that $\frac{1}{p}+\frac{1}{q}=1$. There is a fixed depth tanh neural network $\hu$ with size $\bigO(\epsilon^{-(4q+6)n})$ such that
\begin{equation}
    \max_{c\in \cC}\sup_{\varphi\in\overline{\Phi}_\epsilon}\cR(\hu,\varphi, c)+  \lambda_{s}\norm{\rs[\hu]}^2_{L^2(\partial D\times [0,T])} + \lambda_{t}\norm{\rt[\hu]}^2_{L^2(D)} \leq \epsilon, 
\end{equation}
where $\overline{\Phi}_\epsilon = \{\varphi\: :\: \abs{\varphi}_{W^{1,p}}=\bigO(\epsilon^{-3(1+2(p-1)/p))}\}$.
\end{corollary}
Hence, we have answered Question 1 in an affirmative manner for scalar conservation laws. 
\end{example}
\begin{proof}
We find that $\uhat$ will need to have a size of $\bigO(\epsilon^{-(4q+6)/\beta})$ for it to hold that $\max_{c\in \cC}\sup_{\varphi\in\Phi_\epsilon}\cR(\uhat,\varphi, c) \leq \epsilon/3$, where we used that $q = p/(p-1)$. Since in the proof of Lemma \ref{lem:pw-nn} the network $\uhat$ is constructed as an approximation of piecewise Taylor polynomials, the spatial and temporal boundary residuals ($\rs$ and $\rt$) are automatically minimized as well, given that Taylor polynomials provide approximations in $C^0$-norm. 
\end{proof}

%\subsection{General result for neural networks}\label{sec:approx-nn}
\subsection{Neural networks overcome curse of dimensionality in approximation error}\label{sec:approx-cod}

We investigate whether (physics-informed) neural networks can overcome the curse of dimensionality (CoD) for the approximation error. Concretely, we want to prove the existence of a neural network with parameter $\hat{\theta}$ for which $\cE_G(\hat{\theta})<\epsilon$ without the size of the network growing exponentially in the input dimension $d$. We discuss two frameworks in which this can be done, both of which exploit the properties of the PDE. 
\begin{itemize}
    \item The first framework is tailored to time-dependent PDEs and considers initial conditions and PDE solutions that are Sobolev regular (or continuously differentiable). The crucial assumption is that one must be able to approximate the initial condition $u_0$ with a neural network without incurring the CoD, i.e. approximate it to accuracy $\epsilon$ with a network of size $\bigO(d^\alpha \epsilon^{-\beta})$ with $\alpha,\beta>0$ independent of $d$. This framework can also be used to prove results for (physics-informed) operator learning. 
    \item The second framework circumvents this assumption by considering PDE solutions that belong to a smaller space, namely the so-called \emph{Barron space}. It can be proven that functions in this space can be approximated without the curse of dimensionality. Hence, in this case the challenge lies in proving that the PDE solution is indeed a \emph{Barron function}. 
\end{itemize}

\subsubsection{Results based on Sobolev- and $C^k$-regularity}

We first direct our focus to the first framework. To start with, we give particular attention to the case where it is known that a neural network can efficiently approximate the solution to a time-dependent PDE at a fixed time. Such neural networks are usually obtained by emulating a classical numerical method. Examples include finite difference schemes, finite volume schemes, finite element methods, iterative methods and Monte Carlo methods, e.g. {\color{black} \cite{jentzen2018proof,opschoor2020deep,chen2021representation,marwah2021parametric}}.
In these cases Theorem \ref{thm:approx-f-to-df} can not be directly applied, as there is no upper bound for $\norm{u-u_\theta}_{L^q(D\times [0,T])}$. To allow for a further generalization to operator learning, we will write down the following assumption in terms of operators (see Section \ref{sec:forward}. 

More precisely, for $\epsilon>0$, we assume to have access to an operator {\color{black}$\cU^\epsilon:\cX\times [0,T]\to \cH$} that for any $t\in[0,T]$ maps any initial condition/parameter function $v\in \cX$ to a neural network $\cU^\epsilon(v,t)$ that approximates the PDE solution $\cG(v)(\cdot, t) = u(\cdot, t) {\color{black}\in L^q(D), q\in\{2,\infty\},}$ at time $t$, as specified below. Moreover, we will assume that we know how its size depends on the accuracy $\epsilon$. 

\begin{assumption}\label{ass:nn}
Let $q\in\{2,\infty\}$. For any $B,\epsilon>0$, $\ell \in\N$, $t\in [0,T]$ and any $v\in \cX$ with $\norm{v}_{C^\ell}\leq B$ there exist a neural network $\cU^\epsilon(v,t):D\to\R$ and a constant $C_{\epsilon,\ell}^B>0$ s.t.
\begin{equation}
    \norm{\cU^\epsilon(v,t)-\cG(v)(\cdot,t)}_{L^q(D)} \leq \epsilon \quad \text{and} \quad \max_{t\in [0,T]} \norm{\cU^\epsilon(v,t)}_{W^{\ell,q}(D)} \leq C_{\epsilon,\ell}^B.
\end{equation}
\end{assumption}
For vanilla neural networks and PINNs one can however always set $\cX := \{v\}$, with with $v:=u_0$ or $v:=a$ and $\cG(v) := u$ in Assumption \ref{ass:nn} above. 

Under Assumption \ref{ass:nn}, the existence of space-time neural networks that minimize the generalization error $\cE_G$ (i.e. the physics-informed loss) can be proven \cite{deryck2022generic}. 
\begin{theorem}\label{thm:nn-to-pinn}
Let $s,r\in\N$, let $u\in C^{(s,r)}([0,T]\times D)$ be the solution of the PDE \eqref{eq:pde} and let Assumption \ref{ass:nn} be satisfied. There exists a constant $C(s,r)>0$ such that for every $M\in\N$ and $\epsilon, h>0$ there exists a tanh neural network $u_\theta:[0,T]\times D\to \R$ for which it holds that,
\begin{equation}
\label{eq:pinnerr1}
    \norm{u_\theta-u}_{L^q([0,T]\times D)} \leq C(\norm{u}_{C^{(s,0)}}M^{-s}+\epsilon). 
\end{equation}
and if additionally Assumption \ref{ass:diff-operator} and Assumption \ref{ass:partition} hold then,
\begin{equation}\label{eq:pinn-acc}
\begin{split}
     &\norm{\cL(u_\theta)}_{L^2([0,T]\times D)} + \norm{u_\theta-u}_{L^2(\partial([0,T]\times D))} \\
     &\qquad \leq C\cdot \poly(d)\cdot \ln^k(M)(\norm{u}_{C^{(s,\ell)}}M^{k-s}+M^{2k}(\epsilon h^{-\ell}+C^B_{\epsilon,\ell}h^{r-\ell})). 
\end{split}
\end{equation}
Moreover, the depth is given by $\depth(u_\theta)\leq C\cdot\sup_{t\in[0,T]}\depth(\cU^\epsilon(u(t)))$ and the width by $\width(u_\theta) \leq CM\cdot \sup_{t\in[0,T]}\width(\cU^\epsilon(u(t)))$. 
\end{theorem}
\begin{proof}
We only provide a sketch of the full proof \cite[Theorem 3.5]{deryck2022generic}. The main idea is to divide $[0,T]$ into $M$ uniform subintervals and construct a neural network that approximates a Taylor approximation in time of $u$ in each subinterval. In the obtained formula, we approximate the monomials and multiplications by neural networks and approximate the derivatives of $u$ by finite differences and use the accuracy of finite difference formulas to find an error estimate in $C^k([0,T], L^q(D))$-norm. We use again finite difference operators to prove that spatial derivatives of $u$ are accurately approximated as well. The neural network will also approximately satisfy the initial/boundary conditions as $\Vert u_\theta-u\Vert_{L^2(\partial([0,T]\times D))} \lesssim C \poly(d) \Vert u_\theta-u\Vert_{H^1([0,T]\times D)}$, which follows from a Sobolev trace inequality. 
\end{proof}

As a next step, we use Assumption \ref{ass:nn} to prove estimates deep operator learning. Given the connection between DeepONets and FNOs \cite[Theorem 36]{kovachki2021universal}, we focus on DeepONets in the following. 
In order to prove this error estimate, we need to assume that the operator $\cU^\epsilon$ from Assumption \ref{ass:nn} is stable with respect to its input function, as specified in Assumption \ref{ass:stability} below. Moreover, we will take the $d$-dimensional torus as domain $D=\T^d = [0,2\pi)^d$ and assume periodic boundary conditions for simplicity in what follows. This is not a restriction, as for every Lipschitz subset of $\T^d$ there exists a (linear and continuous) $\T^d$-periodic extension operator of which also the derivatives are $\T^d$-periodic \cite[Lemma 41]{kovachki2021universal}. 
\begin{assumption}\label{ass:stability}
Assumption \ref{ass:nn} is satisfied and let $p\in\{2,\infty\}$. For every $\epsilon>0$ there exists a constant $C^\epsilon_{\mathrm{stab}}>0$ such that for all $v,v'\in \cX$ it holds that,
\begin{equation}
    \norm{\cU^\epsilon(v,T)-\cU^\epsilon(v',T)}_{L^2}\leq C^\epsilon_{\mathrm{stab}} \norm{v-v'}_{L^p}. 
\end{equation}
\end{assumption}
In this setting, we prove a generic approximation result for DeepONets \cite[Theorem 3.10]{deryck2022generic}. 

\begin{theorem}\label{thm:nn-to-pido}
Let $s,r\in\N$, $T>0$,  $\cA\subset C^r(\T^d)$ and let $\cG: \cA \to C^{(s,r)}([0,T]\times \T^d)$ be an operator that maps a function $u_0$ to the solution $u$ of the PDE \eqref{eq:pde} with initial condition $u_0$, let Assumption \ref{ass:diff-operator}, Assumption \ref{ass:partition} and Assumption \ref{ass:stability} be satisfied {\color{black}and let $p^*\in\{2,\infty\}\setminus\{p\}$}. There exists a constant $C>0$ such that for every $Z,N,M\in\N$, $\epsilon, \rho>0$ there is an DeepONet $\cG_\theta: \cA \to L^2([0,T]\times\T^d)$ with $Z^d$ sensors with accuracy, 
\begin{align}
    \norm{\cL(\cG_\theta(v))}_{L^2([0,T]\times\T^d)} &\leq CM^{k+\rho}\big[\norm{u}_{C^{(s,\ell)}}M^{-s}\\&\quad+M^{s-1}N^\ell(\epsilon+C^\epsilon_{\mathrm{stab}}Z^{-r+d/p^*}+C^{CB}_{\epsilon,r}N^{-r})\big], 
\end{align}
for all $v$. Moreover, it holds that, 
\begin{equation}
    \begin{array}{ll}
    \width(\branch)= \bigO(M(Z^d+N^d\width(\cU^\epsilon))), & \depth(\branch) =\depth(\cU^\epsilon), \\ 
    \width(\trunk) = \bigO(MN^d(N+\ln(N))), & \depth(\trunk) = 3, 
    \end{array}
\end{equation}
where $\width(\cU^\epsilon) = \sup_{u_0\in \cA} \width(\cU^\epsilon)(u_0))$ and similarly for $\depth(\cU^\epsilon)$. 
\end{theorem}

Finally, we show how the results of this section can be applied to high-dimensional PDEs, for which it is not possible to obtain efficient approximation results using standard neural network approximation theory (cf. Theorem \ref{thm:approx-sobolev}) as they will lead to convergence rates that suffer from the \emph{curse of dimensionality (CoD)}, meaning that the neural network size scales exponentially in the input dimension. In literature, one has shown for some PDEs that their solution at a fixed time can be approximated to accuracy $\epsilon>0$ with a network that has size $\bigO(\poly(d)\epsilon^{-\beta})$, with $\beta>0$ independent of $d$, and therefore \emph{overcomes the CoD}. 
%We show how our results from Section \ref{sec:taylor} can be used to obtain error bounds for space-time neural networks and PINNs, by proving that PINNs overcome the CoD for linear Kolmogorov PDEs (known result) and nonlinear parabolic PDEs (new result). 

\begin{example}[Linear Kolmogorov PDEs]\label{ex:kolmogorov-approx}
Linear Kolmogorov PDEs are a class of linear time-dependent PDEs, including the heat equation and the Black-Scholes equation, of the following form: Let $s,r\in\N$, $u_0\in C^2_0(\R^d)$ and let $u\in C^{(s,r)}([0,T]\times \R^d)$ be the solution of
\begin{equation}\label{eq:kolmogorov-pde}
    \cL[u] = \partial_t u - \frac{1}{2}\mathrm{Tr}(\sigma(x)\sigma(x)^\top \Delta_x[u]) - \mu^\top \nabla_x[u] = 0, \quad u(0,x)=u_0(x)
\end{equation}
for all $(x,t)\in D \times [0,T]$, where $\sigma:\mathbb{R}^d\to \mathbb{R}^{d\times d}$ and $\mu:\mathbb{R}^d\to \mathbb{R}^d$ are affine functions. We make the assumption that  $\norm{u}_{C^{(s,2)}}$ grows at most polynomially in $d$ and that for every $\epsilon>0$, there is a neural network $\uhat_0$ of width $\bigO(\poly(d)\epsilon^{-\beta})$ such that $\norm{u_0-\uhat_0}_{L^\infty(\R^d)}<\epsilon$. 

In this setting, the authors of \cite{grohs2018proof, berner2020analysis, jentzen2018proof} construct a neural network that approximates $u(T)$ and overcomes the CoD by emulating Monte-Carlo methods based on the Feynman-Kac formula. In \cite{deryck2021pinn} one has proven that PINNs overcome the CoD as well, in the sense that the network size grows as $\bigO(\poly(d\rho_d)\epsilon^{-\beta})$, where $\rho_d$ is a PDE-dependent constant that for a subclass of Kolmogorov PDEs scales as $\rho_d=\poly(d)$, such that the CoD is fully overcome. We demonstrate that the generic bounds of this section (Theorem \ref{thm:nn-to-pinn}) can be used to provide a (much shorter) proof for this result \cite[Theorem 4.2]{deryck2022generic}.

\begin{theorem}\label{thm:kolmogorov}
For every $\sigma,\epsilon>0$ and $d\in\N$, there is a tanh neural network $u_\theta$ of depth $\bigO(\mathrm{depth}(\uhat_0))$ and width $\bigO(\poly(d\rho_d)\epsilon^{-(2+\beta)\frac{r+\sigma}{r-2}\frac{s+1}{s-1}-\frac{1+\sigma}{s-1}})$ such that,
\begin{equation}
    \norm{\cL(u_\theta)}_{L^2([0,T]\times [0,1]^d)} + \norm{u_\theta-u}_{L^2(\partial([0,T]\times [0,1]^d))} \leq \epsilon.
\end{equation}
\end{theorem}
\end{example}

\begin{example}[Nonlinear parabolic PDEs]
Next, we consider nonlinear parabolic PDEs as in \eqref{eq:AC}, which  typically arise in the context of nonlinear diffusion-reaction equations that describe the change in space and time of some quantities, such as in the well-known \emph{Allen-Cahn equation} \cite{allen1979microscopic}. 
Let $s,r\in\N$ and for $u_0\in \cX\subset C^r(\T^d)$ let $u\in  C^{(s,r)}([0,T]\times \T^d)$ be the solution of
\begin{equation}\label{eq:AC}
   {\color{black}\cL (u)(x,t) =  \partial_t u(t,x) - \Delta_x u(t,x) - F(u(t,x))=0, }\qquad u(0,x)=u_0(x), 
\end{equation}
for all $(t,x)\in [0,T] \times D$, with period boundary conditions, where $F:\R\to\R$ is a polynomial. As in Example \ref{ex:kolmogorov-approx}, we assume that $\norm{u}_{C^{(s,2)}}$ grows at most polynomially in $d$ and that for every $\epsilon>0$, there is a neural network $\uhat_0$ of width $\bigO(\poly(d)\epsilon^{-\beta})$ such that $\norm{u_0-\uhat_0}_{L^\infty(\T^d)}<\epsilon$. %Let $\mu$, resp. $\mu^*$, be the normalized Lebesgue measure on $[0,T]\times \T^d$, resp. $\partial([0,T]\times \T^d)$.  

In \cite{hutzenthaler2020proof} the authors have proven that ReLU neural networks overcome the CoD in the approximation of $u(T)$. Using Theorem \ref{thm:nn-to-pinn} we can now prove that PINNs overcome the CoD for nonlinear parabolic PDEs \cite[Theorem 4.4]{deryck2022generic}. 
\begin{theorem}\label{thm:AC-pinn}
For every $\sigma, \epsilon>0$ and $d\in\N$ there is a tanh neural network $u_\theta$ of depth $\bigO(\mathrm{depth}(\uhat_0)+\poly(d)\ln(1/\epsilon))$ and width $\bigO(\poly(d)\epsilon^{-(2+\beta)\frac{r+\sigma}{r-2}\frac{s+1}{s-1}-\frac{1+\sigma}{s-1}})$ such that,
\begin{equation}
     \norm{\cL(u_\theta)}_{L^2([0,T]\times\T^d)} + \norm{u-u_\theta}_{L^2(\partial([0,T]\times\T^d))} \leq  \epsilon. 
\end{equation}
\end{theorem}

Similarly, one can use the results of this section to obtain estimates for (physics-informed) DeepONets for nonlinear parabolic PDEs \eqref{eq:AC} such as the Allen-Cahn equation. In particular, a dimension-independent convergence rate can be obtained if the solution is smooth enough, which improves upon the result of \cite{LMK1}, which incurred the CoD. For simplicity, we present results for $C^{(2,r)}$ functions, rather than $C^{(s,r)}$ functions, as we found that assuming more regularity did not necessarily further improve the convergence rate \cite[Theorem 4.5]{deryck2022generic}. 

\begin{theorem}\label{thm:AC-operator}
Let $\cG: \cX\to C^{r}(\T^d):u_0\mapsto u(T)$ and $\cG^*: \cX \to C^{(2,r)}([0,T]\times \T^d):u_0\mapsto u$. For every $\sigma, \epsilon>0$, there exists a DeepONet $\cG_\theta^*$ such that
\begin{equation}
   \norm{\cL(\cG^*_\theta)}_{L^2([0,T]\times\T^d\times \cX)} \leq \epsilon. 
\end{equation}
Moreover, for $\cG^*_\theta$ we have $\bigO(\epsilon^{-\frac{(3+\sigma)d}{r-2}})$ sensors and,
\begin{equation}
    \begin{array}{ll}
    \width(\branch)= \bigO( \epsilon^{-1-\frac{(3+\sigma)(d+r(2+\beta))}{r-2}}), & \depth(\branch) =\bigO( \ln(1/\epsilon)), \\ 
    \width(\trunk) = \bigO(\epsilon^{-1-\frac{(3+\sigma)(d+1)}{r-2}}), & \depth(\trunk) = 3.
    \end{array}
\end{equation}
\end{theorem}
\end{example}

\subsubsection{Results based on Barron regularity}\label{sec:cod-barron}

A second framework that can be used to prove that neural networks can overcome the curse of dimensionality is that of Barron spaces. These spaces are named after the seminal work of \citet{barron1993universal}, where it is shown that a function $f$ with Fourier transform $\hat{f}$ can be approximated to accuracy $\tfrac{1}{\sqrt{m}}$ by a shallow neural network with $m$ neurons, as long as 
\begin{eqnarray}\label{eq:barron-condition}
    \int_{\R^d} \vert \hat{f}(\xi)\vert \cdot \abs{\xi} d\xi <\infty. 
\end{eqnarray}
Later, the space of functions satisfying condition \eqref{eq:barron-condition} was named a \emph{Barron space} and its definition was generalized in many different ways. One notable generalization is that of \emph{spectral Barron spaces}. 

\begin{definition}\label{def:spectral-barron}
The spectral Barron space with index $s$, denoted $\cB^s(\R^d)$, is defined as the collection of functions $f:\R^d\to\R$ for which the spectral Barron norm is finite, 
\begin{eqnarray}
    \norm{f}_{\cB^s(\R^d)} := \int_{\R^d} \vert \hat{f}(\xi)\vert \cdot (1+\abs{\xi}^2)^{\frac{s}{2}} d\xi <\infty. 
\end{eqnarray}
\end{definition}
First, notice how the initial condition \eqref{eq:barron-condition} of \citet{barron1993universal} corresponds to the spectral Barron space $\cB^1(\R^d)$. Second, from this definition the difference between Barron spaces and Sobolev spaces become apparent: whereas the spectral Barron norm is the $L^1$-norm of $\vert \hat{f}(\xi)\vert \cdot (1+\abs{\xi}^2)^{\frac{s}{2}}$, the Sobolev $H^s(\R^d)$-norm is defined as the $L^2$-norm of that same quantity. Finally, it follows that $\cB^r(\R^d)\subset \cB^s(\R^d)$ for $r\geq s$ and that $\cB^0(\R^d)\subset L^\infty(\R^d)$, see e.g. \cite{chen2023regularity}. 

An example of an approximation result for functions in spectral Barron spaces can be found in \cite{lu2021priori} for functions with the softplus activation function. 
\begin{theorem}\label{thm:spectral-barron-approx}
    Let $\Omega = [0,1]^d$. For every $u\in \cB^2(\Omega)$, there exists a shallow neural network $\hat{u}$ with softplus activation function ($\sigma(x) = \ln(1+\exp(x))$, with width at most $m$, such that
    \begin{eqnarray}
        \norm{u-\hat{u}}_{H^1(\Omega)}\leq \frac{\norm{u}_{\cB^2(\Omega)}(6\log(m)+30)}{\sqrt{m}}. 
    \end{eqnarray}
    Moreover, an exact bound on the weights is given in \citet[Theorem 2.2]{lu2021priori}.
\end{theorem}
There are two key questions that must be answered, before one can use this theorem to argue that neural networks can overcome the CoD in the approximation of PDE solutions: 
\begin{itemize}
    \item Under which conditions does it hold that $u\in \cB^2(\Omega)$? 
    \item How does $\norm{u}_{\cB^2(\Omega)}$ depend on $d$? 
\end{itemize}
Answering these questions requires building a regularity theory for Barron spaces, which can be challenging as they are not Hilbert spaces such as the Sobolev spaces $H^s$. An important contribution for high-dimensional elliptic equations has been made in \citet{lu2021priori}. 

\begin{example}[Poisson's equation]\label{ex:barron-poisson}
Let $s\geq 0$, let $f\in \cB^{s+2}(\Omega)$ satisfy $\int_\Omega f(x)dx = 0$ and let $u$ be the unique solution to Poisson's equation $ - \Delta u = f$ with zero Neumann boundary conditions. Then it holds that $u\in \cB^s(\Omega)$ and that $\norm{u}_{\cB^s(\Omega)}\leq d \norm{f}_{\cB^{s+2}(\Omega)}$ \cite[Theorem 2.6]{lu2021priori}.
\end{example}

\begin{example}[Static Schrödinger equation]\label{ex:barron-schr}
Let $s\geq 0$, let $V\in \cB^{s+2}(\Omega)$ satisfy $V(x) \geq V_{\mathrm{min}} > 0$  for all $x\in \Omega$ and let $u$ be the unique solution to static Schrödinger's equation $ - \Delta u +Vu= f$ with zero Neumann boundary conditions. Then it holds that $u\in \cB^s(\Omega)$ and that $\norm{u}_{\cB^s(\Omega)}\leq C \norm{f}_{\cB^{s+2}(\Omega)}$ \cite[Theorem 2.3]{chen2023regularity}.
\end{example}

Similar results for a slightly different definition of Barron spaces \cite{ma2022barron} can be found in e.g. \cite{weinan2022some}. The authors of these works also show that the Barron space is the \emph{right} space for two-layer neural network models in the sense that optimal direct and inverse approximation theorems hold. Moreover, in \cite{wojtowytsch2020banach, wojtowytsch2022representation} they explore the connection between Barron and Sobolev spaces and provide examples of functions that are and functions that (sometimes surprisingly) are not Barron functions. 

Another slightly different definition can be found in \cite{chen2021representation}, where a Barron function $g$ is defined as a function that can be written as an infinite-width neural network, 
\begin{eqnarray}\label{eq:g-barron}
    g=\int a\sigma(w^\top x+b)\rho(da,dw,db),
\end{eqnarray}
where $\rho$ is a probability distribution on the (now infinite) parameter space $\Theta$. In particular, the neural network $\tfrac{1}{k}\sum_{i=1}^k a_i \sigma(w_i^\top x+b_i)$ is a Barron function. 

\begin{definition}\label{def:barron}
Fix  $\Omega\subset\R^d$ and $R\in[0,+\infty]$. For a function $g$ as in \eqref{eq:g-barron}, we define its Barron norm on $\Omega$ with index $p\in [1,+\infty]$ and support radius $R$ by
\begin{equation}
\begin{split}
    \norm{g}_{\cB^p_R(\Omega)} & =\inf_{\rho\in \cA_g}\biggl\{\biggl(\int|a|^p\rho(da,dw,db)\biggr)^{1/p}\biggr\},
\end{split}
\end{equation}
where $\cA_g$ is the set of probability measures $\rho$ supported on $\R\times \overline{B}^d_R\times\R$, where $\overline{B}_R^d=\{x\in\R^d:\norm{x}\leq R\}$, such that $g = \int a\sigma(w^\top x+b)\rho(da,dw,db)$.
The corresponding Barron space is then defined as
\begin{equation}
    \cB^p_R(\Omega)=\left\{g:\norm{g}_{\cB^p_R(\Omega)}<\infty\right\}.
\end{equation}
\end{definition}
A number of interesting facts have been proven about this space, such has,
\begin{itemize}
    \item The Barron norms and spaces introduced in Definition \ref{def:barron} are independent of $p$ i.e.,  $\norm{g}_{\cB^p_R(\Omega)} = \norm{g}_{\cB^q_R(\Omega)}$ and hence $\cB^p_R(\Omega) =  \cB^q_R(\Omega)$ for any $p,q\in [1,+\infty]$. See \cite[Proposition 2.4]{chen2021representation} and also \cite[Proposition 1]{ma2022barron}. 
    \item Under some conditions on the activation function $\sigma$, the Barron space $\cB^p_R(\Omega)$ is an algebra \cite[Lemma 3.3]{chen2021representation}. 
    \item Neural networks can approximate Barron functions in Sobolev norm without incurring the curse of dimensionality. 
\end{itemize}

We demonstrate the final point by slightly adapting \cite[Theorem 2.5]{chen2021representation}. 

\begin{theorem}
Let $\sigma$ be a smooth activation function, $\ell\in\N$, $R\geq 1$, let $\Omega\subset \R^d$ be open and let $f\in B^1_R(\Omega)$.  Let $\mu$ a probability measure on $\Omega$ and set $C_1 := \max_{m\leq \ell}\norm{\sigma^{(m)}}_\infty<\infty $. For any $k \in\ N$ there exist $\{(a_i,w_i, b_i)\}_{i=1}^k$ such that,
\begin{equation}
    \norm{\frac{1}{k}\sum_{i=1}^k a_i \sigma(w_i^Tx+b_i)-f(x)}_{H^\ell_\mu(\Omega)}\leq \frac{2C_1 (R\sqrt{ed})^\ell \norm{f}_{B^1_R(\Omega)}}{\sqrt{k}}
\end{equation}
\end{theorem}
\begin{proof}
Since $f\in B^1_R(\Omega)$, there must exist a probability measure $\rho$ supported on $\R\times \overline{B}^d_R\times\R$ such that
\begin{equation}
    f(x) = \int a\sigma(w^T x+b)\rho(da, dw, db), \quad \int \abs{a}^2 d\rho \leq \frac{4}{\sqrt{\pi}} \norm{f}_{B^2(\Omega)},  
\end{equation}
where $\tfrac{4}{\sqrt{\pi}}$ is a constant that is strictly larger than 1, which will be convenient later on. 
Define the error,
\begin{equation}
    \cE_k = \frac{1}{k}\sum_{i=1}^k a_i \sigma(w_i^Tx+b_i)-f(x).
\end{equation}
We calculate that for $\ell=\abs{\alpha}_1$,
\begin{eqnarray}
&&\E{\norm{D^\alpha \cE_k}_{L^2(\Omega_0)}^2}\\ &=& \iint_{\Omega_0} \left(\frac{1}{k}\sum_{i=1}^k a_i \sigma^l(w_i^T x+b_i)\prod_{j=1}^d w_{ij}^{\alpha_j}-D^\alpha f(x) \right)^2d\mu d\rho^k\\
&=& \frac{1}{k} \iint_{\Omega_0} \left(a \sigma^l(w^T x+b)\prod_{j=1}^d w_{j}^{\alpha_j}-D^\alpha f(x) \right)^2d\mu d\rho\\
&=&  \frac{1}{k} \int_{\Omega_0} \mathrm{Var}\left(a \sigma^l(w^T x+b)\prod_{j=1}^d w_{j}^{\alpha_j}\right)d\mu \\
&\leq& \frac{1}{k} \int_{\Omega_0} \E{\left(a \sigma^l(w^T x+b)\prod_{j=1}^d w_{j}^{\alpha_j}\right)^2}d\mu \\
&\leq& \frac{(C_1R^l)^2}{k} \E{\abs{a}^2} \leq \frac{4(C_1R^l)^2\norm{f}^2_{B^2(\Omega)}}{k\sqrt{\pi}}.
\end{eqnarray}
Then we find that using \cite[Lemma 2.1]{deryck2021approximation} and the previous inequality that,
\begin{equation}
     \E{\norm{\cE_k}_{H^\ell(\Omega)}^2} \leq \sqrt{\pi}(ed)^\ell \max_{\substack{\alpha\in\N_0^d, \\ \abs{\alpha}_1\leq \ell}}\E{\norm{D^\alpha \cE_k}^2} \leq \frac{4C_1^2(R^2ed)^\ell\norm{f}^2_{B^2(\Omega)}}{k}.
\end{equation}
We can then conclude by using the fact that for a random variably $Y$ that satisfies $\E{\abs{Y}}\leq \epsilon$ for some $\epsilon>0$, it must hold that $\mathbb{P}(\abs{Y}\leq \epsilon)>0$ \cite[Proposition 3.3]{grohs2018proof}.
\end{proof}

If one knows how $\norm{f}_{B^2(\Omega)}$ scales with $d$, then one can prove the existence of neural networks that overcome the CoD for physics-informed loss functions, under Assumption \ref{ass:diff-operator} and by following the approach of Section \ref{sec:approx-general}. More related works can be found in \cite{bach2017breaking, hong2021priori} and references therein. 

\section{Stability}\label{sec:stability}

Next, we investigate whether a small physics-informed loss implies a small total error, often the $L^2$-error, as formulated in the following question. 
\begin{question}[Q2]
    \emph{Given that a model $u_\theta$ has a small generalization error $\cE_G({\theta})$, will the corresponding total error $\Eg({\theta})$ be small as well?}
\end{question}
In the general results we will focus on physics-informed loss functions based on the strong (classical) formulation of the PDE, but we also give some examples when the loss is based on the weak solution (Example \ref{ex:scl-stab}) or the variational solution (Example \ref{ex:poisson-stab}). We first look at stability results for forward problems (Section \ref{sec:stab-forward}) and afterwards for inverse problems (Section \ref{sec:stab-inverse}). 

\subsection{Stability for forward problems}\label{sec:stab-forward}

We investigate whether a small PDE residual implies that the total error \eqref{eq:pinn-error} will be small as well (Question Q2). 
Such a stability bound can be formulated as the requirement that for any $\bu,\bv \in X^*$, the differential operator $\df$ satisfies
\begin{equation}\label{eq:assm2}
\|\bu - \bv\|_{X} \leq C_{\text{PDE}} \|\df(\bu) - \df(\bv)\|_{Y},
\end{equation}
where the constant $C_{\text{PDE}} > 0$ is allowed to depend on $\|\bu\|_{X^*}$ and $\|\bv\|_{X^*}$.

As a first example of a PDE with solutions satisfying \eqref{eq:assm2}, we consider a linear differential operator $\df: X \to Y$, i.e. $\df(\alpha \bu + \beta \bv) = \alpha \df(\bu) + \beta \df(\bv)$, for any $\alpha,\beta \in \R$. For simplicity, let $X^{\ast} = X$ and $Y^{\ast} = Y$. By the assumptions on the existence and uniqueness of the underlying linear PDE \eqref{eq:pde}, there exists an \emph{inverse} operator $\df^{-1}: Y \to X$. Note that the assumption \eqref{eq:assm2} is satisfied if the inverse is bounded i.e, $\|\df^{-1}\|_{\mathrm{op}} \leq C < +\infty$, with respect to the natural norm on linear operators from $Y$ to $X$. Thus, the assumption \eqref{eq:assm2} on stability boils down to the boundedness of the inverse operator for linear PDEs. Many well-known linear PDEs possess such bounded inverses \cite{DL1}. 

As a second example, we will consider a nonlinear PDE \eqref{eq:pde}, but with a well-defined linearization i.e, there exists an operator $\overline{\df}: X^{\ast} \mapsto Y^{\ast}$, such that 
\begin{equation}
    \label{eq:lin}
    \df(\bu) - \df(\bv) = \overline{\df}_{(\bu,\bv)}\left(\bu - \bv\right), \quad \forall \bu,\bv \in X^{\ast}.
\end{equation}
Again for simplicity, we will assume that $X^{\ast} = X$ and $Y^{\ast} = Y$. We further assume that the inverse of $\overline{\df}$ exists and is bounded in the following manner,
\begin{equation}
    \label{eq:lin1}
    \|\overline{\df}_{(\bu,\bv)}^{\: -1}\|_{\mathrm{op}} \leq C\left(\|\bu\|_{X},\|\bv\|_{X}  \right) < +\infty, \quad \forall \bu,\bv \in X,
    \end{equation}
with the norm of $\overline{\df}^{-1}$ being an operator norm, induced by linear operators from $Y$ to $X$. Then a straightforward calculation shows that \eqref{eq:lin1} suffices to establish the stability bound \eqref{eq:assm2}. 

We summarize our findings with the following informal theorem. 

\begin{theorem}
    Let $\cL$ be a linear operator or a nonlinear operator with linearization as in \eqref{eq:lin1}. If $\cL$ has a bounded inverse operator, then equation \eqref{eq:assm2} will hold i.e., a small physics-informed loss will imply a small total error. 
\end{theorem}

We give examples of the above theorem, with explicit stability constants, for the semilinear heat equation (strong formulation), Navier-Stokes equation (strong formulation), scalar conservation laws (strong and weak formulation) and the Poisson equation (variational formulation). Further examples can be found in \citep{MM1, lu2021priori}. 

\begin{example}[Semilinear heat equation]
We address the {stability} (Q2) of the semilinear heat equation (see Example \ref{def:semilinear-heat}). The following theorem \cite[Theorem 3.1]{MM1} ensures that one can indeed bound the total error $\cE(\theta)$ in terms of the generalization error $\cE_G(\theta)$. A generalization to linear Kolmogorov equations (including the Black-Scholes model) can be found in \cite[Theorem 3.7]{deryck2021pinn}.

\begin{theorem}
\label{thm:stability-heat}
Let $u \in C^k(\bar{D} \times [0,T])$ be the unique classical solution of the semilinear parabolic equation \eqref{eq:heat} with the source $f$ satisfying \eqref{eq:assf} and let $v\in C^2(\bar{D} \times [0,T])$. Then the total error \eqref{eq:pinn-error} can be estimated as, 
\begin{eqnarray}
    \label{eq:hegenb0}
    \norm{u-v}_{L^2(D\times [0,T])}^2
     &\leq& C_1 \big[\norm{\rpde[v]}_{L^2(D\times [0,T])}^2 + \norm{\rt[v]}_{L^2(D)}^2 \\&&+ C_2  \norm{\rs[v]}_{L^2(\partial D\times [0,T])}  \big],
\end{eqnarray}
with constants given by,
\begin{equation}
    \label{eq:hct}
    \begin{aligned}
        C_1 &= \sqrt{T + (1+2C_f)T^2e^{(1+2C_f)T}}, \\
        C_2 &= \sqrt{T^{\frac{1}{2}}|\partial D|^{\frac{1}{2}}\left(\|u\|_{C^1([0,T] \times \partial D)} + \|v\|_{C^1([0,T] \times \partial D)}\right)}. 
\end{aligned}
\end{equation}
\end{theorem}
\begin{proof}
First, note that for $w = v-u$ it holds that, 
\begin{eqnarray}
    \partial_t w = \Delta w + f(v)-f(u)+\rpde[v], \quad w(x,0) = \rt[v](x), \quad v = \rs[v]. 
\end{eqnarray}
Multiplying this PDE with $w$, integrating over $D$ and performing integration by parts then gives rise to the inequality, 
\begin{eqnarray}
    \frac{1}{2}\frac{d}{dt}\int_D \abs{w}^2 &\leq& \int_{\partial D} \rs[v] \nabla w\cdot \hat{n}_D + \int_D w (f(v)-f(u)+\rpde[v]) \\
    &\lesssim& \left(\int_{\partial D} \abs{\rs[v]}^2\right)^{1/2} + \int_D \abs{w}^2 + \int_D \abs{\rpde[v]}^2. 
\end{eqnarray}
Integrating the above inequality over $[0,\tau]\subset [0,T]$, applying Grönwall's inequality and then integrating once again over time yields the claimed result. The full version can be found in \cite[Theorem 3.1]{MM1}. 
\end{proof}
\end{example}

\begin{example}[Radiative transfer equation]
The study of radiative transfer is of vital importance in many fields of science and engineering including astrophysics, climate dynamics, meteorology, nuclear engineering and medical imaging \cite{Modbook}. The fundamental equation describing radiative transfer is a \emph{linear partial integro-differential equation}, termed as the \emph{radiative transfer equation}. Under the assumption of a static underlying medium, it has the following form \cite{Modbook}, 
\begin{equation}
\label{eq:genRTE}
\begin{aligned}
  \frac{1}{c}u_t +   \omega\cdot\nabla_x u  + ku + \sigma\Bigg(u - \frac{1}{s_d}\int\limits_{\Lambda}\int\limits_{S} \Phi(\omega, \omega^{\prime}, \nu, \nu^{\prime})u(t,x,\omega^{\prime}, \nu^{\prime}) d\omega^{\prime}d\nu^{\prime} \Bigg)= & f,
\end{aligned}
\end{equation}
with time variable $t \in [0,T]$, space variable $x \in D \subset \R^d$ (and $D_T = [0,T] \times D$), \emph{angle} $\omega \in S = {\mathbb S}^{d-1}$ i.e. the $d$-dimensional sphere and \emph{frequency} (or group energy) $\nu \in \Lambda \subset \R$. The constants in \eqref{eq:genRTE} are the speed of light $c$ and the surface area $s_d$ of the $d$-dimensional unit sphere. The unknown of interest in \eqref{eq:genRTE} is the so-called \emph{radiative intensity} $u: D_T \times S \times \Lambda \mapsto \R$, while $k = k(x,\nu): D \times \Lambda \mapsto \R_+$ is the \emph{absorption coefficient} and $\sigma =  \sigma(x,\nu): D \times \Lambda \mapsto \R_+$ is the \emph{scattering coefficient}. The integral term in \eqref{eq:genRTE} involves the so-called \emph{scattering kernel} $\Phi: S\times S \times \Lambda \times \Lambda \mapsto \R$, which is normalized as $\int_{S\times \Lambda} \Phi(\cdot 
,~\omega^{\prime},~\cdot,~\nu^{\prime}) d\omega^{\prime} d \nu^{\prime} = 1$, in order to account for the conservation of photons during scattering. The dynamics of radiative transfer are driven by a source (emission) term $f = f(x,\nu):  D \times \Lambda \mapsto \R$. 

Although the radiative transfer equation \eqref{eq:genRTE} is linear, explicit solution formulas are only available in very special cases \cite{Modbook}. Hence, numerical methods are essential for the simulation of the radiative intensity. 
Fortunately, \citet{mishra2020physics} have provided an affirmative answer to Question 2 for the radiative transfer equations, so that one can use physics-informed techniques (based on the strong PDE residual) to retrieve an approximation of the solution of \eqref{eq:genRTE}. 

\begin{theorem}
Let $u \in L^2(\dom)$ be the unique weak solution of the radiative transfer equation \eqref{eq:genRTE}, with absorption coefficient $0 \leq k \in L^{\infty}(D\times \Lambda)$, scattering coefficient $0 \leq \sigma \in L^{\infty}(D\times \Lambda)$ and a symmetric scattering kernel $\Phi \in C^\ell(S \times \Lambda \times S \times \Lambda)$, for some $\ell \geq 1$, such that the function $\Psi$, given by,
\begin{equation}
    \label{eq:psi}
    \Psi(\omega,\nu) = \int\limits_{S \times \Lambda} \Phi(\omega,\omega^{\prime},\nu,\nu^{\prime}) d\omega^{\prime} d\nu^{\prime},
\end{equation}
is in $L^{\infty}(S \times \Lambda)$. For any sufficiently smooth model $u_\theta$ it holds that, 
\begin{eqnarray}
    \norm{u-u_\theta}^2_{L^2} \leq C (\norm{\rpde[u_\theta]}^2_{L^2} + \norm{\rs[u_\theta]}^2_{L^2}+\norm{\rt}^2_{L^2}), 
\end{eqnarray}
where $C>0$ is a constant that only depends (in a monotonously increasing way) on $T$ and the quantity ${s_d^{-1}}({\|\sigma\|_{L^{\infty}}+\|\Psi\|_{L^{\infty}}})$, where $s_d$ is the surface area of the $d$-dimensional unit sphere. 
\end{theorem}
\end{example}

\begin{example}[Navier-Stokes equations]\label{ex:NS-stability}
Next, we revisit Example \ref{def:NS-equation} to show that neural networks for which the PINN residuals are small, will provide a good $L^2$-approximation of the true solution $u:\Omega=D\times [0,T]\to \mathbb{R}^d$, $p:\Omega\to\mathbb{R}$ of the Navier-Stokes equation \eqref{eq:navier-stokes} on the torus $D=\mathbb{T}^d=[0,1)^d$ with periodic boundary conditions. The analysis can be readily extended to other boundary conditions, such as no-slip boundary condition i.e., $u(x,t)=0$ for all $(x,t)\in \partial D\times [0,T]$, and no-penetration boundary conditions i.e., $u(x,t)\cdot \hn_D=0$ for all $(x,t)\in \partial D\times [0,T]$. 

For neural networks $(u_\theta, p_\theta)$, we define the following PINN-related residuals, 
\begin{align}\label{eq:new-pinn-residuals}
    \begin{split}
        &\rpde = \partial_t u_\theta + (u_\theta\cdot\nabla)u_\theta + \nabla p_\theta - \nu\Delta u_\theta, \qquad
        \rdiv = \div{u_\theta},\\ &\rsu(x) = u_\theta(x)-u_\theta(x+1), \qquad \rsp(x) = p_\theta(x)-p_\theta(x+1),\\ &\rsgu(x) =\nabla u_\theta(x)-\nabla u_\theta(x+1), \qquad \rs = (\rsu, \rsp, \rsgu), \\&\rt = u_\theta(t=0)-u(t=0), 
    \end{split}
\end{align}
where we drop the $\theta$-dependence in the definition of the residuals for notational convenience. 

The following theorem \cite{deryck2021navierstokes} then bounds the $L^2$-error of the PINN in terms of the residuals defined above. We write $\abs{\partial D}$ for the $(d-1)$-dimensional Lebesgue measure of $\partial D$ and $\abs{D}$ for the $d$-dimensional Lebesgue measure of $D$. 

\begin{theorem}\label{thm:stability-ns}
Let $d\in\mathbb{N}$, $D=\mathbb{T}^d$ and $u\in C^1(D\times [0,T])$ be the classical solution of the Navier-Stokes equation \eqref{eq:navier-stokes}. Let $(u_\theta,p_\theta)$ be a PINN with parameters $\theta$, then the resulting $L^2$-error is bounded as follows, 
\begin{align}
\begin{split}
   \int_\Omega \norm{u(x,t)-u_\theta(x,t)}_2^2 dxdt &\leq  \mathcal{C} T \exp(T(2d^2\norm{\nabla u}_{L^\infty(\Omega)}+1)),
\end{split}
\end{align}
where the constant $\mathcal{C}$ is defined as,
\begin{eqnarray}
     \mathcal{C} &=& \: \norm{\rt}^2_{L^2(D)} + \norm{\rpde}^2_{L^2(\Omega)} + C_1\sqrt{T}\big[\sqrt{\abs{D}}\norm{\rdiv}_{L^2(\Omega)} \\&&+ (1+\nu)\sqrt{\abs{\partial D}}\norm{\rs}_{L^2(\partial D\times [0,T])}\big], 
\end{eqnarray}
and $C_1 = C_1\big(\norm{u}_{C^1},\norm{u_\theta}_{C^1},\norm{p}_{C^0},\norm{p_\theta}_{C^0}\big)<\infty$. 
\end{theorem}
\begin{proof}
    The proof is in a similar spirit to that of Theorem \ref{thm:stability-heat} and can be found in \cite{deryck2021navierstokes}. 
\end{proof}
\end{example}

\begin{example}[Scalar conservation law]\label{ex:scl-stab}
In this example, we consider viscous scalar conservation laws, as introduced in Example \ref{def:scl}. We first assume that the solution to it is sufficiently smooth, so that we can use the strong PDE residual as in Section \ref{sec:classical-residual}. One can then prove the following stability bound \cite[Theorem 4.1]{MM1}. 

    \begin{theorem}
\label{thm:burg}
Let $\Omega = (0,T) \times (0,1)$, $\nu > 0$ and let $u \in C^k(\Omega)$ be the unique classical solution of the viscous scalar conservation law \eqref{eq:vscl}. Let $u^{\ast} = u_{\theta^{\ast}}$ be the PINN, generated following Section \ref{sec:summary}. Then, the total error is bounded by,
\begin{eqnarray}\label{eq:begenb}
    \norm{u-u_\theta}_{L^2(\Omega)}^2 &\leq& (T+C_1 T^2 e^{C_1T}) \big[\norm{\rpde[u_\theta]}_{L^2(\Omega)}^2 \\&&+ \norm{\rt[u_\theta]}_{L^2(D)}^2 + 2C_2  \norm{\rs[u_\theta]}_{L^2(\partial D\times [0,T])}^2 \\&&+ 2\nu\sqrt{T}C_3 \norm{\rs[u_\theta]}_{L^2(\partial D\times [0,T])} \big]
\end{eqnarray}
Here the constants are defined as, 
\begin{eqnarray}
    C_1 &=& 1+2 \left|f''(\max\{\|u\|_{L^\infty},\|u_\theta\|_{L^\infty}\})\right|\|u_x\|_{L^{\infty}}, \\
    C_2 &=& \|\partial_x u\|_{C(\overline{\Omega})} + \|\partial_x u_\theta\|_{C(\overline{\Omega})}, \\
    C_3 &=& C_3(\|f^{\prime}\|_{\infty},\|u_\theta\|_{C^0(\Omega)}). 
\end{eqnarray}
\end{theorem}
The proof and further details can be found in \cite{MM1}. A close inspection of the estimate \eqref{eq:begenb} reveals that at the very least, the classical solution $u$ of the PDE \eqref{eq:vscl} needs to be in $L^{\infty}((0,T);W^{1,\infty}((0,1)))$ for the RHS of \eqref{eq:begenb} to be bounded. This indeed holds as long as $\nu > 0$. However, it is well known (see \cite{GRbook} and references therein) that if $u^{\nu}$ is the solution of \eqref{eq:vscl} for viscosity $\nu$, then for some initial data,
\begin{equation}
    \label{eq:bbup}
    \|u^{\nu}\|_{L^{\infty}((0,T);W^{1,\infty}((0,1)))} \sim \frac{1}{\sqrt{\nu}}.
\end{equation}
Thus, in the limit $\nu \rightarrow 0$, the constant $C_1$ can blow up (exponentially in time) and the bound \eqref{eq:begenb} no longer controls the generalization error. This is not unexpected as the whole strategy of this paper relies on pointwise realization of residuals. However, the
zero-viscosity limit of \eqref{eq:vscl}, leads to a scalar conservation law with discontinuous solutions (shocks) and the residuals are measures that do not make sense pointwise. Thus, the estimate \eqref{eq:begenb} also points out the limitations of a PINN for approximating discontinuous solutions. 

The above discussion is also the perfect motivation to consider the weak residual, rather than the strong residual, for scalar conservation laws. \citet{deryck2022wpinns} therefore introduce a weak PINN (wPINN) formulation for scalar conservation laws. The wPINN loss function also reflects the fact that \emph{physically admissable} weak solutions should also satisfy an entropy condition, giving rise to an entropy residual. The details of this weak residual and the corresponding loss function were already discussed in Example \ref{ex:scl-weak-residual}. Using the famous doubling of variables argument of Kruzkhov, one can following stability bound on the $L^1$-error with \emph{wPINNs} \cite[Theorem 3.7]{deryck2022wpinns}.

\begin{theorem}\label{thm:stability-scl}
Assume that $u$ is the piecewise smooth entropy solution of \eqref{eq:vscl} with $\nu=0$, essential range $\cC$ and $u(0,t)=u(1,t)$ for all $t\in[0,T]$. 
There is a constant $C>0$ such that for every $\epsilon>0$ and $u_\theta\in C^1(D\times[0,T])$, it holds that 
\begin{equation}\label{eq:stability}
    \begin{split}
        \int_0^1 \abs{u_\theta(x,T)-u(x,T)}dx &\leq C\bigg(\int_0^1 \abs{u_\theta(x,0)-u(x,0)}dx + \max_{c\in\cC, \varphi\in\Phi_\epsilon}\cR(u_\theta,\varphi, c) \\
        &\quad + (1+\norm{u_\theta}_{C^1})\ln(1/\epsilon)^3\epsilon + \int_0^T \abs{u_\theta(1,t)-u_\theta(0,t)}dt\bigg).
    \end{split}
\end{equation}
\end{theorem}
Whereas the bound of Theorem \ref{thm:burg} becomes unusable in the low viscosity regime ($\nu \ll 1 $), the bound of Theorem \ref{thm:stability-scl} is still valid if the solution contains shocks. 

Hence, we can give an affirmative answer to Question 2 for scalar conservation laws using both the PDE residual (classical formulation, only for $\nu>0$) and the weak residual (for $\nu=0$).  
\end{example}
\begin{example}[Poisson's equation]\label{ex:poisson-stab}
We revisit Poisson's equation (Example \ref{def:poisson}) for $\Omega  = [0,1]^d$, but now with zero Neumann conditions i.e., $\tfrac{\partial u}{\partial \nu} = 0$ on $\partial \Omega$, and with $f\in L^2(\Omega)$ with $\int_{\Omega} f dx = 0$. Following Section \ref{sec:variational-residual}, the solution of Poisson's equation is a minimizer of the following loss function, which is equal to the energy functional or variational formulation of the PDE,
\begin{eqnarray}
     I[w] := \frac{1}{2} \int_{\Omega} |\nabla w|^2   dx + \frac{1}{2}\Big(\int_{\Omega} w dx\Big)^2 - \int_{\Omega} f w dx. 
\end{eqnarray}
Now define $\Tilde{u}=  \argmin_{w\in H^1(\Omega)} I[w]$ and define
\begin{equation}
    \cE_G(\theta) := I[u_\theta] - I[\Tilde{u}], 
\end{equation}
where we slightly adapted the definition of the generalization error to reflect the fact that $I[\cdot]$ merely needs to be minimized and does not need to vanish to produce a good approximation. In this setting, one can prove the following stability result \cite{lu2021priori}. 
\begin{proposition}
For any $w\in H^1(\Omega)$ it holds that, 
\begin{equation}
    2(I[w]-I[\Tilde{u}])\leq \|w-\Tilde{u}\|^2_{H^1(\Omega)} \leq  2 \max\{2C_P +1, 2\} (I[w]-I[\Tilde{u}]), 
\end{equation}
where $C_P$ is the Poincar\'e constant on the domain $\Omega$, i.e.,
for any $v\in H^1(\Omega)$,
\begin{eqnarray}
    \Big\|v- \int_{\Omega} v dx  \Big\|_{L^2(\Omega)}^2 \leq C_P \|\nabla v\|_{L^2(\Omega)}^2 .
\end{eqnarray}
\end{proposition}
In particular, by setting $w=u_\theta$ this implies that 
\begin{eqnarray}
    \cE(\theta)^2 := \|u_\theta-\Tilde{u}\|^2_{H^1(\Omega)} \leq  2 \max\{2C_P +1, 2\} \cE_G(\theta). 
\end{eqnarray}
\end{example}

The examples and theorems above might give the impression that stability holds for all PDEs and loss functions. It is important to highlight that this is not the case. The next example discusses a negative answer to Question 2 for the Hamilton-Jacobi Bellman (HJB) equation when using an $L^2$-based loss function \cite{wang20222}. 

\begin{example}[Hamilton-Jacobi Bellman equation] 
\citet{wang20222} consider the class of HJB equations in which the cost rate function is formulated as $r(x,m)= a_1|m_1|^{\alpha_1}+\cdots+a_n|m_n|^{\alpha_n}-f(x,t)$, giving rise to HJB equations of the form,
\begin{equation}
\label{eq:hjb-study}
\begin{cases}
\displaystyle{
\mathcal{L}[u]:=\partial_t u(x,t)+\frac 1 2 \sigma^2 \Delta u(x,t)-\sum_{i=1}^n A_i |\partial_{x_i}u|^{c_i}=f(x,t)}\\
\mathcal{B}[u]:=u(x,T)=g(x),
\end{cases}
\end{equation}
for $x\in\R^d$ and $t\in [0,T]$ and where %$A_i\in(0,+\infty)$ and $c_i\in(1,+\infty)$.
$A_i=(a_i\alpha_i)^{-\frac{1}{\alpha_i-1}}-a_i(a_i\alpha_i)^{-\frac{\alpha_i}{\alpha_i-1}}\in(0,+\infty)$ and $c_i={\frac{\alpha_i}{\alpha_i-1}}\in(1,\infty)$. Their chosen form of the cost function is relevant in optimal control, e.g. in optimal execution problems in finance. 
One of the main results in \citet{wang20222} is that stability can only be achieved when the physics-informed loss is based on the $L^p$ norm of the PDE residual with $p \sim d$. They show that this linear dependency on $d$ cannot be relaxed in the following theorem \cite[Theorem 4.3]{wang20222}. 
\begin{theorem}
\label{thm:lower-bound}
There exists an instance of the HJB equation (\ref{eq:hjb-study}), with exact solution $u$, such that for any $\varepsilon>0$, $A>0$, $r\geq 1$, $m\in\mathbb{N} \cup\{0\}$ and $p\in\left[1,\frac d 4\right]$, there exists a function $w\in C^{\infty}(\mathbb{R}^n\times(0,T])$ such that $\mathrm{supp}(w-u)$ is compact and,
\begin{eqnarray}
    \|\mathcal{L}[w]-f\|_{L^p(\sR^d\times[0,T])}<\varepsilon, \qquad \mathcal{B}[w]=\mathcal{B}[u],
\end{eqnarray}
and yet simultaneously, 
\begin{eqnarray}
    \|w-u\|_{W^{m,r}(\sR^d\times[0,T])}>A.
\end{eqnarray}
\end{theorem}

This shows that for high-dimensional HJB equations, the learned solution may be arbitrarily distant from the true solution $u$ if an $L^2$-based physics-inspired loss based on the classical residual is used. 
\end{example}

Another example that stability of physics-informed learning is not guaranteed can be found in the following two variants on physics-informed neural networks. 

\begin{example}[XPINNs and cPINNs]
\citet{jagtap2020extended} and \citet{jagtap2020conservative} proposed two variants on PINNs, inspired by domain decomposition techniques, in which separate neural networks are trained on different subdomains. In order to ensure continuity of the different subnetworks over the boundaries of the subdomains, both methods add a first term to the loss function to enfore that the function values of the subnetworks are (approximately) equal on the boundaries. \emph{Extended PINNs (XPINNs)} \cite{jagtap2020extended} then propose to add another additional term to make sure the strong PDE residual is (approximately) continuous over the boundaries of the subdomains. In contrast, \emph{conservative PINNs (cPINNs)} \cite{jagtap2020conservative} focus on PDEs of the form $u_t + \nabla_x f(u) = 0$ (i.e. conservation laws) and add an additional term to make sure that the fluxes based on $f$ are (approximately) equal over the boundaries. In \citet{deryck2023stability} an affirmative answer has been given to Question 2 for cPINNs for the Navier-Stokes equations, advection-diffusion equations and scalar conservation laws. On the other hand, XPINNs were not found to be stable (in the sense of Q2) for prototypical PDEs such as Poisson's equation and the heat equation. 
\end{example}

Finally, we highlight some other works that have proven stability results for physics-informed machine learning. 
In \citet{shin2020convergence}, the authors prove a consistency result for PINNs, for linear elliptic and parabolic PDEs, where they show that if $\cE_G(\theta_m)\to 0$ for a sequence of neural networks $\{u_{\theta_m}\}_{m\in\mathbb{N}}$, then $\norm{u_{\theta_m}-u}_{L^\infty}\to0$, under the assumption that one adds a specific $C^{k,\alpha}$-regularization term to the loss function, thus partially addressing question Q2 for these PDEs. However, this result does not provide quantitative estimates on the underlying errors. A similar result, with more quantitative estimates for advection equations is provided in \citet{Zhang1}. 

\subsection{Stability for inverse problems}\label{sec:stab-inverse}
Next, we extend our analysis to inverse problems, as introduced in Section \ref{sec:inverse}. A crucial ingredient to answer Question 2 for inverse problems will be the assumption that solutions to the inverse problem, defined by \eqref{eq:pde} and \eqref{eq:dt}, satisfy the following \emph{conditional stability estimate}.
Let $\hat{X} \subset X^{\ast} \subset X = L^{p}(\dom)$ be a Banach space. For any $\bu,\bv \in \hat{X}$, the differential operator $\df$ and restriction operator $\map$ satisfy,
\begin{equation}
    \label{eq:assm-inv}
(H5):\quad    \|\bu - \bv\|_{L^{p}(E)} \leq C_{\mathrm{pd}} \left(\|\df(\bu) - \df(\bv)\|^{\tau_p}_{Y} + \|\map(\bu) - \map(\bv)\|^{\tau_d}_{Z} \right),
\end{equation}
for some $0 < \tau_p,\tau_d \leq 1$, for any subset $\dom^{\prime} \subset E \subset \dom$ and where $C_{\mathrm{pd}}$ can depend on $\|\bu\|_{\hat{X}}$ and $\|\bv\|_{\hat{X}}$. This bound \eqref{eq:assm-inv} is termed a \emph{conditional stability estimate} as it presupposes that the underlying solutions have sufficiently regularity as $\hat{X} \subset X^{\ast} \subset X$. 
\begin{remark}
We can extend the hypothesis for the inverse problem in the following ways,
\begin{itemize}
    \item Allow the measurement set $\dom^{\prime}$ to intersect the boundary i.e, $\partial \dom^{\prime} \cap \dom \neq \emptyset$.
    \item Replace the bound \eqref{eq:assm-inv} with the weaker bound,
    \begin{equation}
    \label{eq:assm3-inv}
(H6):\quad    \|\bu - \bv\|_{L^{p}(E)} \leq C_{\mathrm{pd}} \omega \left(\|\df(\bu) - \df(\bv)\|_{Y} + \|\map(\bu) - \map(\bv)\|_{Z} \right),
\end{equation}
    with $\omega:\R \mapsto \R_+$ being a modulus of continuity. 
\end{itemize}
\end{remark}

We will prove a general estimate on the error due to a model $u_\theta$ in approximating the solution $\bu$ of the inverse problem for PDE \eqref{eq:pde} with data \eqref{eq:dt}. 
Following \citet{MM1} we set $\dom^{\prime} \subset E \subset \dom$ and define the total error \eqref{eq:pinn-error} as,
 \begin{equation}
    \label{eq:egen-inv}
     \Eg(E;\theta) := \|\bu-u_{\theta}\|_{L^{p}(E)}.
\end{equation}
We will bound the error in terms of the PDE residual \eqref{eq:PDE-residual} and the data residual \eqref{eq:data-residual} \cite[Theorem 2.4]{mishra2021estimates}. 
\begin{theorem}
\label{thm:gen-bound-inv}
Let $\bu \in \hat{X}$ be the solution of the inverse problem associated with PDE 
\eqref{eq:pde} and data \eqref{eq:dt}. Assume that the stability hypothesis \eqref{eq:assm-inv} holds for any $\dom^{\prime} \subset E \subset \dom$. Let $u_\theta$ be any sufficiently smooth function. Assume that the residuals $\rpde$ and  $\cR_d$ are square integrable. Then the following estimate on the generalization error \eqref{eq:egen-inv} holds,
\begin{equation}
    \label{eq:egenb-inv}
     \|\bu - u_\theta \|_{L^{2}(E)} \leq C_{\mathrm{pd}} \left(\norm{\rpde[u_\theta]}_{L^2(\dom)}^{\tau_p} + \norm{\cR_d[u_\theta]}_{L^2(\dom')}^{\tau_d}\right),
\end{equation}
with constants $C_{\mathrm{pd}} = C_{\mathrm{pd}}\left(\|\bu\|_{\hat{X}},\|\bu^{\ast}\|_{\hat{X}}\right)$ as from \eqref{eq:assm-inv}.  
\end{theorem}

Many concrete examples of PDEs where PINNs are used to find a unique continuation can be found in \citet{mishra2021estimates}. The next examples summarize their results for the Poisson equation and heat equation. 

\begin{example}[Poisson equation]
We consider the inverse problem for the Poisson equation, as introduced in Example \ref{ex:poisson-inv}. 
In this case, the conditional stability, cf. \eqref{eq:assm-inv}, is guaranteed by the three balls inequality \cite{ALE1}. Therefore, a result like Theorem \ref{thm:gen-bound-inv} holds. However, note that this theorem only calculates the generalization error on $E\subset \dom$. It can however be guaranteed that the generalization error is small on the whole domain $\dom$ and even in Sobolev norm, as follows from the following lemma \cite[Lemma 3.3]{mishra2021estimates}. 
\begin{lemma}
\label{lem:inv-poisson} 
For $f \in C^{k-2}(\dom)$ and $g \in C^k(\dom^{\prime})$, with continuous extensions of the functions and derivatives up to the boundaries of the underlying sets and with $k \geq 2$, Let $u \in H^1(\dom)$ be the weak solution of the inverse problem corresponding to the Poisson's equation \eqref{eq:ps} and let $u_\theta$ be any sufficiently smooth model. Then, the total error is bounded by, 
\begin{equation}
    \label{eq:pbd1}
    \|u-u_\theta\|_{H^1(D)} \leq C 
    %\left(\|u\|^{1-\tau}_{H^1(D)}+\|u_\theta\|^{1-\tau}_{H^1(D)}\right)
    \left| \log\left(\norm{\rpde[u_\theta]}_{L^2(\dom)} + \norm{\cR_d[u_\theta]}_{L^2(\dom')}\right)\right|^{-\tau}.
\end{equation}
for some $\tau \in (0,1)$ and a constant $C>0$ depending on $u$, $u_\theta$ and $\tau$.
\end{lemma}
\end{example}

\begin{example}[Heat equation]
We consider the data assimilation problem for the heat equation, as introduced in Example \ref{ex:heat-inv}, which amounts to finding the solution $u$ of the heat equation in the whole space-time domain $\Omega = D\times (0,T)$, given data on the observation sub-domain $\Omega' = D'\times (0,T)$. For any $0 \leq \tau < T$, we define the error of interest for the model $u_\theta$ as,
\begin{equation}
    \label{eq:gerht}
    \cE_{\tau}(\theta) = \|u - u_\theta\|_{C([\tau,T];L^2(D))} + \|u - u_\theta\|_{L^2((0,T);H^1(D))}.
\end{equation}
The theory for this data assimilation inverse problem for the heat equation is classical and several well-posedness and stability results are available. Our subsequent error estimates for physics-informed models rely on a classical result of Imanuvilov \cite{Im1}, based on the well-known Carleman estimates. 
Using these results, one can state the following theorem \cite[Lemma 4.3]{mishra2021estimates}. 
\begin{theorem}
For $f \in C^{k-2}(\Omega)$ and $g \in C^k(\Omega')$, with continuous extensions of the functions and derivatives upto the boundaries of the underlying sets and with $k \geq 2$, let $u \in H^1((0,T);H^{-1}(D)) \cap L^2((0,T);H^1(D))$ be the solution of the inverse problem corresponding to the heat equation and satisfies the data \eqref{eq:dtht}. Then, for any $0\leq \tau <T$, the error \eqref{eq:gerht} corresponding to the sufficiently smooth model $u_\theta$ is bounded by, 
\begin{equation}
    \label{eq:hbd}
    \cE_\tau(\theta) \leq C \left(\norm{\rpde[u_\theta]}_{L^2(\dom)} + \norm{\cR_d[u_\theta]}_{L^2(\dom')} + \norm{\rs[u_\theta]}_{L^2(\partial D \times (0,T))} \right),
\end{equation}
for some constant $C$ depending on $\tau$, $u$ and $u_\theta$.
\end{theorem}
\end{example}

\begin{example}[Stokes equation]
The effectiveness of PINNs in approximating inverse problems was showcased in the recent paper \cite{KAR4}, where the authors proposed PINNs for the data assimilation problem with the Navier-Stokes equations (Example \ref{def:NS-equation}). As a first step towards rigorously analyzing this, we follow \citet{mishra2021estimates} and focus on the stationary Stokes equation as introduced in Example \ref{ex:stokes-inv}. 

Recall that the data assimilation inverse problem for the Stokes equation amounts to inferring the velocity field $\bu$ (and the pressure $p)$), given $\f,f_d$ and $g$. In particular, we wish to find solutions $\bu \in H^1(D;\R^d)$ and $p \in L_0^2(D)$ (i.e, square integrable functions with zero mean), such that the following holds, 
\begin{equation}
    \label{eq:wkst}
    \begin{aligned}
    \int\limits_D \nabla \bu\cdot \nabla \bv dx + \int\limits_D p \div(\bv) dx &= \int\limits_D \f \bv dx, \\
    \int\limits_D \div(\bu) w dx &= \int\limits_D f_d w dx, \end{aligned}
\end{equation}
for all test functions $\bv \in H^1_0(D;\R^d)$ and $w \in L^2(D)$.

Let $B_{R_1}(x_0)$ be the largest ball inside the observation domain $D^{\prime} \subset D$. We will consider balls $B_{R}(x_0) \in D$
such that $R > R_1$ and estimate the following error for the model $u_\theta$,
\begin{equation}
    \label{eq:gerst}
    \cE_R(\theta) := \|\bu-\bu_\theta\|_{L^2(B_R(x_0))}. 
\end{equation}
The well-posedness and conditional stability estimates for the data assimilation problem for the Stokes equation \eqref{eq:st}, \eqref{eq:dtst} has been extensively investigated in \citet{Uhl} and references therein. Using these results, we can state the following estimate on $\cE_R(\theta)$ \cite[Lemma 6.2]{mishra2021estimates}. The physics-informed residuals $\rpde$ and $\rdiv$ are defined analogously to those in Example \ref{ex:NS-stability}.
\begin{theorem}
\label{lem:st1} 
For $\f \in C^{k-2}(D;\R^d), f_d \in C^{k-1}(D)$ and $g \in C^k(D^{\prime})$, with $k \geq 2$. Let $\bu \in H^1(D;\R^d)$ and $p \in H^1(D)$ be the solution of the inverse problem corresponding to the Stokes' equations \eqref{eq:st} i.e, they satisfy \eqref{eq:wkst} for all test functions $\bv \in H^1_0(D;\R^d)$, $w\in L^2(D)$ and satisfies the data \eqref{eq:dtst}. Let $u_\theta$ be a sufficiently smooth model and and let $B_{R_1}(x_0)$ be the largest ball inside $D^{\prime} \subset D$.
Then, there exists $\tau\in (0,1)$ such that the generalization error \eqref{eq:gerst} for balls $B_{R}(x_0) \subset D$ with $R > R_1$ is bounded by,
\begin{equation}
    \label{eq:stbd}
    \begin{aligned}
    \cE_R(\theta)^2 \leq C \cdot  \max_{\gamma\in \{1,\tau\}} (\norm{\rpde[u_\theta]}_{L^2(D)}^2+\norm{\rdiv[u_\theta]}_{L^2(D)}^2)^\gamma(1+\norm{\cR_d[u_\theta]}_{L^2(D)}^{2\tau}),
    \end{aligned}
\end{equation}
where $C$ depends on $u$ and $u_\theta$.
\end{theorem}
    
\end{example}
Further examples can be found in \citet{mishra2021estimates}. 

\section{Generalization}\label{sec:generalization}

Next, we move our attention to Question 3. 
\begin{question}[Q3]
\emph{Given a small training error $\Eg_T^*$ and a sufficiently large training set $\S$, will the corresponding generalization error $\cE_G^*$ also be small?}
\end{question}
We will answer this question by proving that for any model $u_\theta$ it holds that
\begin{eqnarray}
    \cE_G(\theta) \leq \cE_T(\theta) + \epsilon(\theta), 
\end{eqnarray}
for some value $\epsilon(\theta)>0$ that depends on the model class and the size of the training set. Using the terminologly of the error decomposition \eqref{eq:error-decomposition}, this will then imply that the \emph{generalization gap} is small,
\begin{eqnarray}
    \sup_{\theta\in\Theta}\abs{\Et(\theta,\S)-\Eg_G(\theta)} \leq \sup_{\theta\in\Theta} \epsilon(\theta). 
\end{eqnarray}
As $\epsilon(\theta)$ can depend on $\theta$, one must choose the model (parameter) space $\Theta$ in a suitable way to avoid that $\sup_{\theta\in\Theta} \epsilon(\theta)$ diverges. 

In view of Section \ref{sec:quad}, one can realize that the training error is nothing more than applying a suitable quadrature rule to each term in the generalization error. 
It is clear that the error bound for a quadrature rule immediately proves that the generalization error can be bounded in terms of the training error and the training set size, thereby answering question Q3, at least for deterministic quadrature rules. We can now combine this generalization result with the stability bounds from Section \ref{sec:stability} to prove an upper bound on the total error of the optimized model $\Eg^*$ \eqref{eq:pinn-error} \cite[Theorem 2.6]{MM1}. 

\begin{theorem}
\label{thm:pinn-error-det}
Let $\bu \in X^{\ast}$ be the unique solution of the PDE 
\eqref{eq:pde} and assume that the stability hypothesis \eqref{eq:assm2} holds. Let $u_\theta \in X^{\ast}$ be a model and let $\S$ be a training set of quadrature points corresponding to the quadrature rule \eqref{eq:quad} with order $\alpha$, as in \eqref{eq:assm3}. Then for any $L^2$-based generalization error $\cE_G$, such as \eqref{eq:epil-ti} or \eqref{eq:epil-td}, it holds that,
\begin{equation}
    \label{eq:egenb}
    \Eg_G(\theta)^2 \leq \Eg_T(\theta)^2 + C_{\text{quad}}N^{-\alpha},
\end{equation}
with constant $C_{\text{quad}}$ stemming from \eqref{eq:assm3}, and which might depend on $u_\theta$. 
\end{theorem}

The previously described approach only makes sense for low to moderately high dimensions ($d<20$) since the convergence rate of numerical quadrature rules with deterministic quadrature decreases with increasing $d$. For problems in very high dimensions, $d \gg 20$, Monte-Carlo quadrature is the numerical integration method of choice. In this case, the quadrature points are randomly chosen, independent and identically distributed (with respect to a scaled Lebesgue measure). Since in this case the optimized model is correlated with all training points, one must be careful when one calculates the convergence rate. We can, however, prove error bounds that hold with a certain probability, or alternatively, that hold averaged over all possible training sets of the same size. 

In this setting, we prove a general \emph{a posteriori} upper bound on the generalization error. Consider $\cF:\cD\to\R$ (an operator or function) and a corresponding model $\cF_\theta:\cD\to\R$, $\theta\in\Theta$.
%For physics-informed learning we can set $f=0$, for supervised function learning one has $f=u$ and for operator learning $f=\cG$. 
Given a training set $\S=\{X_1, \ldots X_N\}$, where $\{X_i\}_{i=1}^N$ are iid random variables on $\cD$ (according to a measure $\mu$), the training error $\Et$ and generalization error $\Eg$ are, 
\begin{equation}
    \cE_T(\theta, \S)^2 = \frac{1}{N}\sum_{i=1}^N \abs{\cF(z_i)-\cF_\theta(z_i)}^2, \qquad \cE_G(\theta)^2 = \int_\cD\abs{\cF_\theta(z)-\cF(z)}^2d\mu(z), 
\end{equation}
where $\mu$ is a probability measure on $\cD$. This setting allows to bound all possible terms and residuals that were mentioned in Section \ref{sec:residuals}:
\begin{itemize}
    \item For the term resulting from the PDE residual \eqref{eq:PDE-residual} we can set $\cD = \Omega$, $\cF=0$ and $\cF_\theta = \cL[u_\theta]-f$. 
    \item For the data term \eqref{eq:data-residual}, we can set $\cD = \dom'$, $\cF=u$ and $\cF_\theta = u_\theta$. Similarly, for the term arising from the spatial boundary conditions we set $\cD = \partial D$ (or $\cD = \partial D\times [0,T]$) and for the term arising from the initial condition we set $\cD = D$. 
    \item For operator learning with an input function space $\cX$ we can set $\cD = \Omega\times \cX$, $\cF=\cG$ and $\cF_\theta = \cG_\theta$. 
    \item Finally, for physics-informed operator learning one can set $\cD = \Omega\times \cX$, $\cF=0$ and $\cF_\theta = \cL(\cG_\theta)$. 
\end{itemize}

With the above definitions in place, we can state the following theorem \cite[Theorem 3.11]{deryck2022generic}, which provides a computable a posteriori error bound on the expectation of the generalization error for a general class of approximators. We refer to e.g. \citet{beck2020error, deryck2021pinn} for bounds on $n$, $c$ and $\mathfrak{L}$. 

\begin{theorem}\label{thm:generalization}
For $R>0$ and $N,n\in\N$, let $\Theta = [-R,R]^{n}$ be the parameter space, and for every training set $\S$, let $\theta^*(\S) \in\Theta$ be an (approximate) minimizer of $\theta\mapsto\Et(\theta,\S)^2$, assume that $\theta\mapsto\Eg(\theta,\S)^2$ and $\theta\mapsto\Et(\theta,\cS)^2$ are bounded by $c>0$ and Lipschitz continuous with Lipschitz constant $\mathfrak{L}>0$.  
If $N\geq 2c^2e^8/(2R\mathfrak{L})^{{n}/2}$ then it holds that 
\begin{equation}\label{eq:generalization}
    \E{\cE_G(\theta^*(\S))^2} \leq \E{\cE_T(\theta^*(\S),\S)^2} + \sqrt{\frac{2c^2(n+1)}{N}\ln(R\mathfrak{L}\sqrt{N})}.
\end{equation}
\end{theorem}
\begin{proof}
The proof combines standard techniques, based on covering numbers and Hoeffding's inequality, with an error composition from \citet{deryck2021pinn}. 
\end{proof}
Due to its generality, Theorem \ref{thm:generalization} provides a satisfactory answer to Question 3 when a randomly generated training set is used. The two central assumptions, the existence of the constants $c>0$ and $\mathfrak{L}>0$, should however not be swept under the rug. In particular for neural networks it might be initially unclear how large these constants are. 

For any type of neural network architecture of depth $L$, width $W$ and weights bounded by $R$, one finds that $N \sim LW(W+d)$. For tanh neural networks and operator learning architectures, one has that $\ln(\mathfrak{L}) \sim L\ln(dRW)$, whereas for physics-informed neural networks and DeepONets one finds that $\ln(\mathfrak{L}) \sim (k+\ell)L\ln(dRW)$ with $k$ and $\ell$ as in Assumption \ref{ass:diff-operator} \cite{LMK1, deryck2021pinn}. Taking this into account, one also finds that the imposed lower bound on $n$ is not very restrictive. Moreover, the RHS of \eqref{eq:generalization} depends at most polynomially on $L, W, R, d, k, \ell$ and $c$. For physics-informed architectures, however, upper bounds on $c$ often depend exponentially on $L$ \cite{deryck2021pinn, deryck2021navierstokes}. 

\begin{remark}
As Theorem \ref{thm:generalization} is an a posteriori error estimate, one can use the network sizes of the trained networks for $L$, $W$ and $R$. The sizes stemming from the approximation error estimates of the previous sections can be disregarded for this result. Moreover, instead of considering the expected values of $\Eg$ and $\Et$ in \eqref{eq:generalization}, one can also prove that such an inequality holds with a certain probability \cite{deryck2022generic}. 
\end{remark}

Next, we show how Theorem \ref{thm:pinn-error-det} and Theorem \ref{thm:generalization} can be used in practice as a posteriori error estimate. %In addition, we discuss for the Navier-Stokes equation also how the proven bounds can be used as an \emph{a priori} error estimate, which guarantees the smallness of the total error $\cE^*$ provided that the hypothesis class and training set are large enough on the one hand, and that the physics-informed loss can be properly minimized on the other hand. 

\begin{example}[Navier-Stokes]
We first state the a posteriori estimate for the Navier-Stokes equation (Example \ref{def:NS-equation}). By using the midpoint rule \eqref{eq:quad-error} on the sets $D\times [0,T]$ with $N_\mathrm{int}$ quadrature points, $D$ with $N_t$ quadrature points and $\partial D \times [0,T]$ with $N_s$ quadrature points, one can prove the following estimate \cite[Theorem 3.10]{deryck2021navierstokes}. 

\begin{theorem}\label{thm:generalization-ns}
Let $T>0$, $d\in\mathbb{N}$, let $(u,p)\in C^4(\mathbb{T}^d\times [0,T])$ be the classical solution of the Navier-Stokes equation \eqref{eq:navier-stokes} and let $(u_\theta,p_\theta)\in C^4(\mathbb{T}^d\times [0,T])$ be a model. Then the following error bound holds,
\begin{equation}
    \label{eq:etrain1}
\begin{aligned}
   \cE(\theta)^2=\int_\Omega \norm{u(x,t)-u_\theta(x,t)}_2^2 dxdt &\leq \bigO\left(\Et(\theta,\S) + N_t^{-\frac{2}{d}} + N_\mathrm{int}^{-\frac{1}{d+1}} + N_s^{-\frac{1}{d}}\right).
\end{aligned}
\end{equation}
\end{theorem}
The exact constant implied in the $\bigO$-notation depends on  $(u_\theta,p_\theta)$ and $(u,p)$ and their derivatives \cite[Theorem 3.10]{deryck2021navierstokes}. There are two interesting observations to be made from equation \eqref{eq:etrain1}: 
\begin{itemize}
    \item When the training set is large enough, the total error scales with the \emph{square root} of the training error, $\cE(\theta) \sim \sqrt{\cE_T(\theta,\cS)}$. This sublinear relation is also observed in practice, see Figure \ref{fig:ns}. 
    \item The bound \eqref{eq:etrain1} reveals different convergence rates in terms of the training set sizes $N_\mathrm{int}$, $N_t$ and $N_s$. In particular, one needs relatively many training points in the interior of the domain compared to its boundary. We will see that these ratios will be different for other PDEs. 
\end{itemize}

In \citet{deryck2021navierstokes} the above result was verified by training a PINN to approximate the Taylor-Green vortex in two space dimensions, the exact solution of which is given by,
\begin{eqnarray}
 u(t,x,y) &=& - \cos(\pi x) \sin(\pi y) \text{exp}(-2\pi^2\nu t)  \\
 v(t,x,y) &=& \sin(\pi x) \cos(\pi y) \text{exp}(-2\pi^2\nu t) \\
p(t,x,y) &=& -\frac{\rho}{4}\left[\cos(2\pi x) + \cos(2\pi y)\right] \text{exp}(-4 \pi^2\nu t)  
\end{eqnarray}
The spatio-temporal domain is $x, y \in [0.5,4.5]^2$ and $t \in [0,1]$.
The results are reported in Figure \ref{fig:ns}. As one can see, all error types decrease with increasing number of quadrature points. In particular, the relationship $\cE(\theta) \sim \sqrt{\cE_T(\theta,\cS)}$ is (approximately) observed as well. 

% It is initially unclear whether Theorem \ref{thm:generalization-ns} can also be used as an a priori estimate, given the dependence of the bound on the derivatives of the model. It turns out that this is possible for PINNs if the true solution is sufficiently smooth and the neural network and training set sufficiently large  \cite[Corollary 3.9]{deryck2021navierstokes}. 

% \begin{corollary}
% Let $\epsilon>0$, $T>0$, $d\in\mathbb{N}$, $k>6(3d+8) =: \gamma$, let $(u,p)\in H^k(\mathbb{T}^d\times [0,T])$ be the classical solution of the Navier-Stokes equation \eqref{eq:navier-stokes}, let the hypothesis space consists of PINNs for which the upper bound on the weights is at least $R \geq \epsilon^{-1/(k-\gamma)}\ln(1/\epsilon) $, the width is at least $W \geq \epsilon^{-(d+1)/(k-\gamma)}$ and the depth is at least $L \geq 3$,  let $(u_{\theta^*(\S)},p_{\theta^*(\S)})$ be the PINN that solves \eqref{eq:minimization-problem} and let the training set $\S$ satisfy  $N_t \geq \epsilon^{-d(1+\gamma/(k-\gamma))}$, $N_{\mathrm{int}} \geq \epsilon^{-2(d+1)(1+\gamma/(k-\gamma))}$ and $N_s \geq \epsilon^{-2d(1+\gamma/(k-\gamma))}$. It holds that
% \begin{equation}
%     \Vert u-u_{\theta^*(\S)}\Vert_{L^2(D\times [0,T])} = \bigO(\epsilon).
% \end{equation}
% \end{corollary}

\begin{figure}
\centering
\begin{minipage}{.52\textwidth}
  \centering
  \includegraphics[width=\linewidth]{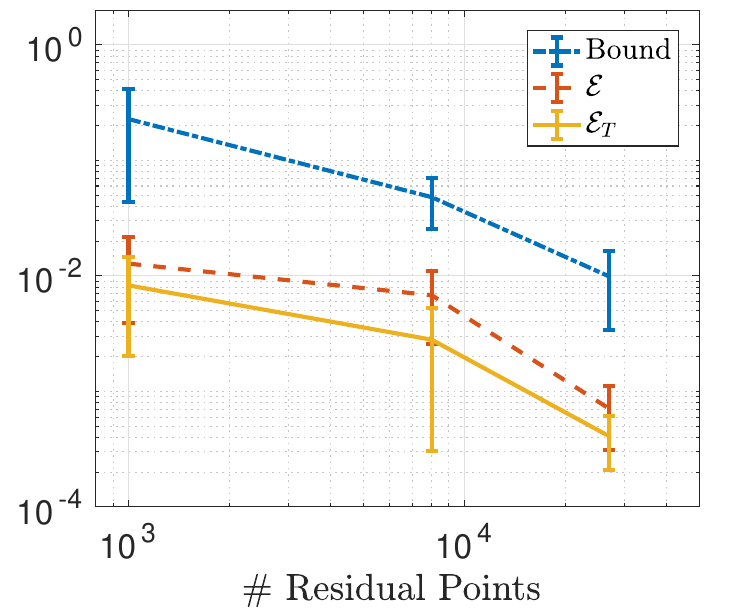}
\end{minipage}%
\begin{minipage}{.48\textwidth}
  \centering
  \includegraphics[width=\linewidth]{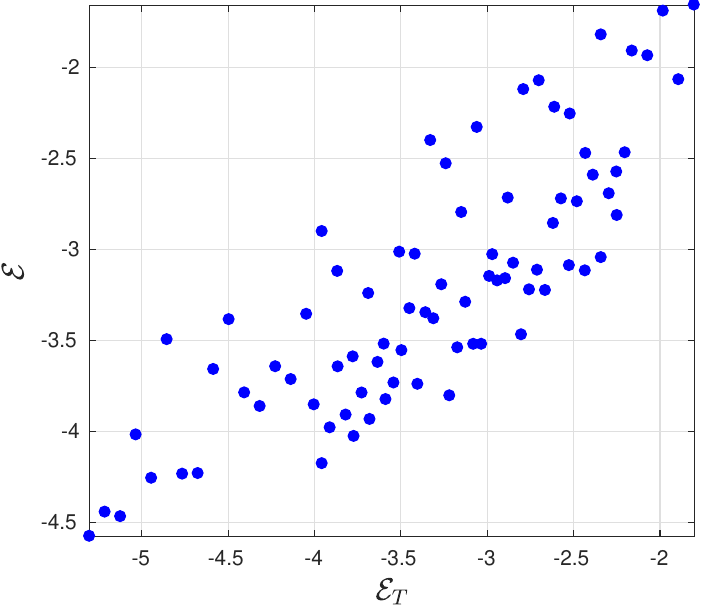}
\end{minipage}
\caption{Experimental results for the Navier-Stokes equation with $\nu=0.01$. The total error $\cE$ and the training error $\cE_T$ are shown in terms of the number of residual points $N_{\mathrm{int}}$ (\emph{left}; and compared with the bound from Theorem \ref{thm:generalization-ns}) and also in terms of each other (\emph{right}). Figure from \citet{deryck2021navierstokes}.}
        \label{fig:ns}
\end{figure}

\end{example}

In the next example we show that Theorem \ref{thm:generalization} can be used to show that also the curse of dimensionality can be overcome in the training set size. 
\begin{example}[High-dimensional heat equation]
We consider the high-dimensional (linear) heat equation (Example \ref{def:semilinear-heat}), which is an example of a linear Kolmogorov equation, for which we have proven that the CoD can be overcome in the approximation error in Example \ref{ex:kolmogorov-approx}. In this example, we show that the training set size only needs to grow at most polynomially in the input dimension $d$ to achieve a certain accuracy. Applying Theorem \ref{thm:generalization} to every term of the RHS of stability result for the heat equation (Theorem \ref{thm:stability-heat}) then gives us the following result.
\begin{theorem}\label{thm:generalization-heat}
Under the assumptions of Theorem \ref{thm:generalization} and given a training set of size $N_{\mathrm{int}}$ on $D\times [0,T]$, size $N_t$ on $D$ and $N_s$ on $\partial D \times [0,T]$, we find that there exist constants $C, \alpha >0$ that are independent of $d$ such that,
    \begin{eqnarray}
         \E{\cE(\theta^*(\S))^2} \leq Cd^\alpha \left(\E{\cE_T(\theta^*(\S),\S)} + \frac{\ln(N_{\mathrm{int}})}{N_{\mathrm{int}}^{2}}+\frac{\ln(N_t)}{N_t^{2}}+\frac{\ln(N_s)}{N_s^{4}}\right)
    \end{eqnarray}
\end{theorem}
A more general result for linear Kolmogorov equations, including a more detailed expression for $C$ and $\alpha$, can be found in \citet{deryck2021pinn}. A comparison with the corresponding result for the Navier-Stokes equations (Theorem \ref{thm:generalization-ns}) reveals: 
\begin{itemize}
    \item Just like for the Navier-Stokes equations, the total error scales with the \emph{square root} of the training error, $\cE(\theta) \sim \sqrt{\cE_T(\theta,\cS)}$ when the training set is large enough. 
    \item For the Navier-Stokes equations (with $d=2$ or $d=3$) we found that one needs to choose $N_{\mathrm{int}} \gg N_s \gg N_t$. However, in this case we find that much less training points in the interior are needed, as $N_s \gg N_t \approx N_{\mathrm{int}}$ is sufficient. 
\end{itemize}

As a first example, the authors of \citet{mishra2021estimates} consider the one-dimensional heat equation on the domain $[-1,1]$ with initial condition $\sin(\pi x)$ and final time $T=1$. The results can be seen in Figure \ref{fig:heat-error}. As one can see, both the generalization and training error decrease with increasing number of quadrature points. In particular, the relationship $\cE(\theta) \sim \sqrt{\cE_T(\theta,\cS)}$ is (approximately) observed as well. 

As a second example, we consider the one-dimensional heat equation with parametric initial condition, and where the parameter space is very high-dimensional. In more detail, the considered initial condition is parametrized by $\mu$ and given by, 
\begin{eqnarray}
    u_0(x;\mu) = \sum_{i=1}^d \frac{-1}{dm^2}\sin(m\pi(x-\mu_i)).
\end{eqnarray}
In Figure \ref{fig:heat-cod}, it is demonstrated that the generalization error $\cE_G$ does not grown exponentially  (but rather sub-quadratically) in the parameter dimension $d$, and that therefore the curse of dimensionality is mitigated. Even for very large $d$ the relative generalization error is less than $2\%$. 

\begin{figure}
\centering
\begin{minipage}{.5\textwidth}
  \centering
  \includegraphics[width=\linewidth]{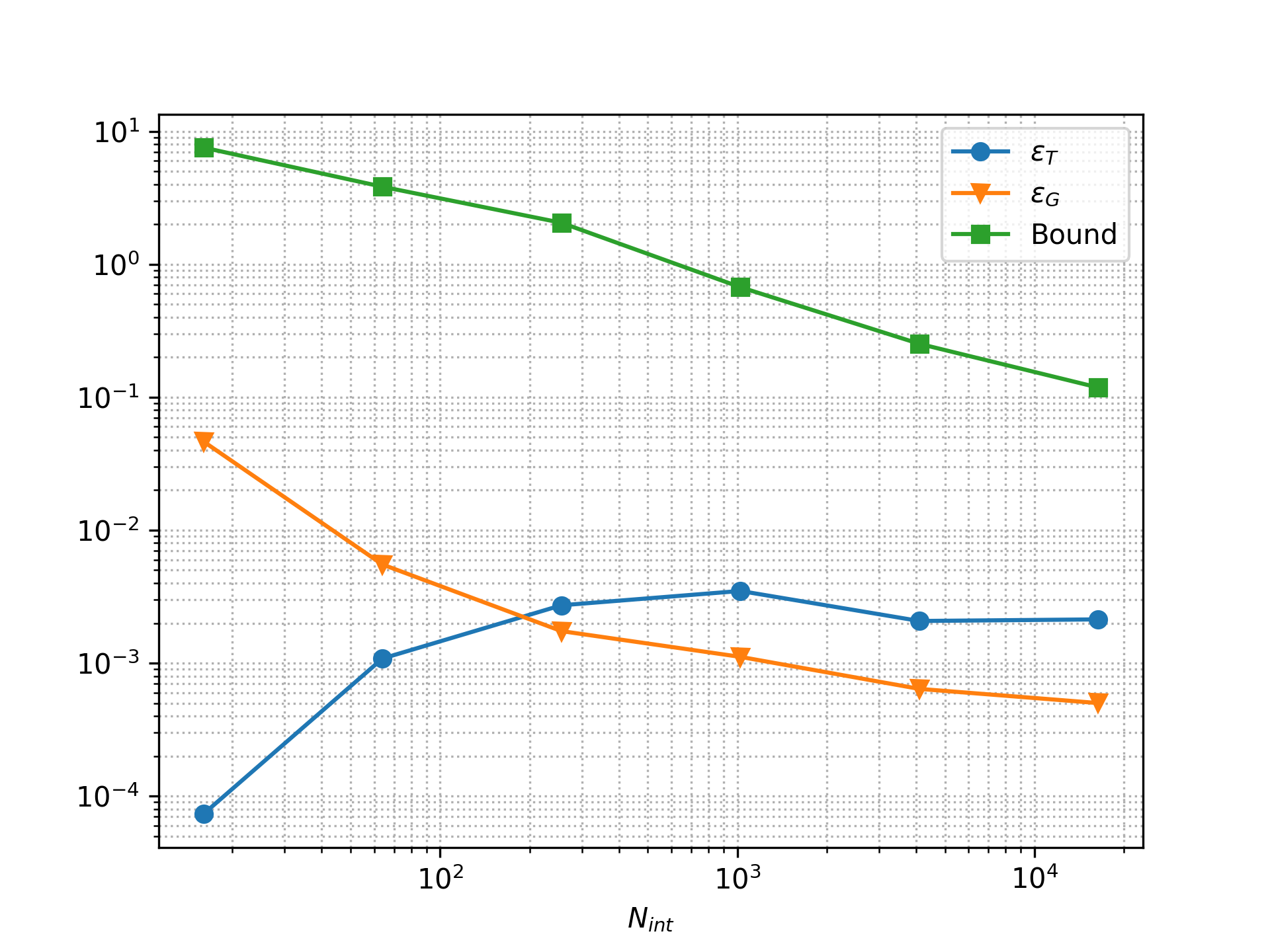}
\end{minipage}%
\begin{minipage}{.5\textwidth}
  \centering
  \includegraphics[width=\linewidth]{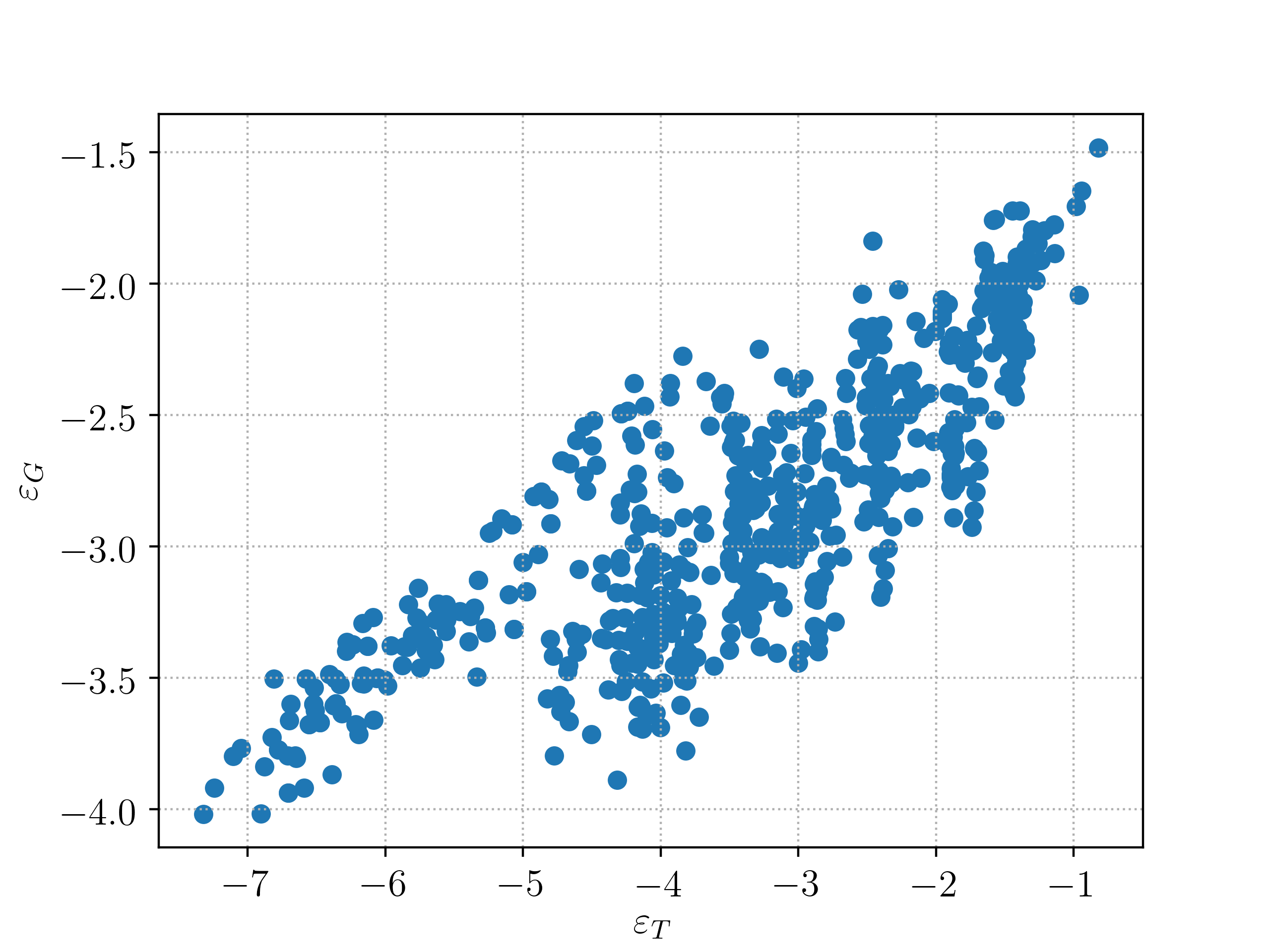}
\end{minipage}
\caption{Experimental results for the heat equation. The generalization error $\cE_G$ and the training error $\cE_T$ are shown in terms of the number of residual points $N_{\mathrm{int}}$ (\emph{left}; and compared with the bound from Theorem \ref{thm:generalization-heat}) and also in terms of each other (\emph{right}). Figure from \citet{molinaro_thesis}. }
        \label{fig:heat-error}
\end{figure}

\begin{figure}
    \centering
    \includegraphics[width=0.5\textwidth]{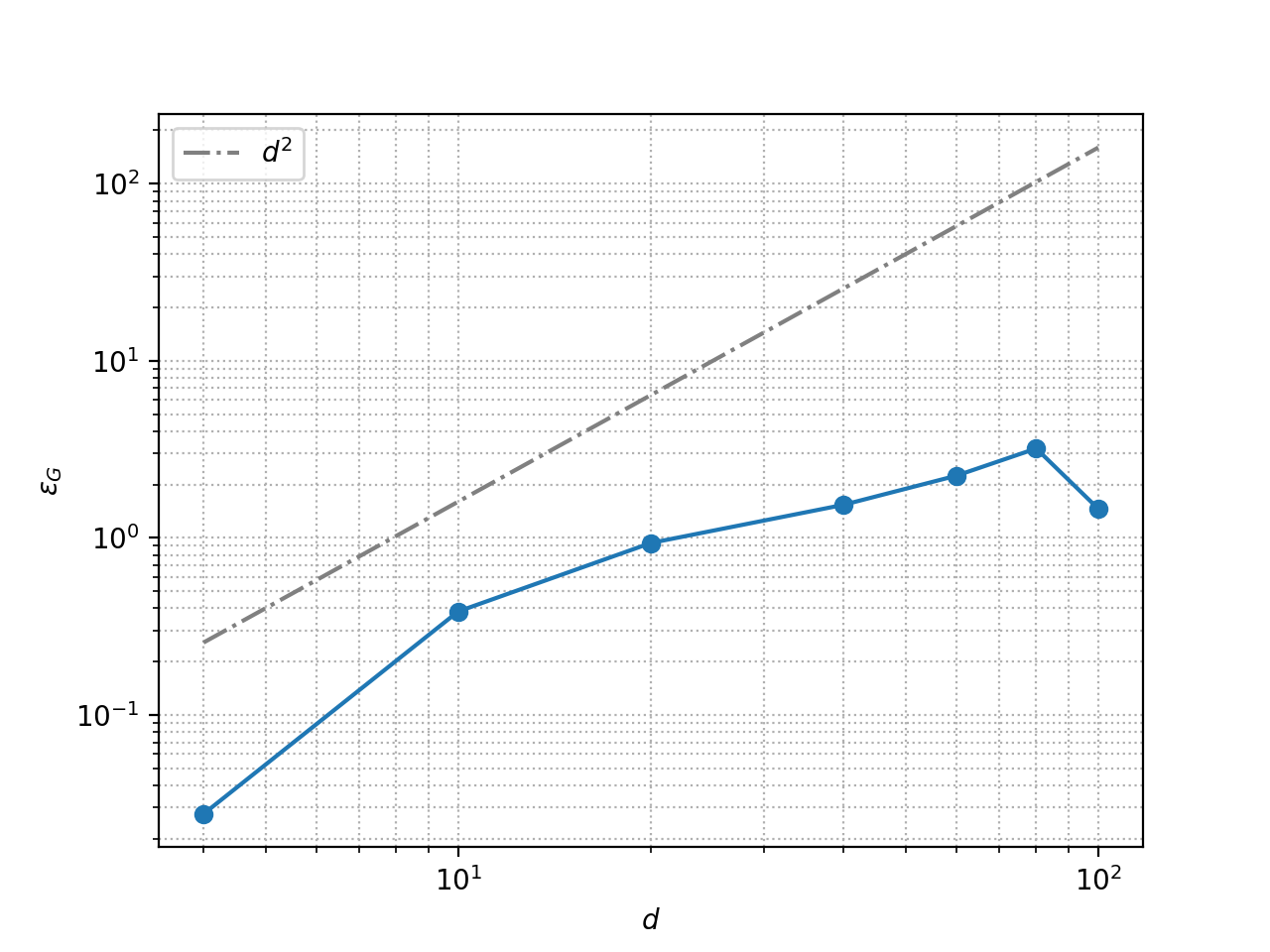}
    \caption{Relative generalization error $\cE_G$ (in percent) in terms of the parameter dimension $d$ for the heat equation with a parametrized initial condition. Figure from \citet{molinaro_thesis}.}
    \label{fig:heat-cod}
\end{figure}

\end{example}

\begin{example}[Scalar conservation laws]
We consider again weak PINNs for the scalar conservation laws. Recall from the stability result (Theorem \ref{ex:scl-stab}) that the generalization error is given by, 
\begin{equation}\label{eq:def-gen-error}
    \begin{split}
        \Eg(\theta,\varphi, c) &= - \int_0^1\int_0^T \left(\abs{u_{\theta^*(\S)}(x,t)-c}\partial_t\varphi(x,t) + Q[u_{\theta^*(\S)}(x,t); c]\partial_x\varphi(x,t) \right)dx dt\\
        &\qquad + \int_0^1 \abs{u_\theta(x,0)-u(x,0)} dx + \int_0^T \abs{u_\theta(0,t)-u_\theta(1,t)} dt. 
    \end{split}
\end{equation}
To this end, we consider the simplest case of random (Monte Carlo) quadrature and generate a set of collocation points, $\S = \{(x_i, t_i)\}_{i=1}^{N} \subset D\times [0,T]$, where all $(x_i, t_i)$ are iid drawn from the uniform distribution on $D\times [0,T]$. 
For a fixed $\theta\in\Theta$, $\varphi\in\Phi_\epsilon$, $c\in\cC$ and for this data set $\S$ we can then define the \emph{training error}, 
\begin{equation}\label{eq:def-training-error}
    \begin{split}
        \Et(\theta, \S,\varphi, c) &= -\frac{T}{N} \sum_{i=1}^{N} \left(\abs{u_\theta(x_i,t_i)-c}\partial_t\varphi(x_i,t_i) + Q[u_\theta(x_i,t_i); c]\partial_x\varphi(x_i,t_i) \right)\\
        &\qquad + \frac{T}{N}\sum_{i=1}^{N} \abs{u_\theta(x_i,0)-u(x_i,0)} +\frac{T}{N}\sum_{i=1}^{N}\abs{u_\theta(0,t_i)-u_\theta(1,t_i)}.
    \end{split}
\end{equation}
During training, one then aims to obtain neural network parameters $\theta^*_\S$, a test function $\varphi^*_\S$ and a scalar $c^*_\S$ such that
\begin{equation}
     \Et(\theta^*_\S, \S,\varphi^*_\S, c^*_\S)\approx \min_{\theta\in\Theta} \max_{\varphi\in\Phi_\epsilon}\max_{c\in\cC} \Et(\theta, \S,\varphi, c),
\end{equation}
for some $\epsilon>0$. We call the resulting neural network $u^*:=u_{\theta^*_\S}$ a weak PINN (\emph{wPINN}). If the network has width $W$, depth $L$ and its weights are bounded by $R$, and the parameter $\epsilon>0$ is as in Theorem \ref{thm:stability-scl}, then the following theorem \cite[Theorem 3.9 and Corollary 3.10]{deryck2022wpinns}. 
\begin{theorem}\label{thm:generalization-scl}
Let $\cC$ be the essential range of $u$ and let $N, L, W\in\mathbb{N}$, $R\geq \max\{1, T, \abs{\cC}\}$ with $L\geq 2$ and $N\geq 3$. Moreover let $u_\theta:D\times [0,T]\to\mathbb{R}$, $\theta\in \Theta,$ be tanh neural networks with at most $L-1$ hidden layers, width at most $W$ and weights and biases bounded by $R$. Assume that $\Eg_G$ and $\Et$ are bounded by $B\geq 1$. It holds with a probability of at least $1-\delta$ that,
\begin{equation}
    \Eg_G(\theta^*_\S,\varphi^*_\S, c^*_\S) \leq \Et(\theta^*_\S, \S,\varphi^*_\S, c^*_\S) + \frac{3BLW}{\sqrt{N}}\sqrt{\ln(\frac{C\ln(1/\epsilon)WRN}{\epsilon^3 \delta B})}. 
\end{equation}
\end{theorem}
\end{example}

\section{Training}\label{sec:training}

Despite the generally affirmative answers to our first three questions (Q1-3) in the previous sections, significant problems have been identified with physics-informed machine learning. Arguably, the foremost problem lies with the \emph{training} these frameworks with (variants of) gradient descent methods. It has been increasingly observed that PINNs and their variants are \emph{slow, even infeasible}, to train even on certain model problems, with the training process either not converging or converging to unacceptably large loss values \citep{krishnapriyan_characterizing,Mos1,wang2021understanding,wang2022and}. These observations highlight the relevance of our fourth question. 

\begin{question}[Q4]
\emph{Can the training error $\Eg_T^*$ be made sufficiently sufficiently close to the true minimum of the loss function $\min_\theta \cJ(\theta, \cS)$?}
\end{question}

To answer this question, we must find out the reason behind the issues observed with training physics-informed machine learning algorithms? Empirical studies such as \cite{krishnapriyan_characterizing} attribute failure modes to the non-convex loss landscape, which is much more complex when compared to the loss landscape of supervised learning. Others like \cite{Mos1,Mos2} have implicated the well-known spectral bias \cite{SB1} of neural networks as being a cause for poor training whereas \cite{wang2021understanding,wang2021learning} used infinite-width NTK theory to propose that the subtle balance between the PDE residual and supervised components of the loss function could explain and possibly ameliorate training issues. Nevertheless, it is fair to say that there is still a paucity of principled analysis of the training process for gradient descent algorithms in the context of physics-informed machine learning. 

In what follows, we first demonstrate that when an affirmative answer to Question 4 is available (or assumed) it become possible to provide \emph{a priori error estimates} on the error of models learned by minimizing a physics-informed loss function (Section \ref{sec:glob-min}. Next, we derive precise conditions under which gradient descent for a physics-informed loss function can be approximated by a \emph{simplified gradient descent} algorithm, which amounts to the gradient descent update for a \emph{linearized} form of the training dynamics (Section \ref{sec:char-GD}). It will then turn out that the speed of convergence of the gradient descent is related to the \emph{condition number} of an operator, which in turn is composed of the \emph{Hermitian square} ($\cL^{\ast}\cL$) of the differential operator ($\cL$) of the underlying PDE and a \emph{kernel integral operator}, associated to the tangent kernel for the underlying model (Section \ref{sec:op-precond}). This analysis automatically suggests that \emph{preconditioning} the resulting operator is necessary to alleviate training issues for physics-informed machine learning (Section \ref{sec:precond}). 
Finally in Section \ref{sec:strategies}, we discuss how different preconditioning strategies can overcome training bottlenecks and also how existing techniques, proposed in the literature for improving training, can be viewed from this operator preconditioning perspective.

\subsection{Global minimum of the loss approximates PDE solution well}\label{sec:glob-min}

We demonstrate that the results of the previous sections can be used to prove \emph{a priori error estimates} for physics-informed learning, provided that one can find a global minimum of the physics-informed loss function \eqref{eq:minimization-problem}. 

We revisit the error decomposition \eqref{eq:error-decomposition} that was proposed in Sector \ref{sec:analysis}. There, it was shown that if one can answer Question 2 in an affirmative way, then it holds for any $\theta^*, \hat{\theta}\in \Theta$ that,
\begin{eqnarray}\label{eq:error-decomposition-2}
    \cE(\theta^*) &\leq& C \cE_G(\theta^*)^\alpha \\ &\leq& C \big({\Eg_G(\hat{\theta})} + 2 {\sup_{\theta\in\Theta}\abs{\Et(\theta,\S)-\Eg_G(\theta)}} + {\Et(\theta^*, \S)-\Et(\hat{\theta}, \S)}\big)^\alpha. 
\end{eqnarray}
For any $\epsilon>0$, we can now let $\hat{\theta}$ be the parameter corresponding to the model from Question 1, and hence we find that $\cE_G(\hat{\theta})<\epsilon$ if the model class is expressive enough. Next, we can use the results of Section \ref{sec:generalization} to deduce a lower limit on the size of the training set $\S$ (in terms of $\epsilon$) such that also ${\sup_{\theta\in\Theta}\abs{\Et(\theta,\S)-\Eg_G(\theta)}}<\epsilon$. Finally, if we assume that $\theta^* = \theta^*(\S)$ is the global minimizer of the loss function $\theta\mapsto\cJ(\theta, \S) = \cE_T(\theta, \S)$ \eqref{eq:minimization-problem}, then it must hold that $\Et(\theta^*, \S)\leq\Et(\hat{\theta}, \S)$ and hence $\Et(\theta^*, \S)-\Et(\hat{\theta}, \S)\leq 0$. As a result, we can infer from \eqref{eq:error-decomposition-2} that
\begin{eqnarray}\label{eq:3eps-bound}
    \cE^* \leq C(3\epsilon)^\alpha. 
\end{eqnarray}
Alternatively, we note that many of the a posteriori error bounds in Section \ref{sec:generalization} are of the form, 
\begin{eqnarray}\label{eq:general-a-posteriori}
     \textstyle\cE(\theta^*) \leq C \big(\Et(\theta^*, \S) + \epsilon\big)^\alpha, 
\end{eqnarray}
if the training set is large enough (depending on $\epsilon$). As before, it holds that 
\begin{eqnarray}\label{eq:a-priori-technique}
    \Et(\theta^*, \S) \leq \Et(\hat{\theta}, \S) \leq \cE_G(\hat{\theta}) + \sup_{\theta\in\Theta}\abs{\Et(\theta,\S)-\Eg_G(\theta)} \leq 2 \epsilon, 
\end{eqnarray}
such that in total we again find that $ \cE^* \leq C(3\epsilon)^\alpha$, in agreement with \eqref{eq:3eps-bound}. 

We first demonstrate this argument for the Navier-Stokes equations, and then show a similar result for the Poisson equation in which the curse of dimensionality is overcome.

\begin{example}[Navier-Stokes equations]
In Sections \ref{sec:approximation-error}, \ref{sec:stability} and \ref{sec:generalization} we have found affirmative answers to Questions Q1-Q3, and hence we should be apply to apply the above argument under the assumptions that an exact global minimizer to the loss function can be found. 

An a posteriori estimate as in \eqref{eq:general-a-posteriori} is provided by Theorem \ref{thm:generalization-ns}, but it is initially unclear whether this can also be used for an a priori estimate, given the dependence of the bound on the derivatives of the model. It turns out that this is possible for PINNs if the true solution is sufficiently smooth and the neural network and training set sufficiently large  \cite[Corollary 3.9]{deryck2021navierstokes}. 

\begin{corollary}
Let $\epsilon>0$, $T>0$, $d\in\mathbb{N}$, $k>6(3d+8) =: \gamma$, let $(u,p)\in H^k(\mathbb{T}^d\times [0,T])$ be the classical solution of the Navier-Stokes equation \eqref{eq:navier-stokes}, let the hypothesis space consists of PINNs for which the upper bound on the weights is at least $R \geq \epsilon^{-1/(k-\gamma)}\ln(1/\epsilon) $, the width is at least $W \geq \epsilon^{-(d+1)/(k-\gamma)}$ and the depth is at least $L \geq 3$,  let $(u_{\theta^*(\S)},p_{\theta^*(\S)})$ be the PINN that solves \eqref{eq:minimization-problem} and let the training set $\S$ satisfy  $N_t \geq \epsilon^{-d(1+\gamma/(k-\gamma))}$, $N_{\mathrm{int}} \geq \epsilon^{-2(d+1)(1+\gamma/(k-\gamma))}$ and $N_s \geq \epsilon^{-2d(1+\gamma/(k-\gamma))}$. It holds that
\begin{equation}
    \Vert u-u_{\theta^*(\S)}\Vert_{L^2(D\times [0,T])} = \bigO(\epsilon).
\end{equation}
\end{corollary}
\begin{proof}
    The results follows from combining results on the approximation error (Theorem \ref{thm:pinn-approx-ns}), the stability (Theorem \ref{thm:stability-ns}) and the generalization error (Theorem \ref{thm:generalization-ns}). 
\end{proof}
\end{example}

\begin{example}[Poisson's equation]
We revisit Poisson's equation with the variation residual, as previously considered in Example \ref{ex:barron-poisson} (Question 1) and Example \ref{ex:poisson-stab} (Question 2). Recall that if $f\in \cB^{0}(\Omega)$ (see Definition \ref{def:spectral-barron}) satisfies $\int_\Omega f(x)dx = 0$ than the corresponding unique solution $u$ to Poisson's equation $ - \Delta u = f$ with zero Neumann boundary conditions is $u\in \cB^2(\Omega)$. 

Now, assume that the training set $\S$ consists of $N$ iid randomly generated points in $\Omega$, and that one optimizes over a subset of shallow softplus neural networks of width $m$, corresponding to parameter space $\Theta^*\subset\Theta$ (a more precise definition can be found in \cite[Eq. 2.13]{lu2021priori}). As usual, we then define $u_{\theta^*(\S)}$ as the minimizer of the discretized energy residual $\cJ(\theta,\S)$. In this setting, one can prove the following a priori generalization result \cite[Remark 2.1]{lu2021priori}. 
\begin{theorem}
    Assume the setting defined above. If one sets $m = N^{\frac{1}{3}}$ then there exists a constant $C>0$ (independent of $m$ and $N$) such that,
    \begin{eqnarray}
        \E{\norm{u-u_{\theta^*(\S)}}^2_{H^1(\Omega)}} \leq C \frac{(\log(N))^2}{N^{\frac{1}{3}}}. 
    \end{eqnarray}
\end{theorem}
In the above theorem, we can see that the convergence rate is independent of $d$, such that the curse of dimensionality is mitigated. The constant $C$ might depend on the Barron norm of $u$ and its dependence on $d$ might therefore not exactly clear. 
\end{example}

\begin{example}[Scalar conservation law]
    Using Theorem \ref{thm:stability-scl} and the above bound on the generalization gap, one can prove the following rigorous upper bound \cite[Corollary 3.10]{deryck2022wpinns} on the total $L^1$-error of the weak PINN, which we denote as,
\begin{equation}\label{eq:def-total-error}
\begin{split}
        \cE(\theta) &= \int_D \abs{u_\theta(x,T)-u(x,T)} dx. 
    \end{split}
\end{equation}

\begin{corollary}\label{cor:generalization}
Assume the setting of Theorem \ref{thm:generalization-scl}. It holds with a probability of at least $1-\delta$ that,
\begin{equation}
\label{eq:erest}
\begin{split}
     \cE^* &\leq C \Bigg[{\Et^*}+ {\frac{3BdLW}{\sqrt{N}}\sqrt{\ln(\frac{C\ln(1/\epsilon)WRN}{\epsilon^3 \delta B})}} 
     %\\ &\qquad 
     + {(1+\norm{u^*}_{C^1})\ln(1/\epsilon)^3\epsilon}\Bigg],
\end{split}
\end{equation}
where $\cE^* := \cE(\theta^*_\S)$ and $\Et^* := \Et(\theta^*_\S, \S,\varphi^*_\S, c^*_\S), c^*_\S)$. 
\end{corollary}
Unfortunately this result is in its current form not strong enough to lead to an a priori generalization error. Using \eqref{eq:a-priori-technique}, one can easily ensure that $\Et^*$ is small, and one can verify that $B, R, N$ and $L$ depend (at most) polynomially on $\epsilon^{-1}$ such that the second term in \eqref{eq:erest} can be made small for large $N$. The problem lies in the third term, as most available upper bounds on $\norm{u^*}_{C^1}$ will be overestimates that grow with $\epsilon^{-1}$. 
\end{example}

\subsection{Characterization of gradient descent}\label{sec:char-GD}

As mentioned in the introduction of Section \ref{sec:training}, for some types of PDEs significant problems during the training of the model have been observed. To investigate this issue, we first consider a linearized version of the gradient descent (GD) update formula. 

First, recall that physics-informed machine learning boils down to minimizing the physics-informed loss, as constructed in Section \ref{sec:quad}, i.e. to find, 
\begin{equation}
\label{eq:opt}
\theta^\ast(\S) = \argmin\limits_{\theta \in \Theta} L(\theta),
\end{equation}
where in this section we focus on loss functions of the form, 
\begin{equation}
\label{eq:loss}
L(\theta) =  \underbrace{\frac{1}{2}\int\limits_{\Omega} \left| \cL u(x) - f(x)\right|^2 dx}_{R(\theta)} + \underbrace{\frac{\lambda}{2}\int\limits_{\partial \Omega} \left| u(x) - g(x)\right|^2 d\sigma(x)}_{B(\theta)}. 
\end{equation}
%Once such an (approximate) minimizer $\theta^\ast$ is obtained, one appeals to theoretical results such as those in \cite{MM1,deryck2021pinn,deryck2021navierstokes} to show that $u(\cdot\:;\theta^\ast)$ approximates the solution $u$ of the PDE (\ref{eq:PDE}) to high accuracy. Moreover, explicit error estimates in terms of the training error $L(\theta^\ast)$ can also be obtained \citep{MM1,deryck2021pinn}.  
As mentioned in Section \ref{sec:intro-optimization}, it is customary in machine learning \citep{DLbook} that the non-convex optimization problem \eqref{eq:opt} is solved with (variants of) a gradient descent algorithm which takes the following generic form,
\begin{equation}
\theta_{k+1} = \theta_k - \eta \nabla_\theta L(\theta_k),
    \label{eq:GD}
\end{equation}    
with descent steps $k > 0$, learning rate $\eta > 0$, loss $L$ and the initialization $\theta_0$ chosen randomly, following Section \ref{sec:initialization}. 

We want to investigate the \emph{rate of convergence} to ascertain the computational cost of training. 
As the loss $L$ (\ref{eq:GD}) is non-convex, it is  hard to rigorously analyze the training process in complete generality. One needs to make certain assumptions on (\ref{eq:GD}) to make the problem tractable. To this end, we fix the step $k$ in (\ref{eq:GD}) and start with the following Taylor expansion,
\begin{align}
\label{eq:texp}
    u_{}(x;\theta_k) = u_{}(x;\theta_0) + \nabla_{\theta} u_{}(x;\theta_0)^\top (\theta_k-\theta_0) + \tfrac{1}{2}(\theta_k-\theta_0)^\top H_k(x) (\theta_k-\theta_0).
\end{align}
Here, $H_k(x) := \textnormal{Hess}_\theta(u_{}(x;\tau_k \theta_0 + (1-\tau_k)\theta_k)$ is the Hessian of $u(\cdot\:,\theta)$ evaluated at intermediate values, with $0\leq \tau_k\leq 1$. 
Now introducing the notation $\phi_i(x) = \partial_{\theta_i} u(x;\theta_0)$, we define the matrix $\sA \in \R^{n\times n}$ and the vector $\sC,  \in \R^n$ as,
\begin{equation}
\label{eq:defn}
\begin{aligned}
\sA_{i,j} &=  \langle \cL \phi_i, \cL \phi_j \rangle_{L^2(\Omega)} + \lambda \langle  \phi_i, \phi_j \rangle_{L^2(\partial \Omega)},\\
\sC_{i} &=  \langle \cL u_{\theta_0}-f, \cL \phi_i \rangle_{L^2(\Omega)} +  \lambda \langle u_{\theta_0}-u, \phi_i \rangle_{L^2(\partial \Omega)}
\end{aligned}
\end{equation}
Substituting the above formulas in the GD algorithm (\ref{eq:GD}), we can rewrite it identically as,
\begin{align}\label{eq:GD-with-error}
    \theta_{k+1} = \theta_k -\eta \nabla_\theta L(\theta_k) = (I-\eta \sA)\theta_k + \eta(\sA \theta_0 + \sC) + \eta \epsilon_k,
\end{align}
where $\epsilon_k$ is an error term that collects all terms that depend on the Hessians $H_k$ and $\cL H_k$, and is explicitly given by,
\begin{equation}
\label{eq:defn-app}
\begin{aligned}
2\epsilon_k =&  \langle \cL u_{\theta_k}-f, \cL H_k(\theta_k-\theta_0) \rangle_{L^2(\Omega)} +  \lambda \langle u_{\theta_k}-u, H_k(\theta_k-\theta_0) \rangle_{L^2(\partial \Omega)} \\
&+ \langle (\theta_k-\theta_0)^\top \cL H_k(\theta_k-\theta_0), \cL \nabla_\theta u_{\theta_0}\rangle_{L^2(\Omega)} \\&+ \lambda \langle (\theta_k-\theta_0)^\top  H_k(\theta_k-\theta_0),  \nabla_\theta u_{\theta_0}\rangle_{L^2(\partial \Omega)}. 
\end{aligned}
\end{equation}
From this characterization of gradient descent, we clearly see that (\ref{eq:GD}) is related to a \emph{simplified} version of gradient descent given by,
\begin{equation}
\label{eq:SGD}
        \Tilde{\theta}_{k+1} = (I-\eta \sA)\Tilde{\theta}_k + \eta(\sA \Tilde{\theta}_0 + \sC), \quad \Tilde{\theta}_0 = \theta_0,
\end{equation}
modulo the \emph{error} term $\epsilon_k$ defined in (\ref{eq:defn-app}).

The following lemma \cite{deryck2023operator} shows that this simplified GD dynamics (\ref{eq:SGD}) approximates the full GD dynamics (\ref{eq:GD}) to desired accuracy as long as the error term $\epsilon_k$ is small. 
% We find that the accuracy also depends on the condition number $\kappa(\sA)$ of the matrix $\sA$, which is defined as,
% %  \begin{equation}
% %  \label{eq:cond}
% %  \kappa(\sA) = {\lambda_{\max}(\sA)}/{\lambda_{\min}(\sA)},
% %  \end{equation}
% % where $\lambda_{\min}(\sA) := \min_j |{\lambda_j(\sA)}|$, $\lambda_{\max}(\sA) := \max_j |{\lambda_j(\sA)}|$ and $\{\lambda_j(\sA)\}_j$ the eigenvalues of $\sA$. 
\begin{lemma}
\label{lem:1}
    Let $\delta>0$ be such that $\max_k \|\epsilon_k\|_2 \leq \delta$. Let 
    If $\sA$ is invertible and $\eta=c/\lambda_{\max}(\sA)$ for some $0<c<1$ then it holds for any $k\in\N$ that,
    \begin{align}
        \|\theta_k - \Tilde{\theta}_k\|_2 \leq { \kappa(\sA)\delta}/{c},
    \end{align}
  with $\lambda_{\min} := \min_j \abs{\lambda_j(\sA)}$, $\lambda_{\max} := \max_j \abs{\lambda_j(\sA)}$ and the \emph{condition number}, \begin{equation}
 \label{eq:cond}
 \kappa(\sA) = {\lambda_{\max}(\sA)}/{\lambda_{\min}(\sA)}.
 \end{equation}

\end{lemma}

The key assumption in Lemma \ref{lem:1} is the smallness of the error term $\epsilon_k$ (\ref{eq:GD-with-error}) for all $k$. This is trivially satisfied for linear models $u_\theta(x) = \sum_k \theta_k \phi_k$ as $\epsilon_k = 0$ for all $k$ in this case. 
From the definition of $\epsilon_k$ \eqref{eq:defn-app}, we see that a more general sufficient condition for ensuring this smallness is to ensure that the Hessians of $u_\theta$ and $\cL u_\theta$ (resp. $H_k$ and $\cL H_k$ in (\ref{eq:texp})) are small during training. This amounts to requiring \emph{approximate linearity} of the parametric function $u(\cdot\:;\theta)$ near the initial value $\theta_0$ of the parameter $\theta$. 
For any differentiable parametrized function $f_\theta$, its linearity is equivalent to the constancy of the associated \emph{tangent kernel} \cite{liu2020linearity},
\begin{eqnarray}
    \Theta[f_\theta](x,y) := \nabla_\theta f_\theta(x) ^\top \nabla_\theta f_\theta(y).
\end{eqnarray}
Hence, it follows that if the tangent kernel associated to $u_\theta$ and $\cL u_\theta$ is (approximately) constant along the optimization path, then the error term $\epsilon_k$ will be small. 

For neural networks this entails that the \emph{neural tangent kernels} (NTK) $\Theta[u_{\theta}]$ and $\Theta[\cL u_{\theta}]$ stay approximately constant along the optimization path. The following informal lemma \cite{deryck2023operator}, based on \cite{wang2022and}, confirms that this is indeed the case for wide enough neural networks.

\begin{lemma}\label{lem:NTK-constant}
For a neural network $u_\theta$ with one hidden layer of width $m$ and a linear differential operator $\cL$ it holds that,
 $\lim_{m\to\infty} \Theta[u_{\theta_k}] = \lim_{m\to\infty} \Theta[u_{\theta_0}]$ and $\lim_{m\to\infty} \Theta[\cL u_{\theta_k}] = \lim_{m\to\infty} \Theta[\cL u_{\theta_0}]$ for all $k$. 
Consequently, the error term $\epsilon_k$ (\ref{eq:GD-with-error}) is small for wide neural networks, $ \lim_{m\to\infty} \max_k \|\epsilon_k\|_2 =0$. 
\end{lemma}

Now, given the much simpler structure of (\ref{eq:SGD}), when compared to (\ref{eq:GD}), one can study the corresponding gradient descent dynamics explicitly and obtain the following convergence theorem \cite{deryck2023operator}.
\begin{theorem}\label{thm:convergence-simplified-gd}
    Let $\sA$ in (\ref{eq:SGD}) be invertible with condition number $\kappa(\sA)$ (\ref{eq:cond}) and let $0<c<1$. Set $\eta = c/\lambda_{\max}(\sA)$ and $\vartheta = \theta_0 + \sA^{-1}\sC$. It holds for any $k\in\N$ that,
    \begin{align}
    \label{eq:crate}
        \|\Tilde{\theta}_k - \vartheta\|_2 \leq \left(1-{c}/{\kappa(\sA)}\right)^{k}\| \theta_0 - \vartheta \|_2.
    \end{align}
\end{theorem}
An immediate consequence of the quantitative convergence rate (\ref{eq:crate}) is as follows: to obtain an error of size $\varepsilon$, i.e., $ \|\Tilde{\theta}_k - \vartheta\|_2 \leq \varepsilon$, we can readily calculate the number of GD steps $N(\varepsilon)$ as,
\begin{equation}
\label{eq:nsteps}
N(\varepsilon) = \ln(\varepsilon/\| \theta_0 - \vartheta \|_2) / \ln(1-c/\kappa(\sA)) = O\left(\kappa(\sA)\ln \tfrac{1}{\epsilon}\right).
\end{equation}
Hence, for a fixed value $c$, large values of the condition number $\kappa(\sA)$ will severely impede convergence of the simplified gradient descent (\ref{eq:SGD}) by requiring a much larger number of steps.

\subsection{Training and operator preconditioning}\label{sec:op-precond}

So far, we have established that, under suitable assumptions, the rate of convergence of the gradient descent algorithm for physics-informed machine learning boils down to the \emph{conditioning} of the matrix $\sA$ (\ref{eq:defn}). However, at first sight, this matrix is not very intuitive and we want to relate it to the differential operator $\cL$ from the underlying PDE (\ref{eq:pde}). To this end, we first introduce the so-called \emph{Hermitian square} $\cA$, given by,
%$\cA = \cL^*\cL$,
\begin{equation}
\label{eq:hsq}
\cA = \cL^*\cL, 
\end{equation}
in the sense of operators, where $\cL^\ast$ is the \emph{adjoint operator} for the differential operator $\cL$. Note that this definition implicitly assumes that the adjoint $\cL^{\ast}$ exists and the Hermitian square operator $\cA$ is defined on an appropriate function space. As an example, consider as differential operator the Laplacian, i.e., $\cL u = -\Delta u$, defined for instance on $u \in H^1(\Omega)$, then the corresponding Hermitian square is $\cA u = \Delta ^2 u$, identified as the \emph{bi-Laplacian} that is well defined on $u \in H^2(\Omega)$. 

Next for notational simplicity, we consider a loss function that only consists of the integral of the squared PDE residual \eqref{eq:PDE-residual} and omit boundary terms in the following. Let $\cH$ be the span of the functions $\phi_k := \partial_{\theta_k} u(\cdot; \theta_0)$. Define the maps $T:\R^n \to \cH, v \to \sum_{k=1}^n v_k\phi_k$ and $T^*: L^2(\Omega) \to \R^n; f \to \{\langle\phi_k, f\rangle\}_{k=1, \dots, n}$. We define the following scalar product on $L^2(\Omega)$, 
\begin{align}
    \langle f, g \rangle_{\cH} := \langle f, TT^* g \rangle_{L^2(\Omega)} = \langle T^* f, T^* g \rangle_{\mathbb{R}^n}.
\end{align}
Note that the maps $T,T^{\ast}$ provide a correspondence between the continuous space ($L^2$) and discrete space ($\cH$) spanned by the functions $\phi_k$. This continuous-discrete correspondence allows us to relate the conditioning of the matrix $\sA$ in (\ref{eq:defn}) to the conditioning of the Hermitian square operator $\cA = \cL^{\ast} \cL$ through the following theorem \cite{deryck2023operator}. 
\begin{theorem}\label{thm:connection-matrix-operator}
    It holds for the operator $\cA\circ TT^*:L^2(\Omega)\to L^2(\Omega)$ that $\kappa(\sA) \geq \kappa(\cA \circ T T^*)$. Moreover, if the Gram matrix $\langle \phi, \phi \rangle_{\cH}$ is invertible then 
    %the restricted operator $\cA\circ TT^*:\cH\to L^2(\Omega)$, 
    equality holds, i.e.,  $\kappa(\sA) = \kappa(\cA \circ T T^*)$.
\end{theorem}

Thus, we show that the conditioning of the matrix $\sA$ that determines the speed of convergence of the simplified gradient descent algorithm (\ref{eq:SGD}) for physics-informed machine learning is intimately tied with the conditioning of the operator $\cA \circ T T^*$. This operator, in turn, composes the Hermitian square of the underlying differential operator of the PDE (\ref{eq:pde}), with the so-called \emph{Kernel Integral operator} $TT^{\ast}$, associated with the (neural) tangent kernel $\Theta[u_\theta]$. Theorem \ref{thm:connection-matrix-operator} implies in particular that if the operator $\cA \circ T T^*$ is ill-conditioned, then the matrix $\sA$ is ill-conditioned and the gradient descent algorithm (\ref{eq:SGD}) for physics-informed machine learning will converge very slowly.

\begin{remark}\label{rem:connection-matrix}
One can readily generalize Theorem \ref{thm:connection-matrix-operator} to the setting with boundary conditions i.e., with $\lambda > 0$ in the loss (\ref{eq:loss}). In this case one can prove for the operator $\cA = \1_{\mathring{\Omega}} \cdot \cL^*\cL + \lambda \1_{{\partial\Omega}} \cdot \Id$,
and its corresponding matrix $\sA$ (as in (\ref{eq:defn})) that $\kappa(\sA) \geq \kappa(\cA \circ T T^*)$, where equality holds if the relevant Gram matrix is invertible. More details can be found in \cite[Appendix A.6]{deryck2023operator}.
\end{remark} 

It is instructive to compare physics-informed machine learning with standard supervised learning through the prism of the analysis presented here. It is straightforward to see that for supervised learning, i.e., when the physics-informed loss is replaced with the supervised loss $\tfrac{1}{2}\|u-u_{\theta}\|^2_{L^2(\Omega}$ by simply setting $\cL = \Id$, the corresponding operator in Theorem \ref{thm:connection-matrix-operator} is simply the kernel integral operator $TT^{\ast}$, associated with the tangent kernel as $\cA = \Id$. Thus, the complexity in training physics-informed machine learning models is entirely due to the spectral properties of the Hermitian square $\cA$ of the underlying differential operator $\cL$.

\subsection{Preconditioning to improve training in physics-informed machine learning}\label{sec:precond}

Having established in the previous section that, under suitable assumptions, the speed of training physics-informed machine learning models depends on the condition number of the operator $\cA \circ TT^{\ast}$ or, equivalently the matrix $\sA$ (\ref{eq:defn}), we now investigate whether this operator is ill-conditioned and if so, how can we better condition it by reducing the condition number. The fact that $\cA \circ TT^{\ast}$ (equiv. $\sA$) is very poorly conditioned for most PDEs of practical interest will be demonstrated both theoretically and empirically below. This makes \emph{preconditioning}, i.e., strategies to improve (reduce) the conditioning of  the underlying operator (matrix), a key component in improving training for physics-informed machine learning models. 

\subsubsection{General framework for preconditioning.} 
Intuitively, reducing the condition number of the underlying operator $\cA \circ TT^{\ast}$ amounts to finding new maps $\Tilde{T}, \Tilde{T}^*$ for which the kernel integral operator $\Tilde{T} \Tilde{T}^*$ is approximately equal to $ \cA^{-1}$, i.e., choosing the architecture and initialization of the parametrized model $u_\theta$ such that the associated kernel integral operator $\Tilde{T} \Tilde{T}^*$ is an (approximate) Green's function for the Hermitian square $\cA$ of the differential operator $\cL$. 
For an operator $\cA$ with well-defined eigenvectors $\psi_k$ and eigenvalues $\omega_k$, the ideal case  $\Tilde{T} \Tilde{T}^* = \cA^{-1}$ is realized when $\Tilde{T} \Tilde{T}^* \phi_k = \tfrac{1}{\omega_k} \psi_k$. This can be achieved by transforming $\phi$ linearly with a (positive definite) matrix $\sP$ such that $(\sP^\top \phi)_k = \tfrac{1}{\omega_k} \psi_k$, which corresponds to the change of variables $\cP u_\theta := u_{\sP \theta}$. 
%This will induce new maps  and a new associated space $\Tilde{\cH}$ which is equal to $\cH$ is $\sP$ is positive-definite. 
Assuming the invertibility of $\langle \phi, \phi\rangle_{\cH}$,  Theorem \ref{thm:connection-matrix-operator}  then shows that $\kappa(\cA \circ \Tilde{T} \Tilde{T}^*) = \kappa(\Tilde{\sA})$ for a new matrix $\Tilde{\sA}$, which can be computed as,
\begin{align}
     \Tilde{\sA} := \langle \cL \nabla_\theta u_{\sP\theta_0}, \cL  \nabla_\theta u_{\sP\theta_0}\rangle_{L^2(\Omega)} =   \langle \cL  \sP^\top \nabla_\theta u_{\theta_0}, \cL \sP^\top \nabla_\theta u_{\theta_0}\rangle_{L^2(\Omega)} 
     = \sP^\top \sA \sP. 
\end{align}

This implies a general approach for preconditioning, namely linearly transforming the parameters of the model, i.e. considering $\cP u_\theta := u_{\sP \theta}$ instead of $u_\theta$, which corresponds to replacing the matrix $\sA$ by its preconditioned variant $\Tilde{\sA} = \sP^\top \sA \sP$. 
The new simplified GD update rule is then %${\theta_{k+1} = \theta_k -\eta \Tilde{\sA} (\theta_k-\theta_0) +\Tilde{\sC}}$. 
\begin{align}
    \theta_{k+1} = \theta_k -\eta \Tilde{\sA} (\theta_k-\theta_0) +\sC.
\end{align} 
Hence, finding $\Tilde{T} \Tilde{T}^* \approx \cA^{-1}$, which is the aim of preconditioning, reduces to constructing a matrix $\sP$ such that $1 \approx \kappa(\Tilde{\sA}) \ll \kappa(\sA)$. We emphasize that $\Tilde{T} \Tilde{T}^*$ need not serve as the exact inverse of $\mathcal{A}$; even an approximate inverse can lead to significant performance improvements, this is the foundational principle of preconditioning.

Moreover, performing gradient descent using the transformed parameters $\hat{\theta}_k := \sP\theta_k$ yields, 
\begin{align}
    \hat{\theta}_{k+1} = \sP\theta_{k+1} = \sP\theta_k -\eta \sP\sP^\top\nabla_\theta L(\sP\theta_k) = \hat{\theta}_k - \eta \sP\sP^\top \nabla_\theta L(\hat{\theta_k}). 
\end{align}
Given that any positive definite matrix can be written as $\sP\sP^\top$, this shows that linearly transforming the parameters is equivalent to preconditioning the gradient of the loss by multiplying with a positive definite matrix. Hence, parameter transformations are all that are needed in this context. 
% We would like to emphasize that this method is general  the new operator $\Tilde{T} \Tilde{T}^*$ need

\subsubsection{Preconditioning {linear} physics-informed machine learning models.}
We know rigorously analyze the effect of preconditioning linear parametrized models 
of the form  $u_\theta(x) = \sum_k \theta _k \phi_k(x)$, where $\phi_1, \ldots, \phi_n$ are any smooth functions, as introduced in Section \ref{sec:linear-models}. A corresponding preconditioned model, as explained above, would have the form $\Tilde{u}_\theta(x) = \sum_k (\sP \theta)_k \phi_k(x)$, where $\sP\in\R^{n\times n}$ is the preconditioner. We motivate the choice of this 
preconditioner with a simple, yet widely used example.

% \begin{figure}[htbp]
%     \centering
%     \includegraphics[scale=0.5]{figures/poisson_condn_vs_epochs_notitle.pdf}
%     \caption{Poisson equation with Fourier features. \emph{Left:} Optimal condition number vs. Number of Fourier features. \emph{Right:} Training for the unpreconditioned and preconditioned Fourier features.}
%      \label{fig:1}
% \end{figure}

Our differential operator is the one-dimensional Laplacian $\cL = \frac{d^2}{dx^2}$, defined on the domain $(-\pi,\pi)$, for simplicity with periodic zero boundary conditions. Consequently, the corresponding PDE is the Poisson equation (Example \ref{def:poisson}). As the machine learning model, we choose
$u_\theta(x) = \sum_{k=-K}^K \theta_k\phi_k(x)$ with,
\begin{eqnarray}
    \phi_0(x)= \tfrac{1}{\sqrt{2\pi}}, \quad \phi_{-k}(x)=\tfrac{1}{\sqrt{\pi}}\cos(kx), \quad \phi_k(x) = \tfrac{1}{\sqrt{\pi}}\sin(kx),
\end{eqnarray}
for $1\leq k\leq K$. This model corresponds to the widely used learnable \emph{Fourier Features} in the machine learning literature \citep{FF} or \emph{spectral methods} in numerical analysis \citep{Specbook}. We can readily verify that the resulting matrix $\sA$ (\ref{eq:defn}) is given by 
\begin{eqnarray}
    \sA = \sD + \lambda vv^\top,
\end{eqnarray}
where  $\sD$ is a diagonal matrix with $\sD_{kk} = k^4$ and $v:=\phi(\pi)$ i.e., that $v$ is a vector with $v_{-k} = (-1)^k/\sqrt{\pi}$, $v_0=1/\sqrt{2\pi}$, $v_k = 0$ for $1\leq k\leq K$. 
%Consequently, we observe that the condition $\kappa(\sA) \sim K^4$, resulting in very poor conditioning as the frequencies are increased. 
Preconditioning solely based on $\cL^*\cL$ would correspond to finding a matrix $\sP$ such that $\sP\sD\sP^\top = \Id$. However, given that $\sD_{00}=0$, this is not possible. We therefore set $\sP_{kk} = 1/k^2$ for $k\neq 0$ and $\sP_{00} = \gamma \in \R$. The preconditioned matrix is therefore
\begin{align}
\label{eq:pcm}
    \Tilde{\sA}(\lambda,\gamma) = \sP\sD\sP^\top + \lambda \sP v(\sP v)^\top. 
\end{align}
The conditioning of the unpreconditioned and preconditioned matrices considered above are summarized in the following theorem \cite{deryck2023operator}.
\begin{theorem}\label{thm:preconditioned-matrix}
The following statements hold for all $K\in\N$:
\begin{enumerate}
    \item The condition number of the unpreconditioned matrix above satisfies, 
    \begin{eqnarray}
    \kappa(\sA(\lambda)) \geq K^4.
    \end{eqnarray}
    \item There exists a constant $C(\lambda,\gamma)>0$ that is independent of $K$ such that,
    \begin{eqnarray}
        \kappa(\Tilde{\sA}(\lambda,\gamma) \leq C.
    \end{eqnarray}
    \item It holds that $\kappa(\Tilde{\sA}(2\pi/\gamma^2, \gamma)) = 1+ O(1/\gamma)$ and hence,
    \begin{eqnarray}
        \lim_{\gamma\to +\infty} \kappa(\Tilde{\sA}(2\pi/\gamma^2, \gamma)) = 1.
    \end{eqnarray}
\end{enumerate}
\end{theorem}
We observe from Theorem \ref{thm:preconditioned-matrix}
that (i) the matrix $\sA$, which governs gradient descent dynamics for approximating the Poisson equation with learnable Fourier features, is very poorly conditioned and (ii) we can (optimally) precondition it by \emph{rescaling} the Fourier features based on the eigenvalues of the underlying differential operator (or its Hermitian square). 

\begin{figure}[htbp]
    \centering
    \includegraphics[width=\textwidth]{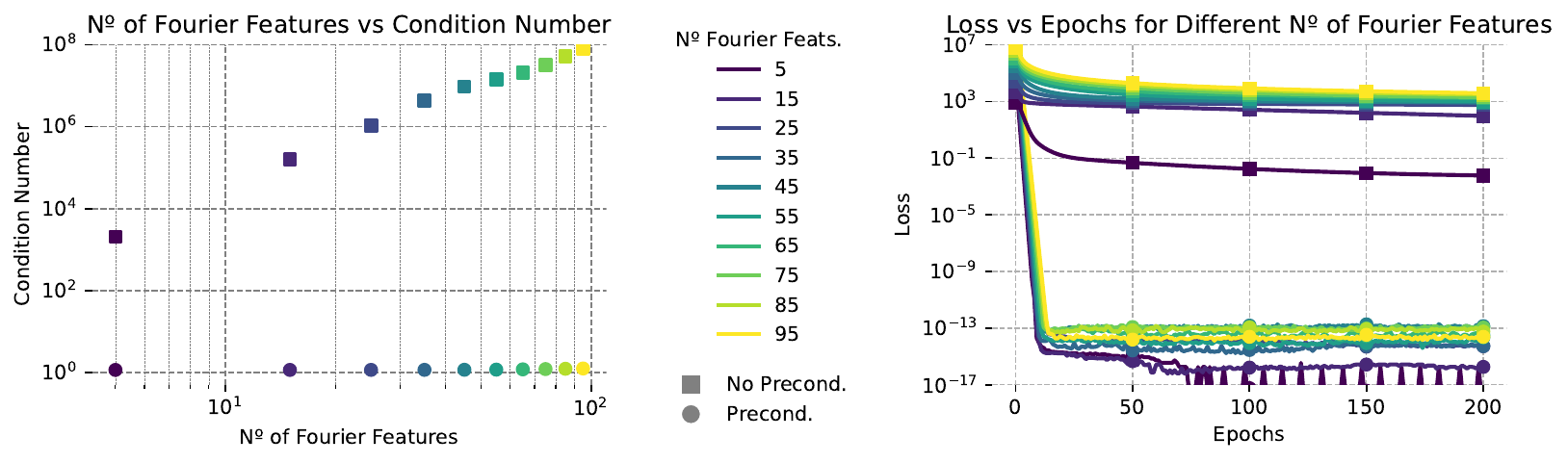}
    \caption{Poisson equation with Fourier features. \emph{Left:} Optimal condition number vs. Number of Fourier features. \emph{Right:} Training for the unpreconditioned and preconditioned Fourier features. Figure from \citet{deryck2023operator}. }
     \label{fig:1}
\end{figure}

These conclusions are also observed empirically. 
In Figure \ref{fig:1} (left), we plot the condition number of the matrix $\sA$, minimized over $\lambda$, as a function of maximum frequency $K$ and verify that this condition number increases as $K^4$, as predicted by Theorem \ref{thm:preconditioned-matrix}. Consequently as shown in Figure \ref{fig:1} (right), where we plot the loss function in terms of increasing training epochs, the model is very hard to train with large losses (particularly for higher values of $K$), showing a very slow decay of the loss function as the number of frequencies is increased. On the other hand, in Figure \ref{fig:1} (left), we also show that the condition number (minimized over $\lambda$) of the preconditioned matrix (\ref{eq:pcm}) remains constant with increasing frequency and is very close to the optimal value of $1$, verifying Theorem \ref{thm:preconditioned-matrix}. As a result, we observe from Figure \ref{fig:1} (right) that the loss in the preconditioned case decays exponentially fast as the number of epochs is increased. This decay is independent of the maximum frequency of the model. The results demonstrate that the preconditioned version of the Fourier features model can learn the solution of the Poisson equation efficiently, in contrast to the failure of the unpreconditioned model to do so. Entirely analogous results are obtained with the Helmholtz equation \cite{deryck2023operator}. 

\begin{figure}[htbp]
    \centering
    \includegraphics[width=\textwidth]{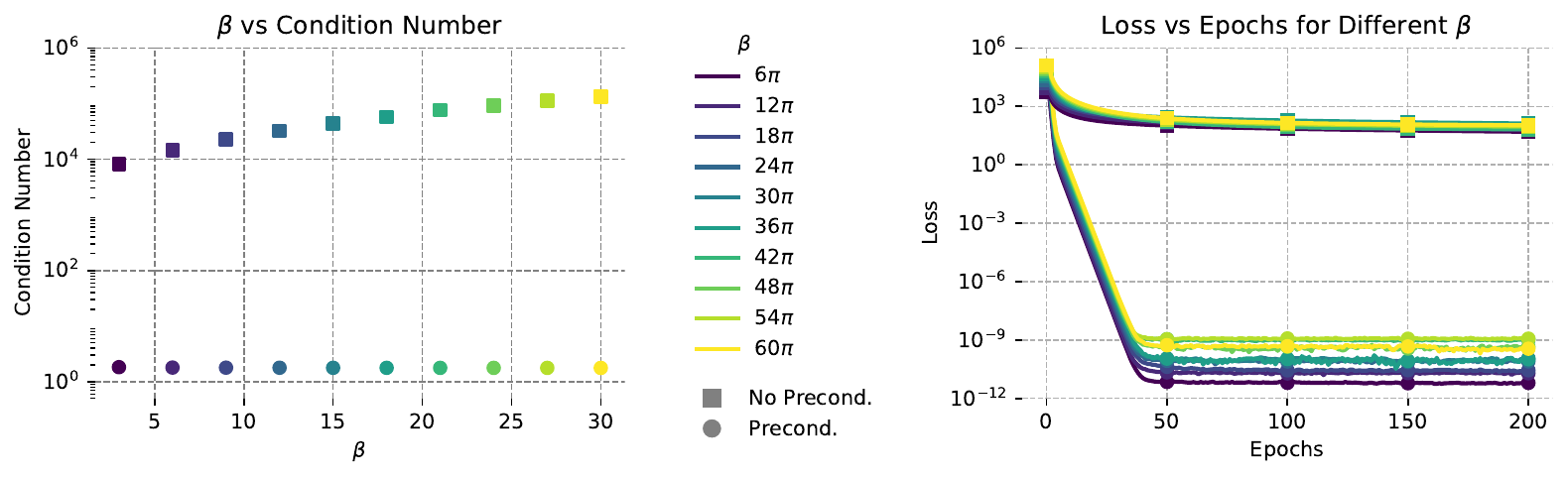}
    \caption{Linear advection equation with Fourier features. \emph{Left:} Optimal condition number vs. $\beta$. \emph{Right:} Training for the unpreconditioned and preconditioned Fourier features. Figure from \citet{deryck2023operator}. }
     \label{fig:2}
\end{figure}

As a different example, we consider the linear advection equation $u_t + \beta u_x =0$ on the one-dimensional spatial domain $x \in [0,2\pi]$ and with $2\pi$-periodic solutions in time with $t\in [0, 1]$. As in \cite{krishnapriyan_characterizing}, our focus in this case is to study how physics-informed machine learning models train when the advection speed $\beta > 0$ is increased. To empirically evaluate this example, we again choose learnable time-dependent Fourier features as the model and precondition the resulting matrix $\sA$ (\ref{eq:defn}). In Figure \ref{fig:2} (left), we see that the condition number of $\sA(\beta) \sim \beta^2$ grows quadratically with advection speed. On the other hand, the condition number of the preconditioned model remains constant. Consequently as shown in Figure \ref{fig:2} (right), the unpreconditioned model trains very slowly (particularly for increasing values of the advection speed $\beta$) with losses remaining high despite being trained for a large number of epochs. In complete contrast, the preconditioned model trains very fast, irrespective of the values of the advection speed $\beta$. %Further details including visualizations of the resulting solutions and a comparison with a MLP are presented in \cite{deryck2023operator}. In particular, we show that the preconditioned Fourier model readily outperforms the MLP. 

\begin{figure}
\centering
\begin{minipage}{.5\textwidth}
  \centering
  \includegraphics[width=\linewidth]{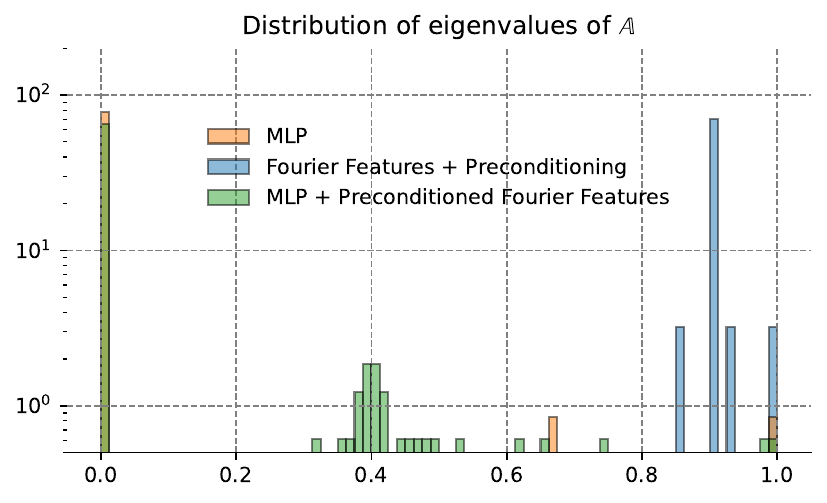}
\end{minipage}%
\begin{minipage}{.5\textwidth}
  \centering
  \includegraphics[width=\linewidth]{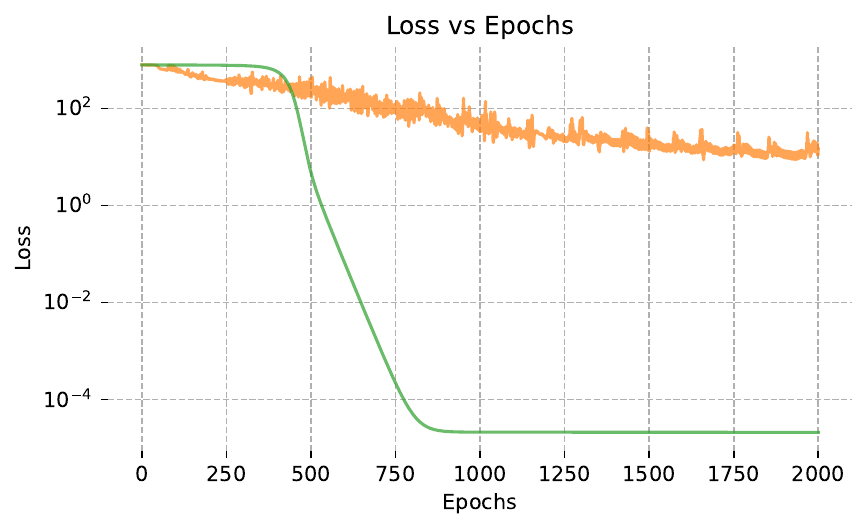}
\end{minipage}
\caption{Poisson equation with MLPs. \emph{Left:} Histogram of normalized spectrum (eigenvalues multiplied with learning rate). \emph{Right:} Loss vs. number of epochs. Figure from \citet{deryck2023operator}. }
        \label{fig:3}
\end{figure}

\subsubsection{Preconditioning nonlinear physics-informed machine learning models.} 
Next, \citet{deryck2023operator} have investigated the conditioning of nonlinear models (Section \ref{sec:nonlinear-models}) by considering the Poisson equation on $(-\pi,\pi)$ and learning its solution with neural networks of the form $u_\theta(x) = \Phi_\theta(x)$. It was observed that  most of the eigenvalues of the resulting matrix $\sA$ \eqref{eq:defn} are clustered near zero. Consequently, it is difficult to analyze the condition number per se. However, they also observed that there are only a couple of large non-zero eigenvalues of $\sA$ indicating a very uneven spread of the spectrum. It is well-known in classical numerical analysis \citep{TREF} that such spread-out spectra are very poorly conditioned and this will impede training with gradient descent. This is corroborated in Figure \ref{fig:3} (right) where the physics-informed MLP trains very slowly. Moreover, it turns out that preconditioning also localizes the spectrum \citep{TREF}. This is attested in Figure \ref{fig:3} (left) where we see the localized spectrum of the preconditioned Fourier features model considered previously, which is also correlated with its fast training (Figure \ref{fig:1} (right)).

However, preconditioning $\sA$ for nonlinear models such as neural networks can, in general, be hard. Here, we consider a 
simple strategy by coupling the MLP with Fourier features $\phi_k$ defined above, i.e., setting $u_\theta = \Phi_\theta\left(\sum_{k} \alpha_k \phi_k\right)$. 
%Then, we choose the coefficients $\alpha_k$ in order to rescale the differential operator (and in turn, the overall parameters). 
Intuitively, one must carefully choose $\alpha_k$ to control
$\frac{d^{2n}}{dx^{2n}}u_\theta(x)$ as it will include terms such as $(\sum_{k\neq 0} \alpha_k (-k)^{2n}\phi_k)\Phi_\theta^{(2n)}\left(\sum_{k} \alpha_k \phi_k\right)$. Hence,
%for any differentiable function $f$ and a constant $a\in\R$ it holds that $\frac{d^k}{dx^k}f(ax) = a^k f^{(k)}(x)$. 
such rescaling could better condition the Hermitian square of the differential operator. To test this for Poisson's equation, we choose $\alpha_k = 1/k^2$ (for $k \neq 0$) in this FF-MLP model and present the eigenvalues in Figure \ref{fig:3} (left) to observe that although there are still quite a few (near)-zero eigenvalues, the number of non-zero eigenvalues is significantly increased leading to a much more even spread in the spectrum, when compared to the unpreconditioned MLP case. This possibly accounts for the fact that the resulting loss function decays much more rapidly, as shown in Figure \ref{fig:3} (right), when compared to unpreconditioned MLP.  

\subsection{Strategies for improving training in physics-informed machine learning}\label{sec:strategies}

Given the difficulties encountered in training physics-informed machine learning models, several ad-hoc strategies have been proposed in the recent literature to improve training. It turns out that many of these strategies can also be interpreted in terms of preconditioning. 

\subsubsection{Choice of $\lambda$.} The parameter $\lambda$ in a physics-informed loss such as (\ref{eq:epil-td}) plays a crucial role as it balances the relative contributions of the physics-informed loss $R$ and the supervised loss at the boundary $B$. Given the analysis of the previous sections, it is natural to suggest that this parameter should be chosen as,
\begin{eqnarray}\label{eq:optimal-lambda}
    \lambda^* := \min_\lambda \kappa(\sA(\lambda))
\end{eqnarray}
in order to obtain the smallest condition number of $\sA$ and accelerate convergence. Another strategy is suggested by \citet{wang2021understanding}, who suggest a learning rate annealing algorithm to iteratively update $\lambda$ throughout training as follows,
\begin{equation}\label{eq:lambda-a}
    \lambda^*_a := \frac{\max_k \abs{(\nabla_\theta R(\theta))_k}}{(2K+1)^{-1}\sum_k \abs{(\nabla_\theta B(\theta))_k}}. 
\end{equation}
A similar work is \citet{wang2022and}, where it is proposed to set
\begin{align}\label{eq:lambda-b}
    \lambda^*_b := \frac{Tr(K_{rr}(n))}{Tr(K_{uu}(n))} \approx \frac{\int_\Omega \nabla_\theta \cL u_\theta^\top \nabla_\theta \cL u_\theta}{\int_{\partial\Omega} \nabla_\theta  u_\theta^\top \nabla_\theta  u_\theta},% = \frac{2\pi \sum_{k=1}^K k^4}{K+1} \sim K^4, 
\end{align}
where we refer to \citet{wang2022and} for the exact definition of $K_{rr}$ and $K_{uu}$. 

In the below example, we compare $\lambda^*$ \eqref{eq:optimal-lambda}, $\lambda^*_a$ \eqref{eq:lambda-a} and $\lambda^*_b$ \eqref{eq:lambda-b} for the Poisson equation. 

\begin{example}[Poisson's equation]
We compare the three proposed strategies above for the 1D Poisson equation and a linear model using Fourier features up to frequency $K$. 
First of all, the authors of \cite{deryck2023operator} computed that $\lambda^*\sim K^2$. 
Next, to calculate $\lambda^*_a$ we observe that $(\nabla_\theta R(\theta))_k = k^4 \theta_k$ for $k\neq \ell$ and $(\nabla_\theta R(\theta))_\ell = \ell^4 (\theta_k-1)$. We find that $\max_k \abs{(\nabla_\theta R(\theta))_k} \sim K^4$ at initialization. Next we calculate that (given that at initialization it holds that $\theta_m\sim \cN(0,1)$ iid),
\begin{align}
    \mathbb{E}\sum_k \abs{(\nabla_\theta B(\theta))_k} =  \sum_{k=0}^K\mathbb{E} \abs{\sum_{m=0}^K\theta_m(-1)^{k+m}} = \sum_{k=0}^K\mathbb{E} \abs{\sum_{m=0}^K\theta_m} = (K+1)^{3/2}
\end{align}
This brings us to the rough estimate that $(2K+1)^{-1}\sum_k \abs{(\nabla_\theta B(\theta))_k} \sim \sqrt{K}$. 
So in total we find that $\lambda^*_a \sim K^{3.5}$. 
Finally, for $\lambda^*_b$ one can make the estimate that,
\begin{align}
    \lambda^*_b \approx \frac{2\pi \sum_{k=1}^K k^4}{K+1} \sim K^4, 
\end{align}

Hence, it turns out that applying these different strategies leads to different scalings of $\lambda$ with respect to increasing $K$ for the Fourier features model. Despite these differences, it was observed that the resulting condition numbers $\kappa(\sA(\lambda^*))$, $\kappa(\sA(\lambda^*_a))$ and $\kappa(\sA(\lambda^*_b))$ are actually very similar. 
\end{example}

\subsubsection{Hard boundary conditions.} From the very advent of PINNs \citep{Lag1,Lag2}, several authors have advocated modifying machine learning models such that the boundary conditions in PDE (\ref{eq:pde}) can be imposed exactly and the boundary loss in \eqref{eq:epil-td} is zero. Such \emph{hard} imposition of boundary conditions (BCs) has been empirically shown to aid training, e.g. \cite{dong2021method,Mos1,Mos2} and references therein. This can be explained as implementing hard boundary conditions leads to a different matrix $\sA$, through changing the operator $TT^*$. 

We first consider a toy model to demonstrate these phenomena in a straightforward way. We consider again the Poisson equation $-\Delta u = -\sin$ with zero Dirichlet boundary conditions (cf. Example \ref{def:poisson}). We choose as model
\begin{eqnarray}
u_\theta(x) = \theta_{-1}\cos(x)+\theta_0+\theta_{1}\sin(x)
\end{eqnarray}
and report the comparison of the condition number between various variants of soft and hard boundary conditions from \cite{deryck2023operator}.
\begin{itemize}
    \item \emph{Soft boundary conditions.} 
    If one sets the optimal $\lambda$ using \eqref{eq:optimal-lambda} we find that the condition number for the optimal $\lambda^*$ is given by $3+2\sqrt{2}\approx 5.83$. 
    \item \emph{Hard boundary conditions - variant 1.} A first common method to implement hard boundary conditions is to multiply the model with a function $\eta(x)$ so that the product exactly satisfies the boundary conditions, regardless of $u_\theta$. In our setting, we could consider $\eta(x)u_\theta(x)$ with $\eta(\pm \pi) = 0$. For $\eta=\sin$ the total model is given by
    \begin{align}
        \eta(x)u_\theta(x) = -\frac{\theta_{-1}}{2}\cos(2x) + \frac{\theta_{-1}}{2} + \theta_0 \sin(x) + \frac{\theta_1}{2}\sin(2x),
    \end{align}
    and gives rise to a condition number of $4$. Different choices of $\eta$ will inevitably lead to different condition numbers. 
    \item \emph{Hard boundary conditions - variant 2.} Another option would be to subtract $u_\theta(\pi)$ from the model so that the boundary conditions are exactly satisfied. This corresponds to the model, 
    \begin{align}
        u_\theta(x)-u_\theta(\pi) = \theta_{-1}(\cos(x)+1) + \theta_1\sin(x)
    \end{align}
    Note that this implies that one can discard $\theta_0$ as parameter, leaving only two trainable parameters. The corresponding condition number is $1$. 
\end{itemize}
Hence, in this example the the condition number for hard boundary conditions is strictly smaller than for soft boundary conditions i.e., 
\begin{eqnarray}
    \kappa(\sA_{\text{hard BC}}) < \min_\lambda \kappa(\sA_{\text{soft BC}}(\lambda)). 
\end{eqnarray}
This phenomenon can also be observed for other PDEs such as the linear advection equation \cite{deryck2023operator}.

\subsubsection{Second-order optimizers.} There are many empirical studies which demonstrate that first-order optimizers such as (stochastic) gradient descent or ADAM are not suitable for physics-informed machine learning and one needs to use second-order (quasi-)Newton type optimizers such as L-BGFS in order to make training of physics-informed machine learning models feasible. As it turns out,  one can prove that for linear physics-informed models the Hessian of the loss is identical to the matrix $\sA$ (\ref{eq:defn}). Given any loss $\cJ(\theta)$, Newton's method's gradient flow is given by
\begin{equation}
    \frac{d\theta(t)}{dt} = - \gamma H[\cJ(\theta(t))]^{-1} \nabla_\theta \cJ(\theta(t)),
\end{equation}
where $H$ is the Hessian. In our case, we consider the physics-informed loss $L(\theta) = R(\theta)+\lambda B(\theta)$ and consider the model $u_\theta(x) = \sum_\ell \theta_\ell \phi_\ell(x)$, which corresponds to a linear model or a neural network in the NTK regime (e.g. Lemma \ref{lem:NTK-constant}).
Using that $\partial_{\theta_i}\partial_{\theta_j} u_\theta = 0$, we calculate, 
\begin{equation}
\begin{split}
    \partial_{\theta_i}\partial_{\theta_j}L(\theta) &= \int_\Omega  (\cL \partial_{\theta_i} u_\theta(x)) \cdot  \cL \partial_{\theta_j} u_\theta(x) dx  +\lambda \int_{\partial\Omega} \partial_{\theta_i}u_{\theta(t)}(x) \cdot  \partial_{\theta_j} u_\theta(x) dx\\
    &= \int_\Omega  \cL \phi_i(x) \cdot \cL \phi_j(x) dx  + \lambda \int_{\partial\Omega} \phi_i(x) \cdot  \phi_j(x) dx \\
    &= \sA_{ij}. 
\end{split}
\end{equation}
Therefore, in this case (quasi-)Newton methods automatically compute an (approximate) inverse of the Hessian and hence, precondition the matrix $\sA$, relating the use of (quasi-)Newton type optimizers to preconditioning operators in this context. See \cite{Zeinhofer} for further analysis of this connection.

\subsubsection{Domain decomposition.} Domain decomposition (DD) is a widely used technique in numerical analysis to precondition linear systems that arise out of classical methods such as finite elements \citep{DDbook}. Recently, there have been attempts to use DD-inspired methods within physics-informed machine learning, see \cite{Mos1,Mos2} and references therein, although no explicit link with preconditioning the models was established. This link was established in \cite{deryck2023operator} for the linear advection equation, where they computed that the condition number of $\sA$ depends quadratically on the advection speed $\beta$. They argued that considering $N$ identical submodels on subintervals of $[0,T]$ essentially comes down to rescaling $\beta$. Since the condition number scales as $\beta^2$ splitting the time domain in $N$ parts will then reduce the condition number of each individual submodel by a factor of $N^2$. This is also intrinsically connected to what happens in {\bf causal learning} based training of PINNs \citep{wang2022respecting}, which therefore also can be viewed as a strategy for improving the condition number. 
\section{Conclusion}
\label{sec:conc}
In this article, we have provided a detailed review of available theoretical results on the numerical analysis of physics-informed machine learning models, such as PINNs and its variants, for the solution of forward and inverse problems for large classes of PDEs. We formulated physics-informed machine learning in terms of minimizing suitable forms (strong, weak, variational) of the \emph{residual} of the underlying PDE within suitable subspaces of Ansatz spaces of \emph{parametric functions}. PINNs are a special case of this general formulation corresponding to the strong form of the PDE residual and neural networks as Ansatz spaces. The residual is minimized with variants of the (stochastic)-gradient descent algorithm. 

Our analysis revealed that as long as the solutions to the PDE are \emph{stable} in a suitable sense, the overall (total) error, which is the mismatch between the exact PDE solution and the approximation produced (after training) of the physics-informed machine learning model can be estimated in terms of the following constituents i) approximation error which measures the smallness of the PDE residual within the Ansatz space (or hypothesis class) ii) the generalization gap which measures the \emph{quadrature} error and iii) the optimization error which quantifies how well is the optimization problem solved. Detailed analysis of each component of the error is provided, both in general terms as well as exemplified for many prototypical PDEs.  
Although a lot of the analysis is PDE-specific, we can discern many features that are shared across PDEs and models, which summarize below,
\begin{itemize}
\item The (weak) stability of the underlying PDE plays a key role in the analysis and this can be PDE specific. In general, error estimates can depend on how the solution of the underlying PDEs change with respect to perturbations. This is also consistent with the analysis of traditional numerical methods. 
\item Sufficient regularity of the underlying PDE solution is often required in the analysis although this regularity can be lessened by changing the form of the PDE. 
\item The key bottleneck in physics-informed machine learning is due to the inability of the models to train properly. Generically, training PINNs and their variants can be difficult as it depends on the spectral properties of the Hermitian square of the underlying differential operator although suitable preconditioning might ameliorate training issues. 
\end{itemize}
Finally, theory does provide some guidance to practitioners about how and when PINNs and their variants can be used. Succinctly, regularity dictates the form (strong vs. weak) and physics-informed machine learning models are expected to outperform traditional methods on high-dimensional problems (see \citet{mishra2020physics} as an example) or problems with complex geometries, where grid generation is prohibitively expensive or in inverse problems, where data supplements the physics. Physics-informed machine learning will be particularly attractive when coupled with operator learning models such as neural operators as the underlying space is infinite-dimensional. More research on physics-informed operator learning models is highly desirable. 

\clearpage

\appendix 

\section{Sobolev spaces}\label{sec:sobolev}

Let $d\in\mathbb{N}$, $k\in\mathbb{N}_0$, $1\leq p\leq \infty$ and let $\Omega \subseteq \mathbb{R}^d$ be open. For a function $f:\Omega\to\mathbb{R}$ and a (multi-)index $\alpha \in \N^d_0$ we denote by 
\begin{equation}
    D^\alpha f= \frac{\partial^{\abs{\alpha}} f}{\partial x_1^{\alpha_1}\cdots \partial x_d^{\alpha_d}}
\end{equation}
the classical or distributional (i.e. weak) derivative of $f$. 
We denote by $L^p(\Omega)$ the usual Lebesgue space and for we define the Sobolev space $W^{k,p}(\Omega)$ as
\begin{equation}
    W^{k,p}(\Omega) = \{f \in L^p(\Omega): D^\alpha f \in L^p(\Omega) \text{ for all } \alpha\in\mathbb{N}^d_0 \text{ with } \abs{\alpha}\leq k\}. 
\end{equation}
For $p<\infty$, we define the following seminorms on $W^{k,p}(\Omega)$, 
\begin{equation}
    \abs{f}_{W^{m,p}(\Omega)} = \left(\sum_{\abs{\alpha}= m}\norm{D^\alpha f}^p_{L^p(\Omega)}\right)^{1/p} \qquad \text{for } m=0,\ldots, k, 
\end{equation}
and for $p=\infty$ we define
\begin{equation}
    \abs{f}_{W^{m,\infty}(\Omega)} =\max_{\abs{\alpha}= m} \norm{D^\alpha f}_{L^\infty(\Omega)}\qquad \qquad \text{for } m=0,\ldots, k. 
\end{equation}
Based on these seminorms, we can define the following norm for $p<\infty$,
\begin{equation}
    \norm{f}_{W^{k,p}(\Omega)} = \left(\sum_{m=0}^k \abs{f}_{W^{m,p}(\Omega)}^p\right)^{1/p}, 
\end{equation}
and for $p=\infty$ we define the norm
\begin{equation}
    \norm{f}_{W^{k,\infty}(\Omega)} =\max_{0\leq m\leq k}  \abs{f}_{W^{m,\infty}(\Omega)}. 
\end{equation}
The space $W^{k,p}(\Omega)$ equipped with the norm $\norm{\cdot}_{W^{k,p}(\Omega)}$ is a Banach space. 

We denote by $C^k(\Omega)$ the space of functions that are $k$ times continuously differentiable and equip this space with the norm $\norm{f}_{C^k(\Omega)} = \norm{f}_{W^{k,\infty}(\Omega)}$.

\clearpage 
\addcontentsline{toc}{section}{References}
\bibliography{ref}
\label{lastpage}
\end{document}